\documentclass[11pt]{article}

\usepackage[table]{xcolor}
\usepackage{fullpage,graphicx,authblk,ifthen}
\usepackage{amsmath,amssymb,amsfonts}
\usepackage{amsthm,mathrsfs}

\usepackage{newtxmath}

\newtheorem{definition}{Definition}

\newtheorem{corollary}{Corollary}
\newtheorem{remark}{Remark}

\newtheorem{proposition}{Proposition}

\usepackage{algorithm}
\usepackage{algpseudocode}
\usepackage{tikz}
\usepackage{physics}
\usepackage{booktabs}
\usepackage{adjustbox}  
\usepackage{array,amsmath}
\usepackage[font=small]{subcaption} 
\usepackage{hyperref}
\usepackage{multirow} 

\definecolor{darkgreen}{rgb}{0.0, 0.5, 0.0}
\definecolor{lightgray}{rgb}{0.83, 0.83, 0.83}
\definecolor{brandeisblue}{rgb}{0.19, 0.55, 0.91}
\definecolor{darkpastelgreen}{rgb}{0.4, 0.69, 0.2}
\definecolor{amber}{RGB}{218,165,32}
\definecolor{red}{rgb}{1.0, 0.0, 0.0}
\definecolor{mediumorchid}{rgb}{0.73, 0.33, 0.83}
\definecolor{coralpink}{rgb}{0.97, 0.51, 0.47}
\definecolor{cerise}{rgb}{0.87, 0.19, 0.39}
\definecolor{RYB1}{rgb}{0.83, 0.83, 0.83}
\definecolor{RYB2}{rgb}{0.19, 0.55, 0.91}
\definecolor{RYB3}{rgb}{0.4, 0.69, 0.2}
\definecolor{RYB4}{rgb}{1.0, 0.75, 0.0}
\definecolor{cadmiumred}{rgb}{0.89, 0.0, 0.13}
\definecolor{darkred}{RGB}{180,30,30}     
\definecolor{darkgreen}{RGB}{30,150,30}   
\definecolor{darkblue}{RGB}{30,30,180}  
\definecolor{MatBlue}  {RGB}{31,119,180}
\definecolor{MatOrange}{RGB}{255,127,14} 
\definecolor{MatGreen} {RGB}{44,160,44}  

\usepackage{tikz}
\usepackage{pgfplots}
\usepackage{pgfplotstable}
\pgfplotsset{compat=1.17}
\usetikzlibrary{shapes.geometric}
\usetikzlibrary{arrows.meta,calc}   
\pgfplotsset{compat=1.10, 
            label style={font=\footnotesize},
            tick label style={font=\footnotesize}}
\tikzstyle{customer}=[circle,draw=black,fill=gray!40,thick,inner sep=0pt,minimum size=8mm]
\tikzstyle{customerNot}=[circle,draw=white,fill=white,thick,inner sep=0pt,minimum size=8mm]

\tikzstyle{depot}=[draw=black,fill=gray!40,thick, inner sep=0pt,minimum size=6mm]
\tikzstyle{fdepot}=[regular polygon,regular polygon sides=5, draw=black,fill=gray!40,thick, inner sep=0pt,minimum size=6mm]
\tikzstyle{timestep}=[diamond, draw=black,fill=mediumorchid, thick, inner sep=0pt,minimum size=6mm]
\tikzstyle{nodeLegend}=[circle,draw=black,fill=gray!40,thick,inner sep=0pt,minimum size=4.5mm]
\tikzstyle{depotLegend}=[draw=black,fill=gray!40,thick, inner sep=0pt,minimum size=4mm]
\tikzstyle{customerLegend}=[circle,draw=black,fill=amber,thick,inner sep=0pt,minimum size=4mm]
\tikzstyle{fdepotLegend}=[regular polygon,regular polygon sides=5, draw=black,fill=gray!40,thick, inner sep=0pt,minimum size=4mm]
\tikzstyle{timestepLegend}=[diamond, draw=black,fill=mediumorchid, thick, inner sep=0pt,minimum size=4mm]
\tikzstyle{intermediate}=[circle,draw=black,fill=gray!20,thick,inner sep=0pt,minimum size=6mm]
\usetikzlibrary{%
    shapes.misc,    
    positioning,    
    patterns,       
}

\usetikzlibrary{decorations.pathreplacing, shapes.geometric, arrows}

\usetikzlibrary{decorations}
\tikzstyle{imp}=[circle,fill,inner sep=0pt, minimum size=2mm]
\tikzstyle{customer}=[circle,draw=black,fill=amber,thick,inner sep=0pt,minimum size=6mm]
\tikzstyle{depot}=[draw=black,fill=gray!40,thick, inner sep=0pt,minimum size=6mm]

\newcommand{\R}{\mathbb{R}} 
\newcommand{\Z}{\mathbb{Z}} 
\newcommand{\Q}{\mathbb{Q}}

\newcommand{\D}{\mathcal{D}} 
 
\newcommand{\E}{\mathcal{E}}

\newcommand{\N}{\mathcal{N}}

\newcommand{\conv}{\text{conv}\!}
\newcommand{\proj}{\text{proj}}
\newcommand{\LP}{LP}

\newcommand{\DW}{DW}
\newcommand{\AF}{AF}
\newcommand{\xbar}{\bar{\vb x}}
\newcommand{\ybar}{\bar{\vb y}}
\newcommand{\lbar}{\bar{\boldsymbol{\lambda}}}

\newcommand{\journalvspace}[1][1em]{%
}

\makeatletter
\newenvironment{mysubequations}[1]
 {%
  \addtocounter{equation}{-1}%
  \begin{subequations}
  \renewcommand{\theparentequation}{#1-}%
  \def\@currentlabel{#1}%
 }
 {%
  \end{subequations}\ignorespacesafterend
 }
\makeatother

\usepackage{natbib}
 \bibpunct[, ]{(}{)}{,}{a}{}{,}%
 \def\newblock{\ }%
\begin{document}

\title{Dantzig-Wolfe and Arc-Flow Reformulations: A Systematic Comparison}

\author[1]{Daniel Yam\'in\thanks{\texttt{dyamin@andrew.cmu.edu}}}
\author[1]{Willem-Jan van Hoeve\thanks{\texttt{vanhoeve@andrew.cmu.edu}}}
\affil[1]{Tepper School of Business, Carnegie Mellon University, Pittsburgh, PA USA}
\author[2]{Ted K. Ralphs\thanks{\texttt{ted@lehigh.edu}}}
\affil[2]{Department of Industrial and Systems Engineering, Lehigh University,
Bethlehem, PA USA}


\maketitle

\begin{abstract}

Dantzig-Wolfe reformulation is a widely used technique for obtaining stronger relaxations in the context of branch-and-bound methods for solving integer optimization problems. Arc-Flow reformulations are a lesser known technique related to dynamic programming that has experienced a resurgence as result of the recent popularization of decision diagrams as a tool for solving integer programs. Although these two reformulation techniques arose independently, the recently proposed solution paradigm known as \emph{column elimination} has revealed that they are in fact closely connected. Building on a unified formulation and notation, this study clarifies the theoretical connections and computational trade-offs between these two reformulations. 

We first revisit the known fact that the LP relaxations of these two reformulations yield the same dual bound. We then dig deeper, establishing conditions under which valid inequalities in the original, Dantzig-Wolfe, or Arc-Flow spaces can be translated across reformulations without loss of strength, and reinterpreting iterative strengthening methods, such as decremental state-space relaxation and column elimination, through the lens of cutting planes. To assess the potential impact of these insights empirically, we benchmark both reformulations under identical conditions on the vehicle routing problem with time windows using state-of-the-art column- and cut-generation techniques. The results identify clear contrasts: the Arc-Flow reformulation benefits from faster convergence and performs best when subproblems are highly relaxed or low-dimensional, whereas the Dantzig-Wolfe reformulation is more efficient when the master problem remains compact. Overall, our study provides a unified perspective and practical guidelines for choosing between Dantzig-Wolfe and Arc-Flow reformulations in large-scale integer optimization.

\end{abstract}



\section{Introduction}

Many real-world problems, such as vehicle routing, bin packing, and crew scheduling, have an underlying combinatorial structure and can be effectively solved by formulating them as integer programs (IPs). Although these IPs may have compact formulations, these formulations often suffer from significant solution symmetry and weak linear optimization problem (LP) relaxations, making them challenging for general-purpose solvers. 

To overcome these issues, a common strategy is the classical Dantzig-Wolfe (DW) reformulation, which reduces symmetry and strengthens the dual bounds obtained by solving the LP relaxation at the expense of introducing exponentially many variables \citep{vanderbeck2000dantzig}, each of which corresponds to one solution of a chosen relaxation. To make this approach computationally tractable, several techniques have been developed, including column generation, problem-specific relaxations and pricing algorithms, and dynamic cut generation (see, e.g., \cite{desrosiers2024branch}). DW combined with column and cut generation in the context of a branch-and-bound algorithm is one of the most effective methods for solving certain classes of large-scale IPs, despite computational challenges such as convergence issues \citep{pessoa2018automation}. 

A related line of research involves Arc-Flow (AF) formulations, which represent the feasible solutions of a bounded IP as paths from a root to a terminal node in a directed acyclic graph (DAG). In AF formulations derived from dynamic programming, the DAG is the unfolded state-transition graph: nodes represent states and arcs represent transitions (typically extending or modifying a partial solution by changing or fixing the value of a variable) \citep{martin1990polyhedral}. The origins of AF models trace back to \cite{ford1958constructing}, who studied maximal dynamic flows on time-expanded networks; \cite{shapiro1968dynamic}, who modeled the knapsack problem as a pseudopolynomial network flow; and \cite{wolsey1973generalized}, who introduced a general algorithm for solving integer optimization problems by recasting them as longest-route problems.

An AF reformulation can also be derived from a decision diagram, which in the context of optimization is a DAG encoding feasible solutions as paths \citep{bergman2016decision}. In fact, the state-transition graph of a dynamic program can be interpreted as a decision diagram under mild conditions. The main difference between AF formulations based on dynamic programming and those based on decision diagrams lies in how the DAG is constructed and manipulated; see \cite{de2022arc}, \cite{castro2022decision}, and \cite{hoeve2024} for reviews.

In contrast to DW, where the solutions to a given relaxation are explicitly enumerated by solving a subproblem, an AF \emph{reformulation} represents those solutions implicitly as paths in a DAG, preserving the improved dual bound while replacing the exponentially many variables associated with individual elements of the given relaxation with a potentially more compact set of arc variables \citep{de2022arc}. These arcs represent \emph{disaggregations} of the columns in DW, which enable their \emph{recombination} into new feasible solutions during column generation, accelerating convergence \citep{sadykov2013column}. Despite being studied for decades, AF has gained renewed attention due to advances in solvers, connections to decision diagrams, and strong results on combinatorial problems, such as cutting and packing \citep{de2023exact}, machine scheduling \citep{sadykov2013column}, and graph coloring \citep{van2022graph}.

While DW and AF are known to be related, formal results on their connection are limited, particularly in the context of advanced algorithms that combine solution of subproblem relaxations with methods for dynamically generating cutting planes, two core techniques in modern exact methods. Establishing these connections will inform the design of more efficient computational methods. Our study aims to establish such connections via a theoretical and computational analysis. First, we revisit the foundations of DW and AF in a unified framework, highlighting their relationships and trade-offs. Building on these foundations, we extend the analysis to settings with cutting planes and subproblem relaxations. For cutting planes, we establish the conditions under which valid inequalities in the original, DW, or AF spaces can be translated across reformulations, showing that these cuts preserve their strength and analyzing their computational impact. For subproblem relaxations, we reinterpret iterative strengthening methods, such as \emph{decremental state-space relaxation} and \emph{column elimination}, through the lens of cutting planes, clarifying their computational role. 

To empirically test our insights, we benchmark DW and AF under identical conditions on the vehicle routing problem with time windows (VRPTW), employing state-of-the-art techniques. The results reveal contrasts between the approaches across instance and methodological features, e.g., when time windows are tight or loose, or when the relaxation is weak or strong. Overall, our study provides new insights into the relative theoretical and computational strengths of the two reformulations.

The remainder of this paper is organized as follows. Section~\ref{s_reformulations} reviews the theoretical and computational foundations of DW and AF. Section~\ref{s_cuts} analyzes their relationship in the presence of cutting planes, and Section~\ref{s_subproblemRelaxation} examines iterative strengthening methods for subproblem relaxations. Section~\ref{s_VRPTW} illustrates our comparative study on the VRPTW, and Section~\ref{s_computational} reports numerical results. Finally, Section~\ref{s_conclusions} outlines guidelines for reformulation choice and future research directions.

\section{Theoretical and Computational Foundations} \label{s_reformulations}
We consider an integer program of the form \citep{vanderbeck2000dantzig}:

\journalvspace[-1.5cm]
\begin{mysubequations}{IP} \label{eq_IP}
\begin{align}
    z := \max \quad & \vb c\vb x \\
    \text{s.t.} \quad 
    & \vb A \vb x \leq \vb b \label{eq_IP-master}\\
    & \vb D \vb x \leq \vb d \label{eq_IP-sub}\\
    & \vb x \in \Z_{+}^{n},
\end{align}
\end{mysubequations}
\journalvspace[-1.5cm]

\noindent where $\vb A \in \Q_{+}^{m \times n}$ and $\vb D \in \Q_{+}^{l \times n}$ are nonnegative rational matrices and $\vb c \in \Q_{+}^{n}, \vb b \in \Q_{+}^{m}$, and $\vb d \in \Q_{+}^{l}$ are nonnegative rational vectors. The reason for separating equations~\eqref{eq_IP-master} from~\eqref{eq_IP-sub} is to define the \emph{subproblem}, a relaxation with the feasible region $ X:=\{\vb x\in \Z_{+}^{n} \, : \, \vb D \vb x \leq \vb d\}$, which we assume has some special structure that makes solution of it more tractable than the original problem. In many applications, for example, the subproblems can be solved effectively using a dynamic programming reformulation. Both $X$ and the feasible region of~\eqref{eq_IP} are assumed to be bounded and feasible. 

\subsection{Dantzig-Wolfe Reformulation}

Since $X$ is bounded, it contains a finite number of integer points 
leading to the following conceptual redefinition:
\begin{equation}
X = \left\{\vb x \in \R_{+}^{n} \, : \,
   \vb x=\sum_{q \in Q} \vb x^{q} \lambda_{q}, \,
   \sum_{q \in Q} \lambda_{q}= 1,\;
   \boldsymbol{\lambda} \in \{0,1\}^{|Q|}\right\} =\{\vb x^{q}\}_{q \in Q},
\label{eq_DWlink}
\end{equation}
where $Q$ indexes the solutions of $X$. The reformulation of $X$ yields the DW reformulation of~\eqref{eq_IP}:

\journalvspace[-1.5cm]

\begin{mysubequations}{DW} \label{eq_DW}
\begin{align}
    \max \quad & \sum_{q \in Q} (\vb c \vb x^{q}) \lambda_{q} \\
    \text{s.t.} \quad 
    & \sum_{q \in Q} (\vb A \vb x^{q}) \lambda_{q} \leq \vb b \label{eq_DW1}\\
    & \sum_{q \in Q} \lambda_{q} = 1 \label{eq_DW2}\\
    & \boldsymbol{\lambda} \in \{0,1\}^{|Q|}. \label{eq_DW3}
\end{align}
\end{mysubequations}

\journalvspace[-1.0cm]

Although the redefinition of $X$ in~\eqref{eq_DWlink} introduces exponentially many variables, dropping integrality restriction on $\boldsymbol{\lambda}$ yields $\conv(X)$, which means that solving the LP relaxation of~\eqref{eq_DW} yields 
\begin{equation}
    z^{\LP}_{\DW} = \max \left\{\sum_{q \in Q} (\vb c \vb x^{q}) \lambda_{q} \,:\, \sum_{q \in Q} (\vb A \vb x^{q}) \lambda_{q} \leq \vb b,
    \sum_{q \in Q} \lambda_{q} = 1, \boldsymbol{\lambda} \geq \vb 0 \right\},
\end{equation}
which is a potentially better dual bound than that yielded by solving the LP relaxation of~\eqref{eq_IP}. In fact, it is not hard to observe that solving the LP relaxation of~\eqref{eq_DW} yields the so-called \emph{decomposition bound}
\begin{equation}
    z^{\LP}_{\DW} = z_{D} := \max \left\{\vb c \vb x\,:\, \, \vb A \vb x \leq \vb b,\, \vb x \in \conv(X) \right\} \,\,\leq \,\, \max \left\{ \vb c \vb x\,:\, \vb A \vb x \leq \vb b,\, \vb D \vb x \leq \vb d, \vb x \in \R^{n}_{+}\right\} =: z^{LP}
\end{equation}
In a branch-and-price algorithm, the LP relaxation of~\eqref{eq_DW} is solved to obtain an improved dual bound on the optimal solution value. This LP relaxation is referred to as the \emph{Dantzig-Wolfe master problem} (DW--MP) and can be solved by column generation. In each iteration, a subset $Q' \subseteq Q$ of variables is considered, resulting in a restricted DW--MP (DW--RMP). To determine whether the current solution is optimal, we must determine whether any of the columns not included in the current DW--RMP have a positive reduced cost. If not, this is a proof that the solution to the current DW--RMP is also optimal for the DW--MP. The determination of whether the current solution is optimal is done by solving the following DW pricing problem:
\begin{equation}
    \overline{c}^{\star}:=-\pi_{0} + \max_{\vb x^{q} \in X} \big\{(\vb c \vb x^{q}) - \boldsymbol{\pi}\vb (\vb A \vb x^{q}) \big\}, \label{eq_DW--SP}
\end{equation}
\noindent where $\boldsymbol{\pi} \in \R_{+}^{m}$ and $ \pi_{0} \in \R$ are the dual variables associated with constraints~\eqref{eq_DW1} and~\eqref{eq_DW2}. If $\overline{c}^{\star} \leq 0$, no positive reduced cost column exists, and the current DW--RMP solution yields an optimal DW--MP solution by setting to zero all nongenerated variables. Otherwise, at least one positive reduced cost column $[(\vb c\vb x^{q}), \, (\vb A \vb x^{q}), \, 1]$ is added to DW--RMP and the procedure is repeated. 

\subsection{Arc-Flow Reformulation} \label{s_AF}

AF provides an alternative reformulation with similar properties. The difference is that the members of $X$ are represented as paths in a DAG rather than as discrete points in a polyhedron. The reformulation of~\eqref{eq_IP} using AF requires two elements: a DAG $\D$ with root node $r$ and terminal node $t$ in which each path is associated with a unique element of $X$; and a \emph{projection matrix} $\vb T$ that maps a unit flow from $r$ to $t$ in $\D$ (which denotes a single $r$-$t$ path) to a unique element of $X$. Such a DAG and an associated projection matrix always exists, as discussed in further detail below.

Unlike DW, AF seems to lack an obvious unified theoretical framework, since multiple techniques can be used to construct the underlying DAG, which is a necessary ingredient to any formulation. Building on the extended formulations proposed by~\citet{sadykov2013column}, we introduce a generic AF characterization that links the DAG to the subproblem $X$ through the linear projection $\vb T$, explicitly relating AF to~\eqref{eq_IP} and~\eqref{eq_DW}.
\\[-0.2in]
\noindent \begin{definition}\label{def_af}
A DAG $\D=(\N, \E)$, along with a projection matrix $\vb T \in \{0,1\}^{n \times |\E|}$, a root node $r \in \N$, and a terminal node $t \in \N$, defines an \emph{arc-flow representation} of $X$ if:
\begin{enumerate}
    \item[(i)] $\text{conv}(X) = \text{proj}_{\vb x}(F)$ where $F$ is the flow polytope defined over $\D$, namely,
    {\small
    \begin{equation*}
        F := \left\{(\vb x, \vb  y) \in \R_{+}^{n+|\E|} \,:\, \vb x = \vb T \vb y, \, \textstyle \sum_{e \in \delta^{+}(r)} y_{e} = 1, \, \sum_{e \in \delta^{+}(u)} y_{e} - \sum_{e \in \delta^{-}(u)} y_{e} = 0 \;\; \forall u \in \N \setminus\{r,t\}\right\},
    \end{equation*}
    }
    where $\delta^{+}(u):=\{e=(u,v) \in \E\}$ and $\delta^{-}(u):=\{e=(v,u) \in \E\}$; and
    \item[(ii)] $X = \text{proj}_{\vb x}(Y)$, where $Y:=\left\{(\vb x, \vb y) \in F \, : \, \vb y \in \{0,1\}^{|\E|} \right\}$. \hfill$\blacksquare$
\end{enumerate}
\label{d_arcflowRefor}
\end{definition}
\ \\[-0.2in]
One obvious way of obtaining a DAG satisfying this definition is by expressing the problem of optimizing over $X$ as a dynamic program, as described in~\citet[p. 308]{wolsey1999integer}. In this construction (which assumes that $X$ is an independence system), the root node represents the zero solution and traversing an arc corresponds to changing the value of a variable. Whenever the arcs in a given DAG correspond to changing (or setting) the value of a single variable, then a projection matrix $\vb T$ exists and we have an arc-flow representation satisfying Definition~\ref{def_af}.

Just as with the DW reformulation, even though the representation of $X$ as paths in a DAG may increase the number of variables substantially, it nevertheless provides a foundation for deriving exact extended formulations, the solution of whose LP relaxations yield improved dual bounds. 

Given $\D=(\N, \E)$ and its associated projection matrix $\vb T$ satisfying Definition~\ref{def_af}, the associated \textit{AF reformulation} of~\eqref{eq_IP} is:

\journalvspace[-1.0cm]
\begin{mysubequations}{AF}
  \label{eq_AF}%
  \begin{align}
    \max \quad & (\vb c \vb T) \vb y  \\
    \text{s.t.} \quad
    & (\vb A \vb T) \vb y \leq \vb b \label{eq_AF1}\\
    & \sum_{e \in \delta^{+}(r)} y_{e} = 1  \label{eq_AF2} \\ 
    & \sum_{e \in \delta^{+}(u)} y_{e} - \sum_{e \in \delta^{-}(u)} y_{e} = 0  &&  u \in \N \setminus\{r,t\}  \label{eq_AF3} \\ 
    & \vb y \in \{0,1\}^{|\E|}  \label{eq_AF4},
  \end{align}
\end{mysubequations}

\journalvspace[-0.75cm]

\noindent where $\vb y \in \{0,1\}^{|\E|}$ implies $\vb x = \vb T \vb y \in \Z_{+}^{n}$. We denote the \emph{flow constraints}~\eqref{eq_AF2} and~\eqref{eq_AF3} as $\vb F \vb y = \vb f$. As with DW, the LP relaxation of~\eqref{eq_AF} yields an improved dual bound 
\begin{equation}
    z^{LP}_{AF} = \max \left\{(\vb c \vb T) \vb y \,:\, (\vb A \vb T) \vb y \leq \vb b,
    \,\, \vb F \vb y = \vb f,
    \,\, \vb y \geq \vb 0 \right\}
\end{equation}
This bound can be computed by column-and-row (or, arc-and-node) generation \citep{de2023exact}. In each iteration, a subset of arcs $\E' \subseteq \E$ and nodes $\N' \subseteq \N$ with $r,t \in \N'$ is considered in a restricted MP (AF--RMP). Given an AF--RMP solution, the pricing step solves:
\begin{equation}
    \overline{c}^{\star}:=-\rho_{r} + \max_{\vb y^{p} \in Y}\big\{(\vb c \vb T) \vb y^{p} - \boldsymbol{\pi}(\vb A \vb T) \vb y^{p}\big\}, \label{eq_AF--SP}
\end{equation}
\noindent where $P$ indexes the solutions in $Y$, i.e., $Y=\{\vb y^{p}\}_{p \in P}$; $\boldsymbol{\pi} \in \mathbb{R}_{+}^{m}$ are the AF--RMP dual variables associated with constraints~\eqref{eq_AF1}; and $\boldsymbol{\rho} \in \mathbb{R}^{|\mathcal{N}'|}$ are the AF--RMP dual variables associated with constraints~\eqref{eq_AF2} and~\eqref{eq_AF3}, with \( \rho_t := 0 \) for notational conciseness. Due to flow constraints $\vb F \vb y = \vb f$, \eqref{eq_AF--SP} is a longest path problem over $\D$ with modified arc costs $\overline{\vb c} :=(\vb c \vb T) - \boldsymbol{\pi}(\vb A \vb T)$. If $\overline{c}^{\star}\leq 0$, the current AF--RMP solution yields an optimal AF--MP solution by setting to zero all nongenerated variables. Otherwise, there exists at least one path $\vb y^{p} \in Y$ with $\overline{\vb c}\vb y^{p} > \rho_{r}$, and therefore the arcs and nodes used by $\vb y^{p}$ are added to AF--RMP (if not already in $\E'$ and $\N'$). Thus, the flow conservation constraints~\eqref{eq_AF3} in AF--RMP are generated \emph{on demand}, with the constraint for node $u \in \N$ added when $u$ first appears as the head or tail of an AF--RMP arc variable. The exactness of this procedure is proved in \cite{de2023exact}. The discussion above extends naturally to~\eqref{eq_IP} with a block-diagonal structure, as detailed in Appendix~\ref{s_blockDiagonal}.

\subsection{Connections Between Reformulations} \label{ss_connections}

In this section, we revisit the key theoretical relationships between DW and AF. First discussed in \cite{sadykov2013column} as formal claims or (informal) observations within the broader context of extended formulations, we derive tailored proofs that establish a deeper connection between DW and AF not previously made explicit, also enabling the analysis of DW and AF under cutting planes and subproblem relaxations. 
We first introduce notation: the feasible set of the LP relaxation of~\eqref{eq_DW} is defined as
$
\mathrm{DW}_{\LP}:=\big\{(\vb x, \boldsymbol{\lambda}) \in \R_{+}^{n+|Q|} \, : \, \vb x = \sum_{q \in Q} \vb x^{q} \lambda_{q}, \, \sum_{q \in Q} (\vb A \vb x^{q}) \lambda_{q} \leq \vb b, \, \sum_{q \in Q} \lambda_{q} =1\big\},
$
and the feasible set of the LP relaxation of~\eqref{eq_AF} is defined as
$
\mathrm{AF}_{\LP}:=\big\{(\vb x, \vb y) \in \R_{+}^{n+|\E|} \, : \, \vb x = \vb T \vb y, \, (\vb A \vb T) \vb y\leq \vb b, \, \vb F \vb y = \vb f \big\}.
$
Proposition~\ref{p_bound} compares the strength of the two reformulations.

\begin{proposition}[\textbf{Observation 1 in \cite{sadykov2013column}}] \label{p_bound}
\begin{equation*}
    z_{D} \, = \, z^{LP}_{DW} \, = \, z^{LP}_{AF} \, \leq \, z^{LP}.
\end{equation*}
\end{proposition}
\begin{proof}
By definition of $\mathrm{DW}_{\LP}$, we have
$\proj_{\vb x}(\mathrm{DW}_{\LP}) = \left\{ \vb x \in \R^{n}_{+} \, : \, \vb A \vb x \leq \vb b \right\} \, \cap \, \big\{\vb x = \sum_{q \in Q} \vb x^{q} \lambda_{q} \, : \,\sum_{q \in Q}\lambda_{q} = 1,   \, \boldsymbol{\lambda} \geq \vb 0 \big\} = \left\{ \vb x \in \R^{n}_{+} \, : \, \vb A \vb x \leq \vb b \right\} \, \cap \, \conv(X)$. Similarly, by definition of $\mathrm{AF}_{\LP}$ and
Definition~\ref{d_arcflowRefor}(i), $\proj_{\vb x}(\mathrm{AF}_{\LP}) = \left\{ \vb x \in \R^{n}_{+} \, : \, \vb A \vb x \leq \vb b \right\} \, \cap \, \left\{\vb x =\vb T \vb y \, : \, \vb F \vb y = \vb f, \vb y \geq \vb 0\right\} = \left\{ \vb x \in \R^{n}_{+} \, : \, \vb A \vb x \leq \vb b \right\} \, \cap \, \conv(X)$. Both LP relaxations maximize $\vb c \vb x = \vb c \sum_{q \in Q} \vb x^{q} \lambda_{q} = \vb c (\vb T \vb y)$ over the same convex set, hence $z_{\DW}^{\LP}=z_{\AF}^{\LP}=z_{D}$. Finally, since $\conv(X)\subseteq \{\vb x \in \R^{n}_{+} \, : \, \vb D\vb x \leq \vb d\}$, it follows that $z_{D} \leq z^{\LP}$. 
\end{proof}

Proposition~\ref{p_functions} relates the number of columns in DW to the number of paths in AF. Based on this, Remark~\ref{r_symmetry} highlights a symmetry drawback in AF.

\begin{proposition} \label{p_functions}
The number of root-to-terminal paths in the DAG of AF is at least as large as the number of DW columns.
\end{proposition}

\begin{proof}
 By Definition~\ref{d_arcflowRefor}(ii), each $\vb y^{p}\in Y$ projects to exactly one $\vb x^{q}\in X$ via $\vb T\vb y^{p}=\vb x^{q}$. Therefore, we can define the index map $\varphi:P\;\longrightarrow\;Q,
   \;
   \varphi(p):=q
   \;\text{ with }\;
   \vb T\vb y^{p}=\vb x^{q}.
$
Note that $\varphi$ is surjective but generally not injective, because two different paths
$\vb y^{p},\vb y^{p'}$ can project to the same $\vb x^{q}$, i.e., $\varphi(p)=\varphi(p')$. Thus, the inverse image
$
    \phi:= \, \varphi^{-1}\, \, : \, \,Q\,\,\longrightarrow\,\,2^{P}\setminus\{\emptyset\},
   \;
   q\,\,\longmapsto\,\,\{p\in P:\varphi(p)=q\},
$
is a nonempty set-valued map. Henceforth, we fix a selection $\sigma: Q \rightarrow P$ with $\sigma(q) \in \phi(q)$ for all $q \in Q$. The surjection of $\varphi$ implies $|Q|\le|P|$.
\end{proof}

\begin{remark}
     AF has a symmetry drawback relative to DW, stemming from the non-uniqueness of lifted solutions, i.e., $\phi(\cdot)$ is not necessarily a singleton. This may result in redundant paths encoded in the DAG and increased computational overhead.
    \label{r_symmetry}
\end{remark}

Remark~\ref{r_lifting} analyzes the computational complexity of lifting solutions from the original to the AF space. Building on this, Proposition~\ref{p_solutions} establishes a constructive correspondence between DW and AF solutions, while Remark~\ref{r_subproblemEquivalence} relates their pricing problems.

\begin{remark}[\textbf{Observation 4 in \cite{sadykov2013column}}]  \label{r_lifting}
    Given $\vb x^{q} \in X$, a lifting procedure outputs $\vb y^{p} \in Y$ such that $\vb x^{q} = \vb T \vb y^{p}$. One generic lifting procedure is solving $
    \max \left\{ 0 \, : \, \vb T \vb y = \vb x^{q}, \,  \vb F \vb y = \vb f, \, \vb y \in \{0,1\}^{|\E|} \right\}
    $,
    which is NP-hard. Dedicated lifting procedures can often recover the corresponding path through a linear scan of the DAG. For example, \cite{martin1990polyhedral} propose a greedy lifting procedure for AF models based on dynamic programming that runs in polynomial time in the size of $\D$.
\end{remark}
\begin{proposition}  \label{p_solutions}
  For every $(\bar{\vb x},\bar{\boldsymbol\lambda})\in \mathrm{DW}_{\LP}$, there exists $\bar{\vb y}$ with $(\bar{\vb x},\bar{\vb y})\in\mathrm{AF}_{\LP}$. Similarly, for every $(\bar{\vb x},\bar{\vb y})\in \mathrm{AF}_{\LP}$, there exists $\bar{\boldsymbol\lambda}$ with $(\bar{\vb x},\bar{\boldsymbol\lambda})\in \mathrm{DW}_{\LP}$.
\end{proposition}
\begin{proof}
    Let $(\xbar,\lbar)\in \mathrm{DW}_{\LP}$. Define $ \ybar :=  \sum_{q \in Q}\vb y^{\sigma(q)}\,\bar{\lambda}_{q}$ with $\vb 1^{\top}\lbar = 1$ and $ \lbar\ge \vb 0$. The existence of $\vb y^{\sigma(q)}$ for every $q\in Q$ is guaranteed by Definition~\ref{d_arcflowRefor}(ii). It follows that $
       \vb T\ybar= \vb T \left( \sum_{q \in Q}\vb y^{\sigma(q)} \bar{\lambda}_{q} \right)
                  = \sum_{q \in Q} \big(\vb T\vb y^{\sigma(q)}\big) \bar{\lambda}_{q}
                  = \sum_{q \in Q}\vb x^{q} \bar{\lambda}_{q}
                  = \xbar, $ $
       \vb F \ybar= \vb F \left( \sum_{q \in Q}\vb y^{\sigma(q)} \bar{\lambda}_{q} \right)
                  = \sum_{q \in Q} \big(\vb F\vb y^{\sigma(q)}\big) \bar{\lambda}_{q}
                  = \sum_{q \in Q}\vb f \bar{\lambda}_{q}
                  = \vb f, \quad $, and $
       (\vb A\vb T)\ybar= \vb A \xbar
                  = \sum_{q \in Q} (\vb A \vb x^{q}) \bar{\lambda}_{q}
                  \;\leq\; \vb b.
    $
    Thus, $(\xbar,\ybar)\in\mathrm{AF}_{\LP}$. For the converse, let $(\xbar,\ybar)\in \mathrm{AF}_{\LP}$. Since $\D$ is acyclic, the flow decomposition theorem \cite[p. 33]{ahuja1993network} guarantees that nonnegative arc flow $\ybar$ can be decomposed into a (not necessarily unique) path flow, that is, $\ybar=\sum_{p \in P}\vb y^{p} \theta_{p}$ with $\vb 1^{\top} \boldsymbol{\theta} = 1$ and $\boldsymbol{\theta} \geq \vb 0$,
    where $\theta_{p}$ is the flow through path $\vb y^{p}$. For every $q \in Q$, set $\bar{\lambda}_{q}:=\sum_{p \in \phi(q)} \theta_{p}$. Consequently, $\mathbf 1^{\top}\lbar=1$ and $\lbar \geq \vb 0$. Since
    $
      \bar{\vb x}
        = \vb T\bar{\vb y}
         = \vb T\Bigl(\sum_{p\in P}\vb y^{p}\theta_{p}\Bigr)
         = \sum_{p\in P}\big(\vb T\vb y^{p}\big)\theta_{p}
         = \sum_{p\in P}\vb x^{\varphi(p)}\theta_{p}
         = \sum_{q\in Q}\vb x^{q}\,\bar{\lambda}_{q} $ and $ 
      \sum_{q\in Q}\bigl(\vb A\vb x^{q}\bigr)\bar{\lambda}_{q}
        = \vb A\bar{\vb x}
         = (\vb A\vb T)\bar{\vb y}
         \leq \vb b
    $, then $(\xbar,\lbar)\in\mathrm{DW}_{\LP}$. 
\end{proof}

\begin{remark} [\textbf{Remark 3 in \cite{sadykov2013column}}] 
For any $\boldsymbol{\pi} \in \R^{m}_{+}$, we have that $\max_{\vb x^{q} \in X} \left\{(\vb c \vb x^{q}) - \boldsymbol{\pi}\vb (\vb A \vb x^{q}) \right\} = \max_{\vb y^{p} \in Y}\{(\vb c \vb T) \vb y^{p} - \boldsymbol{\pi}(\vb A \vb T) \vb y^{p}\}$. As a result, in AF, one can solve \eqref{eq_DW--SP} and \emph{lift} the solution, whereas in DW, one can solve \eqref{eq_AF--SP} and \emph{project} the solution.
\label{r_subproblemEquivalence}
\end{remark}

\subsection{Computational Implications in Practice} \label{s_computationalImplications}

In this section, we discuss the computational implications of these connections when solving the reformulations. From the subproblem perspective, pricing is often performed in the AF space even under a DW reformulation, enabling the use of problem-specific techniques unavailable in the original space. Pricing in the AF space may rely on either an implicit or explicit DAG: DW-based methods typically use implicit representations via labeling algorithms (see, e.g., \cite{desrosiers2024branch}), while AF approaches employ explicit DAGs and shortest-path algorithms (see, e.g., \cite{karahalios2025column}). The implicit approach supports larger DAGs, and the explicit one enables fine-grained manipulation.

From the master problem perspective, DW and AF differ in how they represent subproblem solutions: DW uses compact (aggregated) columns, while AF employs arc-based (disaggregated) ones. Aggregated columns yield smaller RMPs and faster intermediate LP solves, whereas disaggregated columns introduce multiple variables and constraints per path, enlarging the RMP and increasing LP solution time. Consequently, the compact DW representation is more advantageous as the AF space grows. Conversely, disaggregated columns in AF enable path \emph{recombination} into new subproblem solutions, accelerating convergence \citep{sadykov2013column}. Overall, AF is expected to perform better when recombination is sufficiently effective to offset its higher per-iteration cost; otherwise, DW is likely to perform better due to its lower iteration cost.

\section{Cutting Planes} \label{s_cuts}

Since integrality constraints are imposed on the $\vb x$-, $\boldsymbol{\lambda}$-, and $\vb y$-variables in~\eqref{eq_IP},~\eqref{eq_DW}, and~\eqref{eq_AF}, valid inequalities can be derived in any of these spaces. In this section, we examine the relationships between DW and AF when introducing three classes of cuts, namely those derived in the original space (\S\ref{s_cutsX}), in the DW space (\S\ref{s_cutsLambda}), and in the AF space (\S\ref{s_cutsY}). Cuts in the original and DW spaces have been extensively studied \citep{ralphs2006decomposition, desaulniers2011cutting}. Despite recent progress by \cite{karahalios2025column}, no work has systematically addressed the translation of such cuts into the AF space. 

\subsection{Cuts in the Original Space} \label{s_cutsX}

A pair $(\boldsymbol{\alpha}, \alpha_{0}) \in \R^{n+1}$ is a valid inequality for~\eqref{eq_IP} if $\boldsymbol{\alpha} \vb x \leq \alpha_{0}$ holds for all $\vb x \in \Z_{+}^{n}$ satisfying \eqref{eq_IP-master}--\eqref{eq_IP-sub}. In~\eqref{eq_IP}, the inequality may be appended to the constraints $\vb A \vb x \leq \vb b$ or to the constraints $\vb D \vb x \leq \vb d$. In the former case, by \eqref{eq_DWlink}, the inequality $\sum_{q \in Q} (\boldsymbol{\alpha} \vb x^{q}) \lambda_{q} \leq \alpha_{0}$ is added to~\eqref{eq_DW}. In DW--RMP, an additional dual variable $\mu \geq 0$ is introduced and the pricing objective~\eqref{eq_DW--SP} subtracts $\mu(\boldsymbol{\alpha}\vb x^{q})$. Similarly, by Definition \ref{d_arcflowRefor}, the constraint $(\boldsymbol{\alpha} \vb T)  \vb y \leq \alpha_{0}$ is added to~\eqref{eq_AF}, and the pricing objective~\eqref{eq_AF--SP} subtracts $\mu(\boldsymbol{\alpha} \vb T) \vb y^{p}$. Note that the cut $(\boldsymbol{\alpha}, \alpha_{0})$ has the same computational impact in DW and AF, adding one constraint to the MP. With such cuts, both reformulations yield the same bound.

 \begin{proposition}  \label{p_projectionStrength}
 Let $\mathrm{DW}_{\LP}^{\text{cut}}:=\mathrm{DW}_{\LP} \, \cap \, \big\{\boldsymbol{\lambda} \geq \vb 0 \, : \sum_{q \in Q} (\boldsymbol{\alpha} \vb x^{q}) \lambda_{q} \leq \alpha_{0}\big\}$ and $\mathrm{AF}_{\LP}^{\text{cut}}:=\mathrm{AF}_{\LP} \, \cap \, \big\{\vb y \geq \vb 0 \, : (\boldsymbol{\alpha} \vb T)  \vb y \leq \alpha_{0} \big\}$. Then, $\proj_{\vb x}\big(\mathrm{DW}_{\LP}^{\text{cut}}\big)=\proj_{\vb x}\big(\mathrm{AF}_{\LP}^{\text{cut}}\big)$. 
 \end{proposition} \begin{proof}
      Using~\eqref{eq_DWlink} and Definition~\ref{d_arcflowRefor}(i),
        \begin{equation*}
        \begin{aligned}
          \proj_{\vb x}\big(\mathrm{DW}_{\LP}^{\text{cut}}\big)
          &= \Big\{\vb x \in \R^{n}_{+} : \vb A \vb x \leq \vb b,\; \boldsymbol{\alpha}\vb x \leq \alpha_{0} \Big\}
             \cap
             \Big\{ \vb x = \textstyle\sum_{q \in Q} \vb x^{q}\lambda_{q} : \vb 1^{\top}\boldsymbol{\lambda} = 1,\; \boldsymbol{\lambda} \ge 0 \Big\} \\
          &= \Big\{\vb x \in \R^{n}_{+} : \vb A \vb x \leq \vb b,\; \boldsymbol{\alpha}\vb x \leq \alpha_{0} \Big\}
             \cap \conv(X) \\
          &= \Big\{\vb x \in \R^{n}_{+} : \vb A \vb x \leq \vb b,\; \boldsymbol{\alpha}\vb x \leq \alpha_{0} \Big\}
             \cap
             \Big\{ \vb x = \vb T \vb y : \vb F \vb y = \vb f,\; \vb y \ge \vb 0 \Big\} \\
          &= \proj_{\vb x}\big(\mathrm{AF}_{\LP}^{\text{cut}}\big).
        \end{aligned} 
        \end{equation*}

        \journalvspace[-0.75cm]
    
\end{proof}

A second case arises when a valid inequality $(\boldsymbol{\delta}, \delta_{0}) \in \R^{n+1}$ for~\eqref{eq_IP} is appended to $\vb D \vb x \leq \vb d$. In this case, the subproblem is modified to $\widehat{X} := \{\vb x \in X \, :\, \boldsymbol{\delta} \vb x \leq \delta_{0}\}$. Consequently, the DW pricing problem~\eqref{eq_DW--SP} is now defined over $\widehat{X}$, but its objective remains unchanged. In DW--MP, the set $Q$ is redefined as $\widehat{Q}$, where $\widehat{X} = \{\vb x^{q}\}_{q \in \widehat{Q}}$. In dynamic cut generation, if a variable $q \in Q \setminus \widehat{Q}$ has been added to DW--RMP, it must be removed from $Q'$.

In AF, a new DAG $\widehat{\D} = (\widehat{\N}, \widehat{\E})$, satisfying Definition~\ref{d_arcflowRefor}(i)-(ii) for $\widehat{X}$ must be constructed. Consequently, the AF pricing problem~\eqref{eq_AF--SP} is defined over $\widehat{Y}$, where $\widehat{X}=\proj_{\vb x}(\widehat{Y})$. In AF--MP, the arc set $\E$ and node set $\N$ are redefined as $\widehat{\E}$ and $\widehat{\N}$, so generated arcs $e \in \E \setminus \widehat{\E}$ and nodes $u \in \N \setminus \widehat{\N}$ for AF--RMP must be removed from $\E'$ and $\N'$. 

\subsection{Cuts in the Dantzig-Wolfe Space}
\label{s_cutsLambda}

A pair $(\boldsymbol{\gamma}, \gamma_{0}) \in \R^{|Q|+1}$ is a valid inequality for~\eqref{eq_DW} if $\boldsymbol{\gamma}\boldsymbol{\lambda} \leq \gamma_{0}$ holds for all $\boldsymbol{\lambda}\in \{0,1\}^{|Q|}$ satisfying \eqref{eq_DW1}--\eqref{eq_DW2}. As previously discussed, every valid inequality $(\boldsymbol{\alpha}, \alpha_{0})\in \R^{n+1}$ in the original space induces a valid inequality in the DW space via $\gamma_{q}=\boldsymbol{\alpha}\vb x^{q}$ for all $q\in Q$ and $\gamma_{0}=\alpha_{0}$. In some cases, however, valid inequalities can also be generated directly in the extended DW space. Such cuts typically require imposing additional constraints in the subproblem that cannot be readily expressed as linear inequalities in the original space. Nevertheless, they have proven effective in certain problem classes \citep{desaulniers2011cutting}. We discuss these cuts in further detail below.

In a similar spirit to cut-generating functions \citep{balas1971intersection}, the cut coefficients are defined as $\gamma_{q}=g(\vb A \vb x^{q})$ for $q \in Q$, where $g$ is a nonlinear function, e.g., the floor operator in Chvátal-Gomory (CG) rank-1 cuts \citep{chvatal1973edmonds}. If $g$ admits a linearization in the original space by introducing additional variables and constraints, then there exists an augmented \eqref{eq_IP} formulation \citep{villeneuve2005compact}:
$
    \max\left\{\vb c \vb x \, : \, \vb A \vb x \leq \vb b, \, \gamma \leq \gamma_{0}, \, \vb D \vb x \leq \vb d, \, \gamma =g(\vb A \vb x), (\vb x, \gamma) \in \Z_{+}^{n+1}\right\}
$. In this lifted space, the variable $\gamma \in \Z_{+}$ captures the value of the nonlinear function $g(\vb A \vb x)$, and the cut is enforced linearly as $\gamma \leq \gamma_{0}$. 

Performing a DW reformulation while keeping $\gamma \leq \gamma_{0}$ at the master level and $\gamma = g(\vb A \vb x)$ at the subproblem level results in
$
\widehat{X} := \{(\vb x, \gamma) \in \Z^{n+1}_{+} \; : \; \vb D \vb x \leq \vb d,\ \gamma = g(\vb A \vb x)\}
$. Reformulating $\widehat{X}=\{(\vb x^{q}, \gamma_{q})\}_{q \in \widehat{Q}}$ as in \eqref{eq_DWlink} gives $\vb x = \sum_{q \in \widehat{Q}} \vb x^{q} \lambda_{q}$ and $\gamma = \sum_{q \in \widehat{Q}} \gamma_{q} \lambda_{q}$. Thus, DW--MP is exactly \eqref{eq_DW} defined over $\widehat{Q}$, with the additional inequality $\sum_{q \in \widehat{Q}} \gamma_{q} \lambda_{q} \leq \gamma_{0}$. The DW pricing problem becomes $
    \overline{c}^{\star} := -\pi_{0} + \max_{(\vb x^{q}, \gamma_{q}) \in \widehat{X}} \big\{ (\vb c \vb x^{q}) - \boldsymbol{\pi} (\vb A \vb x^{q}) - \mu \gamma_{q} \big\}
$. Introducing the variable $\gamma$ and the nonlinear function $g$ increases the complexity of the subproblem. 

To translate the cut $(\boldsymbol{\gamma}, \gamma_{0})$ into~\eqref{eq_AF}, note that $\widehat{X}$ is a finite integer set and therefore a DAG $\widehat{\D}=(\widehat{\N}, \widehat{\E})$ satisfying Definition~\ref{d_arcflowRefor}(i)-(ii) for $\widehat{X}$ exists. Associated with $\widehat{\D}$, we have a projection matrix 
$
\widehat{\vb T} :=
[
    \widehat{\vb T}_{\vb x} \, ; \,
    \boldsymbol{\beta}
]
\in \{0,1\}^{(n+1)\times|\widehat{\E}|}$
such that
$\vb x=\widehat{\vb T}_{\vb x} \vb y$ and $\gamma =\boldsymbol{\beta} \vb y$. As a result, in AF--MP, we have the following constraints: $(\vb A \widehat{\vb T}_{\vb x}) \vb y \leq \vb b$ and $ \boldsymbol{\beta} \vb y \leq \gamma_{0}$, the latter being the new cut. Accordingly, the AF pricing problem becomes $
    \overline{c}^{\star}:=-\rho_{r} + \max_{\vb y^{p} \in \widehat{Y}}\big\{(\vb c \widehat{\vb T}_{\vb x}) \vb y^{p} - \boldsymbol{\pi}(\vb A \widehat{\vb T}_{\vb x}) \vb y^{p} - \mu ( \boldsymbol{\beta} \vb y^{p})\big\},
$
 where $\{\vb y^{p}\}_{p \in \widehat{P}}$ is the enumerated set of solutions of $\widehat{Y}$. Computationally, this cut can affect AF more severely than DW at the master level, as encoding $\gamma = g(\vb A \vb x)$ may require a substantial number of additional arcs and nodes in the DAG, enlarging AF--MP and reducing recombination. Still, both reformulations yield the same bound.

 \begin{proposition} \label{p_projectionStrengthLambda}
    Let $\mathrm{DW}_{\LP}^{\text{cut}}:=\mathrm{DW}_{\LP} \, \cap \, \big\{\boldsymbol{\lambda} \geq \vb 0 \, : \sum_{q \in \widehat{Q}} \gamma_{q} \lambda_{q} \leq \gamma_{0}\big\}$ and $\mathrm{AF}_{\LP}^{\text{cut}}:=\mathrm{AF}_{\LP} \, \cap \, \big\{\vb y \geq \vb 0 \, : \boldsymbol{\beta} \vb y \leq \gamma_{0} \big\}$. Then, $\proj_{(\vb x, \gamma)}\big(\mathrm{DW}_{\LP}^{\text{cut}}\big)=\proj_{(\vb x, \gamma)}\big(\mathrm{AF}_{\LP}^{\text{cut}}\big)$.
 \end{proposition} \begin{proof}
    From the previous discussion,
    {\small
    \begin{equation*}
    \begin{aligned}
   \proj_{(\vb x, \gamma)}\big(\mathrm{DW}_{\LP}^{\text{cut}}\big)
    &=
       \Big\{(\vb x,\gamma)\in\mathbb R^{n+1}_{+} \, : \, \vb A\vb x \leq \vb b,\, \gamma \leq \gamma_{0}\Big\} \, \cap \, \Big \{ (\vb x,\gamma)=\textstyle\sum_{q\in\widehat{Q}} (\vb x^{q},\gamma_{q})\lambda_{q} \, : \,  \vb 1^{\top}\boldsymbol{\lambda}=1,\; \boldsymbol{\lambda} \geq \vb 0\Big\}  \\
    &= \Big\{(\vb x,\gamma)\in\mathbb R^{n+1}_{+} \, : \, \vb A\vb x\leq\vb b,\, \gamma \leq \gamma_{0}\Big\} \, \cap \, \conv(\widehat{X}) \\
    &= \Big\{(\vb x,\gamma)\in\mathbb R^{n+1}_{+} \, : \, \vb A\vb x\leq\vb b,\, \gamma \leq \gamma_{0}\Big\} \, \cap \, \Big\{(\vb x,\gamma) = \big[\widehat{\vb T}_{\vb x}\; \boldsymbol{\beta} \big]^{\!\top}\vb y \, : \,  \widehat{\vb F} \vb y= \widehat{\vb f}, \,  \vb y \geq \vb 0\Big\} \\
    &= \proj_{(\vb x, \gamma)}\big(\mathrm{AF}_{\LP}^{\text{cut}}\big). 
    \end{aligned}
    \end{equation*}
    }
    \journalvspace[-0.50cm]
\end{proof}

\subsection{Cuts in the Arc-Flow Space} \label{s_cutsY}

A pair $(\boldsymbol{\beta}, \beta_{0}) \in \R^{|\E|+1}$ is a valid inequality for~\eqref{eq_AF} if $\boldsymbol{\beta} \vb y \leq \beta_{0}$ holds for all $\vb y\in \{0,1\}^{|\E|}$ satisfying \eqref{eq_AF1}--\eqref{eq_AF3}. The cut is linear in the original space if there exists $\boldsymbol{\alpha} \in \R^{n}$ such that $\boldsymbol{\beta} = \boldsymbol{\alpha} \vb T$, i.e., $\boldsymbol{\beta}$ lies in the row space of $\vb T$. In that case, $\boldsymbol{\beta} \vb y = \boldsymbol{\alpha} \vb x \leq \beta_{0}$, recovering the case in~\S\ref{s_cutsX}. Otherwise, it cannot be represented as a single linear constraint in~\eqref{eq_IP}. Just as in the DW case, valid inequalities may also be generated directly in the extended AF space. 

To translate the cut $(\boldsymbol{\beta}, \beta_{0})$ into \eqref{eq_DW}, one might write $\sum_{q \in Q} \gamma_{q} \lambda_{q} \leq \beta_{0}$ with $\gamma_{q} := \boldsymbol{\beta}\vb y^{\phi(q)}$. This, however, raises two issues. First, if $\phi(q)$ contains multiple paths mapping to $\vb x^{q}$ and $\boldsymbol{\beta}\vb y^{p} \ne \boldsymbol{\beta}\vb y^{p'}$ for some $p, p' \in \phi(q)$, it is unclear which coefficient to assign to $\lambda_{q}$ (see Remark~\ref{r_symmetry}). Without these symmetry issues, the cut remains valid in DW and preserves its strength.

\begin{proposition} \label{p_cutsAF}
    Let $(\boldsymbol{\beta}, \beta_{0}) \in \R^{|\E|+1}$ be a valid inequality for~\eqref{eq_AF}. If $|Q|=|P|$, then $(\boldsymbol{\gamma}, \beta_{0})\in \R^{|Q|+1}$ with $\gamma_{q} := \boldsymbol{\beta}\vb y^{\phi(q)}$ is a valid inequality for~\eqref{eq_DW}. Moreover, $\proj_{\vb x}\big(\mathrm{DW}_{\LP}^{\text{cut}}\big)=\proj_{\vb x}\big(\mathrm{AF}_{\LP}^{\text{cut}}\big)$, where $\mathrm{DW}_{\LP}^{\text{cut}}:=\mathrm{DW}_{\LP} \, \cap \, \big\{\boldsymbol{\lambda} \geq \vb 0 \, : \sum_{q \in Q} \gamma_{q} \lambda_{q} \leq \beta_{0}\big\}$ and $\mathrm{AF}_{\LP}^{\text{cut}}:=\mathrm{AF}_{\LP} \, \cap \, \big\{\vb y \geq \vb 0 \, : \boldsymbol{\beta} \vb y \leq \beta_{0} \big\}$.
\end{proposition}

\begin{proof}
     Take any $(\vb x, \boldsymbol{\lambda})$ feasible for~\eqref{eq_DW}. By the assumption, there is a unique $\bar{q} \in Q$ with $\lambda_{\bar{q}} = 1$ and $\lambda_{q} = 0$ for $q \in Q \setminus\{\bar{q}\}$. Because $\varphi$ is bijective, $\vb x^{\bar{q}} = \vb T \vb y^{\phi(\bar{q})}$ for a single path $\vb y^{\phi(\bar{q})} \in Y$ and $(\vb A \vb T) \vb y^{\phi(\bar{q})} = \vb A \vb x^{\bar{q}} \leq \vb b$. Thus, $\vb y^{\phi(\bar{q})}$ is feasible for~\eqref{eq_AF} and $\sum_{q \in Q}\gamma_{q} \lambda_{q}  = \sum_{q \in Q}(\boldsymbol{\beta} \vb y^{\phi(q)}) \lambda_{q} = \boldsymbol{\beta}\vb y^{\phi(\bar{q})} \leq \beta_{0}$, so the cut is valid for all feasible $\boldsymbol{\lambda}$ in~\eqref{eq_DW}.  

    Next, we show that the cut preserves its strength. Take any $\xbar \in \proj_{\vb x}(\mathrm{DW}_{\LP}^{\text{cut}})$, so there exists $\lbar \geq \vb 0$ with $(\xbar, \lbar) \in\mathrm{DW}_{\LP}^{\text{cut}}$. Define $\ybar:= \sum_{q \in Q}\vb y^{\phi(q)} \bar{\lambda}_{q}$. Since $\vb 1^{\top} \boldsymbol{\bar{\lambda}} = 1$, we have $\xbar = \vb T \ybar, \, \vb F \ybar = \vb f$, and $(\vb A \vb T) \ybar = \vb A \vb \xbar \leq \vb b$. Thus, $(\xbar, \ybar) \in \mathrm{AF}_{\LP}$. Moreover, $\boldsymbol{\beta} \ybar =  \boldsymbol{\beta} \big(\sum_{q \in Q} \vb y^{\phi(q)} \bar{\lambda}_{q}\big) =  \sum_{q \in Q} (\boldsymbol{\beta} \vb y^{\phi(q)}) \bar{\lambda}_{q} = \sum_{q \in Q} \gamma_{q} \bar{\lambda}_{q} \leq \beta_{0}$. Hence, $(\xbar, \ybar) \in \mathrm{AF}_{\LP}^{\text{cut}}$ and $\xbar \in \proj_{\vb x}(\mathrm{AF}_{\LP}^{\text{cut}})$. Conversely, take any $\xbar\in\proj_{\vb x}(\mathrm{AF}_{\LP}^{\text{cut}})$, so there exists $\ybar \geq \vb 0$ with $(\xbar,\ybar) \in\mathrm{AF}_{\LP}^{\text{cut}}$. By the flow decomposition theorem, $\ybar=\sum_{p\in P}\vb y^{p}\theta_{p}$ for some  $\mathbf 1^{\top}\boldsymbol{\theta}=1$ and $\boldsymbol{\theta} \geq \vb 0$. For every $q \in Q$, define $\bar\lambda_{q}:= \theta_{\phi(q)}$, which is possible because $\varphi$ is bijective. Then, $\mathbf 1^{\top}\lbar=1$ and $\lbar\ge\mathbf 0$. Moreover,  $\xbar =\vb T\ybar = \vb T \big( \sum_{p \in P}\vb y^{p}\theta_{p} \big) = \sum_{p \in P} (\vb T \vb y^{p}) \theta_{p} =\sum_{p \in P}\vb x^{\varphi(p)}\theta_{p} =\sum_{q\in Q}\vb x^{q} \bar\lambda_{q}$ and $\sum_{q\in Q}(\vb A\vb x^{q})\bar\lambda_{q} =\vb A\xbar \leq \vb b$. Thus, $(\xbar,\lbar)\in\mathrm{DW}_{\LP}$. Finally, $\sum_{q\in Q}\gamma_{q}\bar\lambda_{q} = \sum_{q\in Q}\bigl(\boldsymbol{\beta}\vb y^{\phi(q)}\bigr)\bar\lambda_{q} =\sum_{p\in P}\boldsymbol{\beta} (\vb y^{p}\theta_{p}) =\boldsymbol{\beta}\ybar \leq  \beta_{0}$, because $(\xbar,\ybar)\in\mathrm{AF}_{\LP}^{\text{cut}}$. Altogether, $(\xbar,\lbar)\in\mathrm{DW}_{\LP}^{\text{cut}}$. 
\end{proof}

The result also holds under milder conditions, namely when $|Q| \le |P|$ and $\boldsymbol{\beta}\vb y^{p} = \boldsymbol{\beta}\vb y^{p'}$ for all $p, p' \in \phi(q)$ and every $q \in Q$. Otherwise, the cut cannot be represented in DW by a single linear inequality, since AF can distinguish between $p$ and $p'$, whereas DW cannot, as both $\vb y^{p}$ and $\vb y^{p'}$ project onto the same point $\vb x^{q} \in X$. The second issue occurs in the DW pricing problem. The term $\mu\gamma_{q}$ must be subtracted from the objective, where $\gamma_{q} = \boldsymbol{\beta}\vb y^{\phi(q)}$. As noted earlier, $\boldsymbol{\beta}\vb y^{\phi(q)}$ need not be linear in the original space, in which case it must be enforced by modifying the subproblem structure as in \S\ref{s_cutsLambda}. In practice, these two issues can be resolved by defining the DW reformulation over $Y$ rather than $X$. In this case, $\vb x = \sum_{p \in P} \vb x^{p}\lambda_{p}$ with $\sum_{p \in P} \lambda_{p} = 1$, $\boldsymbol{\lambda} \ge 0$, and $\vb x^{p} = \vb T \vb y^{p}$ for all $p \in P$. This reformulation performs pricing in the AF space while preserving the aggregated DW column representation at the master level.

\section{Iterative Subproblem Strengthening} \label{s_subproblemRelaxation}

The reformulation of the original problem is computationally motivated by the special structure of the subproblem $X$, which is assumed to make its solution more tractable than the original problem. In some cases, however, even this structured subproblem may still not be sufficiently tractable for practical instances. Like in a cutting-plane algorithm, the subproblem can then be relaxed and iteratively strengthened as needed. This idea underlies well-established techniques such as decremental state-space relaxation \citep{boland2006accelerated}, dynamic discretization discovery \citep{boland2017continuous}, and column elimination \citep{van2022graph}, which we formalize and reinterpret through the lens of cutting planes.  

Recall the definition of the subproblem in the original space as 
$X:=\{\vb x\in \Z_{+}^{n} \, : \, \vb D \vb x \leq \vb d\}$. We define a {\em subproblem relaxation} as any superset $\widehat{X} \supseteq X$.  We denote the index set of the solutions of $\widehat{X}$ by $\widehat{Q}$, which parametrizes the DW reformulation~\eqref{eq_DW} over this relaxed subproblem. Likewise, we use Definition~\ref{d_arcflowRefor} to define an associated AF reformulation~\eqref{eq_AF} over a DAG $\widehat{\D}$ with projection matrix $\widehat{\vb T}$ (both exist by definition). Proposition~\ref{p_bound} yields the following corollary. 

\begin{corollary}
Under the same subproblem relaxation, DW and AF yield the same LP bound.
\end{corollary}

Given an initial relaxation $\widehat{X}^{(0)} \supseteq X$, an iterative subproblem strengthening method proceeds as follows. At iteration $t$, we solve~\eqref{eq_DW} or~\eqref{eq_AF} with the current relaxation $\widehat{X}^{(t)}$, obtaining the bound $\max\{\vb c \vb x\,:\, \vb A \vb x \leq \vb b, \vb x \in \conv(\widehat{X}^{(t)})\} \geq z_{D}$. One can verify whether the solution lies in $\conv(X)$ by checking whether every element in the support of the decomposition (i.e., columns in DW and paths in AF) corresponds to a point in $X$. If so, the decomposition bound $z_{D}$ has been attained. Otherwise, as described in Definition~\ref{d_strenghtening}, as stronger relaxation $\widehat{X}^{(t+1)}$ is obtained from $\widehat{X}^{(t)}$ by removing infeasible points with respect to $X$, so that $\widehat{X}^{(t)} \supset \widehat{X}^{(t+1)} \supseteq X$. Repeating this process results in an iterative strengthening method, formalized in Definition~\ref{d_iterative}. In practice, the procedure may be terminated early (e.g., after a fixed number of iterations or upon reaching a target bound), and the pricing solution may guide the strengthening step.

\begin{definition} \label{d_strenghtening}
    Given a subproblem relaxation $\widehat{X}^{(t)}$, a \emph{strengthening} defines $\widehat{X}^{(t+1)} = \widehat{X}^{(t)} \setminus \widetilde{X}^{(t)}$, where $\emptyset \ne \widetilde{X}^{(t)} \subset \widehat{X}^{(t)} \setminus X$.
\end{definition}

\begin{definition} \label{d_iterative}
Given an initial subproblem relaxation $\widehat{X}^{(0)} \supset X$, an \emph{iterative strengthening method} defines a finite sequence of relaxations $\widehat{X}^{(0)} \supset \widehat{X}^{(1)} \supset \cdots \supset \widehat{X}^{(T)} \supseteq X$ via strengthenings. 
\end{definition}

Observe that, by definition, each strengthening removes at least one infeasible point from the subproblem relaxation.  As a consequence, for bounded and feasible integer programs, iterative strengthening will converge to the decomposition bound in a finite number of iterations.

Iterative strengthening methods can be classified by their strengthening step being {\em global} or {\em local}. Global methods, common in DW-based approaches, modify the transition function of the underlying dynamic program. An example is decremental state-space relaxation \citep{boland2006accelerated, righini2008new}, which reintroduces relaxed constraints by expanding the state space. Local methods, more common in AF-based approaches, refine an explicit DAG by removing infeasible paths. For example, column elimination \citep{van2022graph, karahalios2025column} detects infeasible paths in the flow decomposition of the solution and removes them along with all paths sharing the same infeasible subpath. Proposition~\ref{l_strengtheningCuts} characterizes the bound improvement achieved by the strengthening step from a cutting-plane perspective.

\begin{proposition} \label{l_strengtheningCuts}
    The bound improvement of a strengthening step can be equivalently obtained by introducing a finite number of subproblem-level cuts. 
\end{proposition}

\begin{proof}
Let $\widehat{X}^{(t)}$ and $\widehat{X}^{(t+1)} \subset \widehat{X}^{(t)}$ be consecutive subproblem relaxations. After the strengthening step, the subproblem bound is that of $\conv(\widehat{X}^{(t+1)})$. 
By the separation theorem, for every extreme point $\bar{\vb x} \in \operatorname{ext}\!\big(\conv(\widehat{X}^{(t)}) \big) \setminus \conv(\widehat{X}^{(t+1)})$, there exists an inequality $(\boldsymbol{\delta}, \delta_{0}) \in \R^{n+1}$ valid for $\conv(\widehat{X}^{(t+1)})$ that is violated by $\bar{\vb x}$. Because $\conv(\widehat{X}^{(t)})$ is a polytope with finitely many extreme points, a finite set of inequalities $\{(\boldsymbol{\delta}^{k}, \delta_{0}^{k})\}_{k \in K}$ suffices. Imposing these inequalities at the subproblem level yields $\conv(\widehat{X}^{(t+1)})=\conv\big(\big\{\vb x \in \widehat{X}^{(t)}\,:\, \boldsymbol{\delta}^{k}\vb x \leq \delta_{0}^{k}\,\, \forall k\big\}\big)$, because $\conv(\widehat{X}^{(t+1)}) \subseteq \conv(\widehat{X}^{(t)})$ and the inequalities cut off the extreme points of $\conv(\widehat{X}^{(t)})$ that lie outside $\conv(\widehat{X}^{(t+1)})$. 
\end{proof}

An illustration of strengthening by cutting planes is given in Figure~\ref{f_subproblemCut}. Proposition~\ref{l_strengtheningCuts} provides a polyhedral characterization of iterative strengthening methods with several practical implications. First, the effect of a strengthening step on DW and AF mirrors that of subproblem-level cuts discussed in~\S\ref{s_cutsX}. Second, global strengthenings are stronger--since they implicitly introduce more cuts--but also more computationally demanding, as the DAG size grows exponentially with the number of constraints~\citep{vanderbeck2000dantzig}. Finally, this polyhedral view offers an alternative proof of convergence for these techniques.

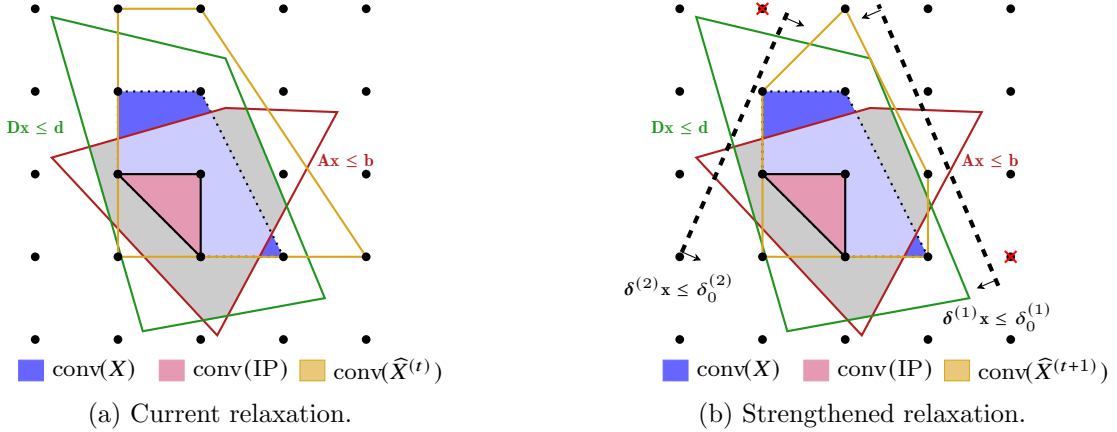
\begin{figure}[t]
  \centering
    \begin{subfigure}[t]{0.475\textwidth}
    \centering
    \begin{tikzpicture}[scale=1.1, transform shape]

    \begin{scope}
      \clip (2.2,-0.95) -- (3.65,1.75) -- (2.3,1.8) -- (0.2,1.2) -- cycle;
      \clip (0.2,2.9) -- (2.3,2.4) -- (3.5,-0.5) -- (1.3,-0.9) -- cycle;
      \fill[gray!40] (-1,-1) rectangle (5,4);
    \end{scope}

    \fill[blue!60]
      (1,2) -- (1,1) -- (2,0) -- (3,0) -- (2,2) -- cycle;

    \begin{scope}
      \clip (1,2) -- (1,1) -- (2,0) -- (3,0) -- (2,2) -- cycle;
      \clip (2.2,-0.95) -- (3.65,1.75) -- (2.3,1.8) -- (0.2,1.2) -- cycle; 
      \clip (0.2,2.9) -- (2.3,2.4) -- (3.5,-0.5) -- (1.3,-0.9) -- cycle;  
      \fill[blue!20] (-1,-1) rectangle (5,4);
    \end{scope}

    \draw[dotted, thick]
      (1,2) -- (1,1) -- (2,0) -- (3,0) -- (2,2) -- cycle;

    \fill[purple!40]  (2,0) -- (2,1) -- (1,1) -- cycle;
    \draw[black, thick] (2,0) -- (2,1) -- (1,1) -- cycle;
    \foreach \pt in {(2,0),(2,1),(1,1)} \fill[black] \pt circle (1.5pt);

    \draw[darkred, thick]
      (2.2,-0.95) -- (3.65,1.75) -- (2.3,1.8) -- (0.2,1.2) -- cycle;
    \draw[darkgreen, thick]
      (0.2,2.9) -- (2.3,2.4) -- (3.5,-0.5) -- (1.3,-0.9) -- cycle;

    \draw[amber, thick]
      (1,0) -- (1,3) -- (2,3) -- (4,0) -- cycle;

    \node[darkgreen, above] at (0,1.35) {\tiny $\vb D\vb x \leq \vb d$};
    \node[darkred,   above] at (3.75,0.95) {\tiny $\vb A\vb x \leq \vb b$};

    \foreach \x in {0,...,4}
      \foreach \y in {-1,...,3}
        \fill[black] (\x,\y) circle (1.5pt);

    \begin{scope}[shift={(-0.2,-1.5)}]
      \def\w{0.30} \def\h{0.25} \def\dx{1.4} \def\y{0.00}
      \draw[blue!60, fill=blue!60] (0,\y) rectangle ++(\w,\h);
      \node[anchor=west, inner sep=0pt] at (\w+0.1,\y+0.5*\h)
        {\scriptsize $\conv(X)$};
      \begin{scope}[shift={({\w+\dx},0)}]
        \draw[purple!40, fill=purple!40] (0,\y) rectangle ++(\w,\h);
        \node[anchor=west, inner sep=0pt] at (\w+0.1,\y+0.5*\h)
          {\scriptsize $\conv\left(\text{IP}\right)$};
      \end{scope}
      \begin{scope}[shift={({2*(\w+\dx)},0)}]
        \draw[amber, fill=amber!60] (0,\y) rectangle ++(\w,\h);
        \node[anchor=west, inner sep=0pt] at (\w+0.1,\y+0.5*\h)
          {\scriptsize $\conv(\widehat{X}^{(t)})$};
      \end{scope}
    \end{scope}
  \end{tikzpicture}%
  \caption{Current relaxation.}
\end{subfigure} \hfill
\begin{subfigure}[t]{0.475\textwidth}
  \centering
  \begin{tikzpicture}[scale=1.1, transform shape]

    \begin{scope}
      \clip (2.2,-0.95) -- (3.65,1.75) -- (2.3,1.8) -- (0.2,1.2) -- cycle;
      \clip (0.2,2.9) -- (2.3,2.4) -- (3.5,-0.5) -- (1.3,-0.9) -- cycle;
      \fill[gray!40] (-1,-1) rectangle (5,4);
    \end{scope}

    \fill[blue!60]
      (1,2) -- (1,1) -- (2,0) -- (3,0) -- (2,2) -- cycle;

    \begin{scope}
      \clip (1,2) -- (1,1) -- (2,0) -- (3,0) -- (2,2) -- cycle;
      \clip (2.2,-0.95) -- (3.65,1.75) -- (2.3,1.8) -- (0.2,1.2) -- cycle; 
      \clip (0.2,2.9) -- (2.3,2.4) -- (3.5,-0.5) -- (1.3,-0.9) -- cycle;  
      \fill[blue!20] (-1,-1) rectangle (5,4);
    \end{scope}

    \draw[dotted, thick]
      (1,2) -- (1,1) -- (2,0) -- (3,0) -- (2,2) -- cycle;

    \fill[purple!40]  (2,0) -- (2,1) -- (1,1) -- cycle;
    \draw[black, thick] (2,0) -- (2,1) -- (1,1) -- cycle;
    \foreach \pt in {(2,0),(2,1),(1,1)} \fill[black] \pt circle (1.5pt);

    \draw[darkred, thick]
      (2.2,-0.95) -- (3.65,1.75) -- (2.3,1.8) -- (0.2,1.2) -- cycle;

    \draw[darkgreen, thick]
      (0.2,2.9) -- (2.3,2.4) -- (3.5,-0.5) -- (1.3,-0.9) -- cycle;

    \draw[amber, thick] (1,0) -- (1,2) -- (2,3) -- (3,1) -- (3,0) -- cycle;

    \node[darkgreen, above] at (0,1.35) {\tiny $\vb D\vb x \leq \vb d$};
    \node[darkred,   above] at (3.75,0.95) {\tiny $\vb A\vb x \leq \vb b$};

    \foreach \x in {0,...,4}
      \foreach \y in {-1,...,3}
        \fill[black] (\x,\y) circle (1.5pt);

    \draw[red,thick] (1,3) +(-2pt,-2pt) -- +(2pt,2pt);
    \draw[red,thick] (1,3) +(-2pt,2pt) -- +(2pt,-2pt);
    
    \draw[red,thick] (4,0) +(-2pt,-2pt) -- +(2pt,2pt);
    \draw[red,thick] (4,0) +(-2pt,2pt) -- +(2pt,-2pt);
    
    \coordinate (L1a) at (3.85,-.35);
    \coordinate (L2a) at (2.40,3.05); 
    
    \draw[black, dashed, line width=1.6pt] (L1a) -- (L2a);
    
    \node[below=2pt of L1a, align=center] {\tiny $\boldsymbol{\delta}^{(1)}\vb x \leq \delta^{(1)}_{0}$};
    
    \def\arr{0.07} \def\ofs{0.02} \def\ah{3pt}
    \coordinate (B1a) at ($(L1a)!\ofs!(L2a)$);
    \coordinate (B2a) at ($(L2a)!\ofs!(L1a)$);
    \draw[-{Stealth[length=\ah,width=\ah]}] (B1a) -- ($(B1a)!\arr!+90:(L2a)$);
    \draw[-{Stealth[length=\ah,width=\ah]}] (B2a) -- ($(B2a)!\arr!-90:(L1a)$);

    \coordinate (L1b) at (1.30,2.95);
    \coordinate (L2b) at (0,0); 
    
    \draw[black, dashed, line width=1.6pt] (L1b) -- (L2b);
    \node[below=2pt of L2b, align=center] {\tiny $\boldsymbol{\delta}^{(2)}\vb x \leq \delta^{(2)}_{0}$};
    
    \def\arr{0.08} \def\ofs{0.02} \def\ah{3pt}
    \coordinate (B1b) at ($(L1b)!\ofs!(L2b)$);
    \coordinate (B2b) at ($(L2b)!\ofs!(L1b)$);
    \draw[-{Stealth[length=\ah,width=\ah]}] (B1b) -- ($(B1b)!\arr!+90:(L2b)$);
    \draw[-{Stealth[length=\ah,width=\ah]}] (B2b) -- ($(B2b)!\arr!-90:(L1b)$);
    
    \begin{scope}[shift={(-0.2,-1.5)}]
      \def\w{0.30} \def\h{0.25} \def\dx{1.4} \def\y{0.00}
      \draw[blue!60, fill=blue!60] (0,\y) rectangle ++(\w,\h);
      \node[anchor=west, inner sep=0pt] at (\w+0.1,\y+0.5*\h)
        {\scriptsize $\conv(X)$};
      \begin{scope}[shift={({\w+\dx},0)}]
        \draw[purple!40, fill=purple!40] (0,\y) rectangle ++(\w,\h);
        \node[anchor=west, inner sep=0pt] at (\w+0.1,\y+0.5*\h)
          {\scriptsize $\conv\left(\text{IP}\right)$};
      \end{scope}
      \begin{scope}[shift={({2*(\w+\dx)},0)}]
        \draw[amber, fill=amber!60] (0,\y) rectangle ++(\w,\h);
        \node[anchor=west, inner sep=0pt] at (\w+0.1,\y+0.5*\h)
          {\scriptsize $\conv(\widehat{X}^{(t+1)})$};
      \end{scope}
    \end{scope}
  \end{tikzpicture}%
  \caption{Strengthened relaxation.}
\end{subfigure} \hfill
  \caption{Iterative strengthening methods from a cutting-plane perspective.}
  \label{f_subproblemCut}
\end{figure}

\section{Case Study: The Vehicle Routing Problem with Time Windows} \label{s_VRPTW}

We illustrate our comparative analysis on the VRPTW. The problem is defined on a directed graph $G = (V, A)$, where $N$ is the customer set, $V := N \cup \{\underline{0}, \overline{0}\}$ is the vertex set, and $A \subset V \times V$ is the arc set. The objective is to determine a set of minimum-cost routes for $K$ identical vehicles, each departing from a single depot represented by two vertices: the source $\underline{0}$ and the sink $\overline{0}$. Each customer $i \in N$ must be visited exactly once within its time window $[e_i, l_i]$ to satisfy its demand $d_i \in \mathbb{Z}_{+}$, while the vehicle load must not exceed capacity $C \in \mathbb{Z}_{+}$. We set $d_{\underline{0}} = d_{\overline{0}} = 0$ and $[e_{\underline{0}}, l_{\underline{0}}] = [e_{\overline{0}}, l_{\overline{0}}] = [E, L]$. Each arc $(i,j) \in A$ has cost $c_{ij} \ge 0$ and travel time $\tau_{ij} \ge 0$, including any service time at vertex $i$. Let $x_{ij}^{k} \in \{0,1\}$ indicate whether vehicle $k$ traverses arc $(i,j)$, $z_{i}^{k} \in \{0,1\}$ whether it visits customer $i$, and $t_{i}^{k} \ge 0$ the service start time at vertex $i$. The IP formulation is given by \eqref{eq_IPVRPTW}, where $M_{ij} := \max\{0, l_i + \tau_{ij} - e_j\}$ for $(i,j) \in A$. The objective~\eqref{eq_IP_0} minimizes total routing cost. Constraints~\eqref{eq_IP_1} ensure that each customer is visited once. Constraints~\eqref{eq_IP_2}--\eqref{eq_IP_5} define feasible routes. Constraints~\eqref{eq_IP_6} link $\vb x^k$ and $\vb z^k$, where the $\vb z^k$ variables could be \emph{substituted out} but are kept for clarity. Finally, Constraints~\eqref{eq_IP_7} define the variable domains.

\journalvspace[-0.5cm]
{\small
\begin{subequations} \label{eq_IPVRPTW}
\begin{align}
\min \quad & \sum_{k =1}^{K} \sum_{(i,j) \in A} c_{ij} x_{ij}^k \label{eq_IP_0} \\
\text{s.t.} \quad 
& \sum_{k = 1}^{K} z_{i}^{k} = 1 && i \in N \label{eq_IP_1} \\
& \sum_{j\in\delta^{+}(i)} x^{k}_{ij} - \sum_{j\in \delta^{-}(i)} x^{k}_{ji} =
\begin{cases}
1, & i = \underline{0} \\
0, & i \in N \\
-1, & i = \overline{0}
\end{cases} && i \in V,\; k=1,\ldots, K \label{eq_IP_2} \\
& \sum_{i \in N} d_i\,z^{k}_i \le C && k =1,\ldots, K \label{eq_IP_3} \\
& e_i\,z^{k}_i \le t^{k}_i \le l_i\,z^{k}_i && i \in N,\; k = 1,\ldots, K \label{eq_IP_4} \\
& E \leq t_{i}^{k} \leq L && i \in \{\underline{0}, \overline{0}\}, \; k =1, \ldots, K \\
& t_{i}^{k} + \tau_{ij} - M_{ij}\bigl(1 - x^{k}_{ij}\bigr) \le t_{j}^{k} && (i,j) \in A,\; k = 1,\ldots, K \label{eq_IP_5} \\
& z^{k}_i = \sum_{j \in\delta^{+}(i)} x^{k}_{ij} && i \in N, \; k = 1,\ldots, K \label{eq_IP_6} \\
& \vb x^{k} \in \{0,1\}^{|A|},\; \;\vb z^{k} \in \{0,1\}^{|N|}, \; \; \vb t^{k} \in \R^{|V|}_{+} && k = 1,\ldots, K, \label{eq_IP_7}
\end{align}
\end{subequations}
}

\subsection{Dantzig-Wolfe Reformulation}  \label{ss_VRPTW_DW}

Selecting~\eqref{eq_IP_1} as the \emph{complicating} constraints and~\eqref{eq_IP_2}--\eqref{eq_IP_6} as the \emph{nice} ones (separable and identical for each vehicle $k=1,\ldots, K$) yields $X:= \{ (\vb x, \vb z) \in \{0,1\}^{|A|+|N|}  \, : \exists \, \vb t \in \R^{|V|}_{+} \text{ s.t. } ~\eqref{eq_IP_2}-\eqref{eq_IP_6} \}=\{(\vb x^{q}, \vb z^{q})\}_{q \in Q}$. Because of~\eqref{eq_IP_6}, we write $\vb x^{q} \in X$. DW reformulation follows:

\journalvspace[-1.0cm]

\begin{subequations} \label{eq_IMP}
\begin{align}
\min \quad & \sum_{k=1}^{K} \sum_{q \in Q} c_{q} \lambda^{k}_{q} \label{eq_IMP_1} \\
\text{s.t.} \quad & \sum_{k=1}^{K} \sum_{q \in Q} z_{i}^{q} \lambda^{k}_{q} = 1 && i \in N \label{eq_IMP_2} \\
& \sum_{q \in Q} \lambda^{k}_{q} = 1 && k = 1, \ldots, K \label{eq_IMP_3} \\
&\lambda_{q}^{k} \in \{0,1\} \quad && k=1,\ldots, K, \, q \in Q,  \label{eq_IMP_6}
\end{align}
\end{subequations}

\journalvspace[-0.75cm]

\noindent where $c_{q}:= \sum_{(i,j) \in A}c_{ij}x_{ij}^{q}$ for every $q \in Q$. As discussed in Appendix \ref{s_blockDiagonal}, an aggregation can be performed to remove index $k$, but we keep it for the cutting planes discussion.

In DW--RMP, we have dual variables $\boldsymbol{\pi}\in \R^{|N|}$ and $\pi_{0} \in \R$ associated with constraints~\eqref{eq_IMP_2} and~\eqref{eq_IMP_3}, and a modified cost $\overline{c}_{ij}=c_{ij} - \pi_{i}$ can be associated with each arc $(i,j) \in A$, where $\pi_{\underline{0}}:=\pi_{0}$. The subproblem $X$ is an Elementary Shortest Path Problem with Resource Constraints (ESPPRC), which can be more efficiently solved as a dynamic program in which states are represented by labels encoding subpaths in $G$ (see \cite{desrosiers2024branch}). Each label is a tuple $\mathscr{L} = (\mathscr{v}, \mathscr{c}, \mathscr{d}, \mathscr{t}, \mathscr{E})$, where $\mathscr{v} \in V$ is the current vertex, $\mathscr{c} \in \mathbb{R}$ is the cumulative reduced cost, $\mathscr{d} \in \mathbb{Z}_+$ is the cumulative load, $\mathscr{t} \in [E, L]$ is the earliest service start time, and $\mathscr{E} \subseteq N$ is the set of visited customers. These tuple components are commonly referred to as \emph{resources}.

The initial label at the source depot $\underline{0}$ is $\mathscr{L} = (\underline{0}, 0, 0, E, \emptyset)$ and labels are extended through outgoing arcs until reaching the sink depot $\overline{0}$. Given a label $\mathscr{L}$ with $\mathscr{v}(\mathscr{L}) = i$, an extension along arc $(i, j) \in A$ (i.e., a state-transition), such that $j \notin \mathscr{E}(\mathscr{L})$, produces a new label $\mathscr{L}'$, where:
\begin{equation}\label{eq_transition}
\begin{aligned}
\mathscr{v}(\mathscr{L}') &= j, \quad
\mathscr{c}(\mathscr{L}') = \mathscr{c}(\mathscr{L}) + \overline{c}_{ij}, \quad
\mathscr{d}(\mathscr{L}') = \mathscr{d}(\mathscr{L}) + d_j, \\
\mathscr{t}(\mathscr{L}') &= \max\{\mathscr{t}(\mathscr{L}) + \tau_{ij}, e_j\}, \quad
\mathscr{E}(\mathscr{L}') = \mathscr{E}(\mathscr{L}) \cup \{j\}.
\end{aligned}
\end{equation}
The resulting label $\mathscr{L}'$ is feasible if $\mathscr{d}(\mathscr{L}') \leq C$ and $\mathscr{t}(\mathscr{L}') \leq l_{j}$. To avoid enumerating all feasible solutions, a \emph{labeling algorithm} discards non-useful subpaths by verifying a dominance rule between labels. Given two label $\mathscr{L}_{1}$ and $\mathscr{L}_{2}$ with $\mathscr{v}(\mathscr{L}_{1})= \mathscr{v}(\mathscr{L}_{2})$, $\mathscr{L}_{1}$ dominates $\mathscr{L}_{2}$ if:  
\begin{equation}\label{eq_dominance}
\mathscr{c}(\mathscr{L}_{1}) \leq \mathscr{c}(\mathscr{L}_{2}) \;\wedge\;
\mathscr{d}(\mathscr{L}_{1}) \leq \mathscr{d}(\mathscr{L}_{2}) \;\wedge\;
\mathscr{t}(\mathscr{L}_{1}) \leq \mathscr{t}(\mathscr{L}_{2}) \;\wedge\;
\mathscr{E}(\mathscr{L}_{1}) \subseteq \mathscr{E}(\mathscr{L}_{2}).
\end{equation}
Conditions~\eqref{eq_dominance} guarantee that any feasible extension of the dominated label $\mathscr{L}_{2}$ is also a feasible extension of the dominant label $\mathscr{L}_{1}$ with an equal or better reduced cost. 

Solving the ESPPRC is strongly NP-hard, whereas the SPPRC can be solved in pseudopolynomial time \citep{desrochers1992new}. For this reason, several elementarity-based relaxations have been proposed, with the \emph{ng}-SPPRC being the most widely used \citep{baldacci2011new}. In this relaxation, each customer $i \in N$ is associated with an \emph{ng}-set $M_i \subseteq N$ containing the $\Delta \in \mathbb{Z}_+$ closest customers to $i$ and customer $i$ itself. An \emph{ng}-route allows a cycle starting and ending at customer $j\in N$ if and only if there exists a customer $i\in N$ in the cycle for which $j \notin M_{i}$.

In the labeling algorithm, the set $\mathscr{E}(\mathscr{L})$ is replaced by $\mathscr{U}(\mathscr{L})$, the set of customers whose visit would violate the \emph{ng}-route cycling restrictions. Unlike $\mathscr{E}(\mathscr{L})$, in the set $\mathscr{U}(\mathscr{L})$, customers may be forgotten and subsequently revisited. When a new label $\mathscr{L}'$ is created, then $\mathscr{U}(\mathscr{L}') = \left(\mathscr{U}(\mathscr{L}) \cap M_j\right) \cup \{j\}$. In the dominance rule, $\mathscr{E}(\mathscr{L}_{1}) \subseteq \mathscr{E}(\mathscr{L}_{2})$ is replaced by $\mathscr{U}(\mathscr{L}_1) \subseteq \mathscr{U}(\mathscr{L}_2)$. Since $|\mathscr{U}(\mathscr{L})| \leq \Delta$, this dominance criterion is less restrictive, potentially allowing more labels to be discarded. The bound quality produced by the \emph{ng}-SPPRC depends on parameter $\Delta$. When $\Delta = |N| - 1$, the \emph{ng}-SPPRC reduces to the ESPPRC; when $\Delta = 0$, it becomes the SPPRC. In the remainder, $X$ denotes the \emph{ng}-SPPRC.

\subsection{Arc-Flow Reformulation} \label{ss_VRPTW_AF}

A DAG $\mathcal{D} = (\mathcal{N}, \mathcal{E})$ satisfying Definition~\ref{d_arcflowRefor}(i)-(ii) for $X$ can be derived from the unfolded dynamic program associated with the \emph{ng}-SPPRC. Each node $u \in \mathcal{N}$ represents a state defined by a tuple $(\mathscr{v}, \mathscr{d}, \mathscr{t}, \mathscr{U})$, and each arc $e \in \mathcal{E}$ corresponds to a state transition, as defined in~\eqref{eq_transition}. Moreover, the root node $r \in \N$ encodes the initial state $(\underline{0}, 0, E, \emptyset)$. Figure~\ref{f_example} illustrates this construction for a small instance with $K = 2$, $C = 3$, $c_{ij}=\tau_{ij}$ for all $(i,j) \in A$, and vertex $0$ represents both $\underline{0}$ and $\overline{0}$.

Due to state-transitions \eqref{eq_transition}, the column associated with arc $e = (u,v) \in \E$ in the projection matrix $\vb T \in \{0,1\}^{|A| \times |\E|}$ has an entry equal to one in the row corresponding to arc $\big(\mathscr{v}(u), \mathscr{v}(v)\big) \in A$. Explicitly, this projection is $x^{k}_{ij}=\sum_{e=(u,v) \in \E \, : \, \mathscr{v}(u) =i, \mathscr{v}(v) = j}y^{k}_{e}$ for $(i,j) \in A$ and $k=1,\ldots, K$. The cost coefficient of arc $e \in \E$ is $c_{e} = (\vb c \vb T)_{e}  = c_{\mathscr{v}(u), \mathscr{v}(v)}$. The coefficient of a constraint~\eqref{eq_IP_1} associated with customer $i \in N$ for an arc $e = (u, v) \in \E$ is $1$ if $\mathscr{v}(u) = i$, and $0$ otherwise. Appendix~\ref{A_example_computation} shows these computations explicitly for the example in Figure~\ref{f_example}. 

The AF reformulation of~\eqref{eq_IPVRPTW} is given in \eqref{eq_AFVRPTW}. As in~\eqref{eq_IMP}, we keep index $k$ for the cutting planes discussion. In AF--RMP, we have dual variables $\boldsymbol{\pi}\in \R^{|N|}$ and $\boldsymbol{\rho} \in \R^{|\N'|}$ associated with constraints~\eqref{eq_AF_VRPTW_1} and~\eqref{eq_AF_VRPTW_2}--\eqref{eq_AF_VRPTW_3}, and a modified cost $\overline{c}_{e}=c_{e} - \pi_{\mathscr{v}(u)}$ can be associated with each arc $e=(u,v) \in \E$, where $\pi_{\underline{0}}:=\rho_{r}$.

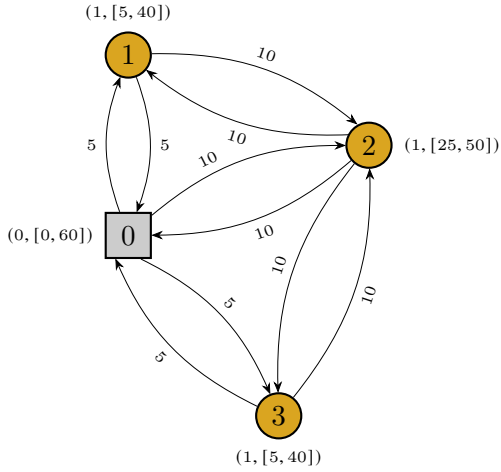
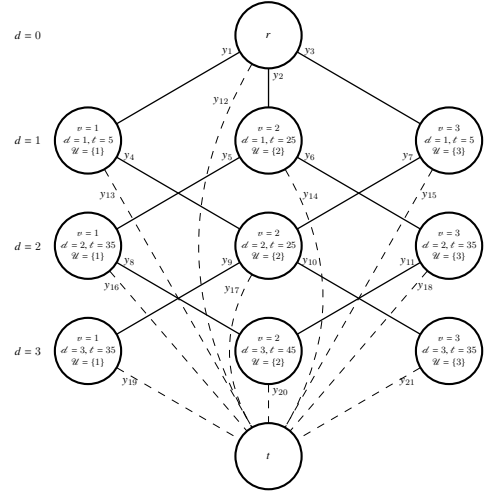
\begin{figure}[htbp]
  \centering
  \begin{subfigure}[b]{0.45\textwidth}
  \centering
  \begin{tikzpicture}[thick, scale=0.80, >=Stealth]
    \node[depot,  label=left:{\tiny$(0,[0,60])$}]      (0) at (0,0) {0};
    \node[customer,label=above:{\tiny$(1,[5,40])$}]    (1) at (0,3) {1};
    \node[customer,label=right:{\tiny$(1,[25,50])$}]   (2) at (4,1.5) {2};
    \node[customer,label=below:{\tiny$(1,[5,40])$}]    (3) at (2.5,-3) {3};

    \draw[->,line width=0.2pt,bend left=20] (0) to node[midway,left]            {\tiny $5$}  (1);
    \draw[->,line width=0.2pt,bend left=20] (1) to node[midway,right]           {\tiny $5$}  (0);

    \draw[->,line width=0.2pt,bend left=20]
  (0) to node[pos=0.35,above,sloped] {\tiny $10$} (2);
    \draw[->,line width=0.2pt,bend left=20] (2) to node[midway,below,sloped]    {\tiny $10$} (0);

    \draw[->,line width=0.2pt,bend left=20] ([xshift=6pt]0.south) to node[midway,above,sloped] {\tiny $5$}  (3);
    \draw[->,line width=0.2pt,bend left=20] (3) to node[midway,below,sloped]    {\tiny $5$}  ([xshift=-6pt]0.south);

    \draw[->,line width=0.2pt,bend left=20]  (1) to node[midway,above,sloped]   {\tiny $10$} ([xshift=5.5pt,yshift=10pt]2.west);
    \draw[->,line width=0.2pt,bend left=20]  ([xshift=2pt,yshift=5pt]2.west) to node[midway,below,sloped] {\tiny $10$} (1);

    \draw[->,line width=0.2pt,bend right=20] (2) to node[midway,above,sloped]   {\tiny $10$} (3);
    \draw[->,line width=0.2pt,bend right=20] (3) to node[midway,below,sloped]   {\tiny $10$} (2);
  \end{tikzpicture}

  \caption{Small-size instance.}
  \label{fig_instance}
\end{subfigure}
  \hfill
  \begin{subfigure}[b]{0.52\textwidth}
    \centering
    \begin{tikzpicture}[
        scale=0.40,
        every node/.style={transform shape},
        >=Stealth,
        thick,
        node distance=2.2cm and 3.0cm,
        on grid,
        auto,
        myNode/.style={
          circle, draw,
          fill=white,
          minimum size=22mm,
          inner sep=1pt,
          font=\footnotesize,
          align=center      
        },
        every edge/.style={->},
        every label/.append style={midway,font=\tiny},
      ]
      \pgfdeclarelayer{background}
      \pgfsetlayers{background,main}

      \node (layer0) at (-8,   0) {$d=0$};
      \node (layer1) at (-8, -3.5) {$d=1$};
      \node (layer2) at (-8, -7.0) {$d=2$};
      \node (layer3) at (-8,-10.5) {$d=3$};

      \node[myNode] (R) at (0,0) {\normalsize $r$};

      \node[myNode] (A) at (-6,-3.5) {
        $\mathscr{v}=1$\\$\mathscr{d}=1,\mathscr{t}=5$\\$\mathscr{U}=\{1\}$
      };
      \node[myNode] (B) at (0,-3.5) {
        $\mathscr{v}=2$\\$\mathscr{d}=1,\mathscr{t}=25$ \\$\mathscr{U}=\{2\}$
      };
      \node[myNode] (C) at (6,-3.5) {
        $\mathscr{v}=3$\\$\mathscr{d}=1,\mathscr{t}=5$ \\$\mathscr{U}=\{3\}$
      };

      \node[myNode] (M) at (0,-7.0) {
        $\mathscr{v}=2$\\$\mathscr{d}=2,\mathscr{t}=25$ \\ $\mathscr{U}=\{2\}$
      };
      \node[myNode] (F) at (-6,-7.0) {
        $\mathscr{v}=1$\\$\mathscr{d}=2,\mathscr{t}=35$ \\$\mathscr{U}=\{1\}$
      };
      \node[myNode] (G) at (6,-7.0) {
        $\mathscr{v}=3$ \\$\mathscr{d}=2,\mathscr{t}=35$ \\ $\mathscr{U}=\{3\}$
      };

      \node[myNode] (I) at (-6,-10.5) {
        $\mathscr{v}=1$ \\$\mathscr{d}=3,\mathscr{t}=35$ \\ $\mathscr{U}=\{1\}$
      };

     \node[myNode] (I2) at (0,-10.5) {
        $\mathscr{v}=2$ \\$\mathscr{d}=3,\mathscr{t}=45$ \\ $\mathscr{U}=\{2\}$
      };
      
      \node[myNode] (J) at (6,-10.5) {
        $\mathscr{v}=3$\\$\mathscr{d}=3,\mathscr{t}=35$ \\ $\mathscr{U}=\{3\}$
      };

      \node[myNode] (T) at (0,-14) {\normalsize $t$};

      \draw[line width=0.5pt] (R) -- node[pos=0.1, above] {$y_{1}$} (A);
      \draw[line width=0.5pt] (R) -- node[pos=0.2, right] {$y_{2}$} (B);
      \draw[line width=0.5pt] (R) -- node[pos=0.1, above] {$y_{3}$} (C);
      \draw[line width=0.5pt] (A) -- node[pos=0.1, above] {$y_{4}$} (M);
      \draw[line width=0.5pt] (B) -- node[pos=0.1, above] {$y_{5}$} (F);
      \draw[line width=0.5pt] (B) -- node[pos=0.1, above] {$y_{6}$} (G);
      \draw[line width=0.5pt] (C) -- node[pos=0.1, above] {$y_{7}$} (M);
      \draw[line width=0.5pt] (F) -- node[pos=0.1, above] {$y_{8}$} (I2);
      \draw[line width=0.5pt] (M) -- node[pos=0.1, above] {$y_{9}$} (I);
      \draw[line width=0.5pt] (M) -- node[pos=0.1, above] {$y_{10}$} (J);
      \draw[line width=0.5pt] (G) -- node[pos=0.1, above] {$y_{11}$} (I2);

      \begin{pgfonlayer}{background}
        \draw[dashed, line width=0.1pt, bend right=30] (R) to node[pos=0.1, left] {$y_{12}$} (T);
        \draw[dashed, line width=0.1pt] (A) -- node[pos=0.1, left] {$y_{13}$} (T);
        \draw[dashed, line width=0.1pt, bend left=30] (B) to node[pos=0.1, right] {$y_{14}$} (T);
        \draw[dashed, line width=0.1pt] (C) -- node[pos=0.1, right] {$y_{15}$} (T);
        \draw[dashed, line width=0.1pt, bend right=30] (M) to node[pos=0.1, left] {$y_{17}$} (T);
        \draw[dashed, line width=0.1pt] (F) -- node[pos=0.1, left] {$y_{16}$} (T);
        \draw[dashed, line width=0.1pt] (G) -- node[pos=0.1, right] {$y_{18}$} (T);
        \draw[dashed, line width=0.1pt] (I) -- node[pos=0.2, left] {$y_{19}$} (T);
        \draw[dashed, line width=0.1pt] (I2) -- node[pos=0.2, right] {$y_{20}$} (T);
        \draw[dashed, line width=0.1pt] (J) -- node[pos=0.2, right] {$y_{21}$} (T);
      \end{pgfonlayer}

    \end{tikzpicture}

    \caption{DAG $\D=(\N, \E)$ with $\Delta = 0$.}
    \label{f_DAG}
  \end{subfigure}

  \caption{Illustrative example.}
  \label{f_example}
\end{figure}

\begin{subequations} \label{eq_AFVRPTW}
  \begin{align}
    \min \quad & \sum_{k=1}^{K} \sum_{e \in \E} c_{e} y^{k}_{e}  \\
    \text{s.t.} \quad
    & \sum_{k=1}^{K}\sum_{\substack{e =(u,v) \in \E : \\ \, \mathscr{v}(u) = i}} y^{k}_{e} = 1 && i \in N \label{eq_AF_VRPTW_1}\\
    & \sum_{e \in \delta^{+}(r)} y^{k}_{e} = 1  && k =1,\ldots,K \label{eq_AF_VRPTW_2}\\ 
    & \sum_{e \in \delta^{+}(u)} y^{k}_{e} - \sum_{e \in \delta^{-}(u)} y^{k}_{e} = 0 && u \in \N \setminus\{r,t\},\, k =1,\ldots,K \label{eq_AF_VRPTW_3}\\ 
    & y^{k}_{e} \in \{0,1\} && e \in \E, k = 1,\ldots, K.
  \end{align}
\end{subequations}

\subsection{Cutting Planes}

\paragraph{\textbf{Cuts in the Original Space}} Several families of valid inequalities for the VRP have been derived over the original space, with generalized subtour elimination constraints (GSECs) being among the most popular ones \citep{laporte1986generalized}. Given a subset $S \subseteq N$ of customers and a lower bound $\kappa_{S}$ on the number of vehicles required to serve all customers in $S$, a GSEC for~\eqref{eq_IPVRPTW} is written as:
\begin{equation}
    \sum_{k=1}^{K}\sum_{(i,j) \in \delta^{-}(S)} x^{k}_{ij} \geq \kappa_{S},
    \label{eq_kPCs}
\end{equation}
where $\delta^{-}(S):=\{(i,j) \in A \, : \, i \in V \setminus S, \, j \in S\}$ is the set of arcs entering $S$. In the reformulations, a cut of the form~\eqref{eq_kPCs} is kept at the master level, as it involves all vehicles. In DW~\eqref{eq_IMP}, this cut is:
$
\sum_{k=1}^{K} \sum_{q \in Q} \left( \sum_{(i,j) \in \delta^{-}(S)} x_{ij}^{q} \right) \, \lambda_{q}^{k} \geq \kappa_{S},
$
where $\mu \geq 0$ is the associated dual variable in DW--RMP. For each arc $(i,j) \in \delta^{-}(S)$, the modified cost becomes $\overline{c}_{ij} = c_{ij} - \pi_{i} - \mu$. In AF~\eqref{eq_AFVRPTW}, the lifted cut is:
\begin{equation}
  \sum_{k=1}^{K}
  \sum_{\substack{e=(u,v)\in \mathcal{E}:\\
                  \mathscr{v}(u)\notin S,\;
                  \mathscr{v}(v)\in S}}
        y^{k}_{e}
  \;\ge\;
  \kappa_{S}, \label{eq_AF_kPC}
\end{equation}
\noindent and the modified arc cost becomes $\overline{c}_{e}=c_{e} - \pi_{\mathscr{v}(u)}-\mu$ for each arc $e =(u,v) \in \E$ such that $\mathscr{v}(u) \notin S$ and $\mathscr{v}(v) \in S$. In Figure~\ref{f_example}(a) and~\ref{f_example}(b), a GSEC~\eqref{eq_kPCs} for $S=\{1,3\}$ and its lifted counterpart~\eqref{eq_AF_kPC} are: $\sum_{k=1}^{K} (x^{k}_{\underline{0},1} + x^{k}_{\underline{0}, 3}+x^{k}_{2,1} + x^{k}_{2,3}) \geq \kappa_{S}$ and $\sum_{k=1}^{K} (y^{k}_{1} + y^{k}_{3} + y^{k}_{5} + y^{k}_{6} + y^{k}_{9} + y^{k}_{10}) \geq \kappa_{S}$. 

\paragraph{\textbf{Cuts in the Dantzig-Wolfe Space}} General cuts, such as the CG rank-1 cuts, have been used for DW~\eqref{eq_IMP}. A CG rank-1 cut for~\eqref{eq_IMP} on constraints~\eqref{eq_IMP_2} is given by:
\begin{equation}
    \sum_{k=1}^{K} \sum_{q \in Q}\left\lfloor  \sum_{i \in N} u_{i} z_{i}^{q} \right \rfloor \lambda_{q}^{k} \leq \left\lfloor \sum_{i \in N} u_{i} \right \rfloor,
    \label{eq_CGCut}
\end{equation}
where $\vb u \in [0,1)^{|N|} \cap \Q^{|N|}_{+}$ and $g(\vb A \vb x^{q}) := \left\lfloor  \sum_{i \in N} u_{i} z_{i}^{q} \right \rfloor$ is the function computing the cut coefficient for the $\lambda_{q}^{k}$-variable. As outlined by \cite{petersen2008chvatal}, function $g$ can be linearized in the original space by adding integer variables $\gamma^{k}$ for $k=1,\ldots, K$ and the following constraints to IP~\eqref{eq_IPVRPTW}:
\[
\sum_{k=1}^{K} \gamma^{k} \leq \left\lfloor \sum_{i \in N} u_{i}\right\rfloor, \;\;\;\;
    \sum_{i \in N} u_{i}  z_{i}^{k} - (1-\epsilon) \leq \gamma^{k} \leq  \sum_{i \in N}u_{i}  z_{i}^{k}  \;\;\; \forall k,\;\;\;\; \text{and} \;\;\;\;
    \gamma^{k} \in \Z_{+}   \;\;\; \forall k,
\]

\noindent where a small $\epsilon>0$ is used to model strict inequality. Keeping the first constraint at the master level and the remaining ones at the subproblem level yields $\widehat{X}:=\big\{(\vb x, \gamma) \in \{0,1\}^{|A|}\times \Z_{+} \, : \, \vb x \in X, \gamma = g(\vb A \vb x)\big\}$. In the dynamic program, variable $\gamma$ expands the state-space by introducing a new resource $\mathscr{g}(\mathscr{L}) \in [0,1)$, capturing the fractional part of the cut coefficient. When extending a label $\mathscr{L}$ along arc $(i,j) \in A$ with $j \in N$, this resource is updated as $\mathscr{g}(\mathscr{L}') =  \left( \mathscr{g}(\mathscr{L}) + u_{j} \right) \mod 1$. Whenever $\mathscr{g}(\mathscr{L}) > \mathscr{g}(\mathscr{L}')$, the cut coefficient increases by one. So, the cumulative reduced cost is updated as $\mathscr{c}(\mathscr{L}') =  \mathscr{c}(\mathscr{L}) + \overline{c}_{ij} - \mu \cdot [\mathscr{g}(\mathscr{L})>\mathscr{g}(\mathscr{L}')]$, where $\mu \leq 0$ is the dual variable associated with cut~\eqref{eq_CGCut} in DW--RMP and $[\mathbb{P}] = 1$ if $\mathbb{P}$ is true, and $[\mathbb{P}] = 0$ otherwise. Finally, the dominance rule~\eqref{eq_dominance} must account for these new resources (one per cut), as detailed in \cite{jepsen2008subset}.

A DAG $\widehat{\D}=(\widehat{\N}, \widehat{\E})$ satisfying Definition~\ref{d_arcflowRefor}(i)-(ii) for $\widehat{X}$ can be constructed from this expanded dynamic program. The projection is given by $\gamma^{k} = \sum_{e=(u,v) \in \widehat{\E} \, : \, \mathscr{g}(u) > \mathscr{g}(v)} y_{e}^{k}$, leading to the following lifted cut in AF~\eqref{eq_AFVRPTW}:
\begin{equation*}
\sum_{k=1}^{K}\sum_{\substack{e=(u,v) \in \widehat{\E} : \\ \mathscr{g}(u) > \mathscr{g}(v)}}y_{e}^{k} \leq \left\lfloor \sum_{i \in N} u_{i} \right \rfloor,
\end{equation*}
\noindent and the modified arc cost becomes $\overline{c}_{e}=c_{e}-\pi_{\mathscr{v}(u)}-\mu$ for $e=(u,v) \in \widehat{\E}$ such that $ \mathscr{g}(\mathscr{L})>\mathscr{g}(\mathscr{L}')$. Appendix \ref{A_CGRank1Cut} illustrates a CG rank-1 cut in the DW and AF spaces for the example in Figure \ref{f_example}. 

\paragraph{\textbf{Iterative Subproblem Strengthening}} 
We apply two strengthening methods for the VRPTW: decremental state-space relaxation and column elimination.  Appendix~\ref{A_subproblemRelax} illustrates both methods on the example in Figure \ref{f_example}. 

\section{Empirical Comparison of Dantzig-Wolfe and Arc-Flow Reformulations} \label{s_computational}

\subsection{Experimental Setup} \label{ss_setup}
To test our computational insights, we benchmark DW and AF on the linear relaxation (root node) of the VRPTW. We use the benchmark of \cite{solomon1987algorithms}, which follows a four-parameter naming convention \texttt{DTm-n}. Parameter \texttt{D} is the customer geographical distribution: R for random, C for clustered, or RC for a mix of random and clustered. Parameter \texttt{T} is the tightness of the time windows, where instances of type 1 have tighter time windows than instances of type 2. Parameter \texttt{m} is the two-digit instance number, and parameter \texttt{n} is the number of customers.

We solve both DW and AF using column generation. The subproblem relaxation is defined as an \emph{ng}-SPPRC and is solved by the labeling algorithm described before. The (exact) labeling algorithm is invoked only when a heuristic algorithm fails to find a negative reduced cost route. This heuristic algorithm is derived from the exact one by disregarding the condition $\mathscr{U}(\mathscr{L}_1) \subseteq \mathscr{U}(\mathscr{L}_2)$ in the dominance check~\eqref{eq_dominance}. Since the vehicle capacity is not very constraining in these instances, we add a single rounded capacity inequality~\eqref{eq_kPCs} involving all customers a priori. We employ the $\alpha$-schedule dual-price smoothing technique, which helps mitigate common convergence issues in column generation \citep{pessoa2018automation}. Based on preliminary experiments, we set the initial smoothing factor to $0.65$ for DW and $0.50$ for AF. For aggregated results, we report the geometric mean, a more robust and conservative measure than the arithmetic mean, as it is less sensitive to large values. Note that we benchmark DW and AF under identical conditions, without tailored enhancements.

All algorithms are implemented in Java using the Branch-and-Price framework of the Java OR library (\href{https://github.com/coin-or/jorlib}{\texttt{jORLib}}) and the Java graph theory library (\href{https://jgrapht.org/}{\texttt{JGraphT}}) by \cite{michail2020jgrapht}. The LPs are solved using IBM CPLEX solver, version 22.1. All experiments are run on a machine with an Intel(R) Xeon(R) Gold 6230R CPU @ 2.10GHz  with 32GB of RAM and a two-hour time limit. 

\subsection{Assessment of Column Generation and Reformulation Sizes} \label{ss_cg_sizes}

We first evaluate the computational benefits of using column generation to solve DW and AF, i.e., solving DW--RMP and AF--RMP, against solving DW--MP and AF--MP statically, after full enumeration. We also analyze reformulation sizes to gain insight into the dimension of the DW and AF spaces. Because enumeration is computationally expensive, these experiments are limited to the \texttt{R1} 15-customer instances with 
\emph{ng}-sets of size $\Delta = 6$. We hypothesize that optimal RMPs are orders of magnitude smaller than fully enumerated MPs. By construction, the number of columns in DW--MP equals the number of paths in AF--MP; we expect a similar correspondence between~RMPs.

Table~\ref{t_combined_15} compares the full MPs and optimal RMPs of DW and AF. Columns 2--5 report the number of columns and CPU time for solving DW--MP and DW--RMP, while Columns 6--7 present the corresponding ratios. Columns 8--11 show the number of paths and CPU time for solving AF--MP and AF--RMP, and Columns 12--15 provide the associated ratios, including the number of arcs and nodes in AF--MP relative to AF--RMP. CPU times for the full MPs exclude variable enumeration.

For DW, the MP includes about 155 times more columns than the RMP. While the static approach is slightly faster on a few small instances (e.g., \texttt{R101-15}), the dynamic approach is on average 1.3 times faster. For AF, the DAG in the MP contains roughly 150 times more paths than that of the RMP, resulting in AF–MP having about 90 times more arcs (variables) and 100 times more nodes (constraints). Consequently, the dynamic approach to AF is, on average, approximately six times faster than static AF. Finally, the number of DW–MP columns and AF–MP paths are equivalent, while their RMP counterparts are similar, averaging 133 columns for DW and 139 paths for AF.

\begin{table}[tb]
  \centering
    \begin{adjustbox}{max width=\textwidth}
    \begin{tabular}{l*{14}{r}}
      \toprule
      & \multicolumn{6}{c}{Dantzig-Wolfe}
      & \multicolumn{8}{c}{Arc-Flow}
      \\  \cmidrule(lr){2-7} \cmidrule(lr){8-15}
      
      & \multicolumn{2}{c}{MP}
      & \multicolumn{2}{c}{RMP}
      & \multicolumn{2}{c}{MP-to-RMP ratios}
      & \multicolumn{2}{c}{MP}
      & \multicolumn{2}{c}{RMP}
      & \multicolumn{4}{c}{MP-to-RMP ratios} \\
      \cmidrule(lr){2-3} \cmidrule(lr){4-5} \cmidrule(lr){6-7} \cmidrule(lr){8-9} \cmidrule(lr){10-11} \cmidrule(lr){12-15}
      Instance
        & Columns & Time (s)
        & Columns & Time (s)
        & Columns & Time
        & Paths & Time (s)
        & Paths & Time (s)
        & Paths & Time
        & Arcs & Nodes \\
      \midrule
      R101-15
        &    204 & 0.01 &    39 & 0.05 &   5.2 & 0.20
        &    204 & 0.02 &    55 & 0.07 &   3.7 & 0.29
        &   3.6  &   3.5  \\
      R102-15
        & 27,739 & 0.15 &   110 & 0.07 & 252.2 & 2.14
        & 27,739 & 0.4  &   109 & 0.1  & 254.5 & 4.00
        & 152.7 & 158.2 \\
      R103-15
        & 36,707 & 0.22 &   125 & 0.09 & 293.7 & 2.44
        & 36,707 & 0.63 &   128 & 0.13 & 286.8 & 4.85
        & 160.9 & 170.6 \\
      R104-15
        &164,120 & 0.93 &   205 & 0.14 & 800.6 & 6.64
        &164,120 & 7.78 &   182 & 0.13 & 901.8 &59.85
        &455.8  &510.8 \\
      R105-15
        &    680 & 0.01 &    78 & 0.06 &   8.7 & 0.01
        &    680 & 0.02 &    98 & 0.11 &   6.9 & 0.18
        &   5.3  &   5.8  \\
      R106-15
        & 53,679 & 0.30 &   129 & 0.09 & 416.1 & 3.33
        & 53,679 & 1.71 &   134 & 0.11 & 400.6 &15.55
        &234.1  &256.1 \\
      R107-15
        & 66,629 & 0.36 &   144 & 0.10 & 462.7 & 3.60
        & 66,629 & 2.30 &   144 & 0.12 & 462.7 &19.17
        &244.2  &288.8 \\
      R108-15
        &236,941 & 1.32 &   226 & 0.12 &1,048.4 &11.00
        &236,941 &26.32 &   221 & 0.18 &1,072.1 &146.22
        &507.6  &638.4 \\
      R109-15
        &  4,315 & 0.04 &   103 & 0.07 &  41.9 & 0.57
        &  4,315 & 0.07 &   117 & 0.12 &  36.9 & 0.58
        & 26.8  & 28.6  \\
      R110-15
        & 17,392 & 0.11 &   189 & 0.13 &  92.0 & 0.85
        & 17,392 & 0.35 &   173 & 0.13 & 100.5 & 2.69
        & 61.6  & 69.4  \\
      R111-15
        & 53,407 & 0.28 &   156 & 0.10 & 342.4 & 2.80
        & 53,407 & 2.17 &   171 & 0.19 & 312.3 &11.42
        &177.4  &206.9 \\
      R112-15
        & 83,275 & 0.43 &   277 & 0.17 & 300.6 & 2.53
        & 83,275 & 7.32 &   249 & 0.19 & 334.4 &38.53
        &171.7  &199.0 \\
      \midrule
      \textbf{Geo.\ Mean}
        & \textbf{20,799.7} & \textbf{0.16}
        & \textbf{132.9}    & \textbf{0.09}
        & \textbf{156.5}    & \textbf{1.33}
        & \textbf{20,799.7} & \textbf{0.75}
        & \textbf{138.6}    & \textbf{0.13}
        & \textbf{150.1}    & \textbf{5.91}
        & \textbf{90.8}     & \textbf{100.7} \\
      \bottomrule
    \end{tabular}%
  \end{adjustbox}
  \caption{Comparison of DW and AF full MPs and optimal RMPs on R1 15-customer instances.}
  \label{t_combined_15}
\end{table}
\renewcommand{\arraystretch}{1}

\subsection{Benchmark of Dantzig-Wolfe and Arc-Flow} \label{ss_benchmarkDW_AF}

We next compare the performance of DW and AF on 25- and 50-customer instances, yielding a testbed of 112 instances with varying characteristics. We evaluate two different \emph{ng}-relaxations: A weaker relaxation with $\Delta = 6$ and a stronger one with $\Delta = 12$. Based on the discussion in~\S\ref{ss_connections}, we hypothesize that AF will perform better under weaker relaxations or clustered instances, where path recombination is more likely; otherwise, DW is expected to outperform AF.

Table~\ref{t_instanceFeatures} compares DW and AF across instance features and relaxation strengths, considering only instances solved within the time limit. For DW, Columns 2--5 present the lower bound, the time spent on the RMP, the time spent on the pricing problem, and the total CPU time. Columns 6--10 present the corresponding AF metrics, with Column 10 additionally reporting the ratio of DAG paths to priced routes as a measure of path recombination. Finally, Columns 11--13 report DW-to-AF per-instance ratios for the number of variables, pricing iterations, and total CPU time. Detailed results by instance group can be found in Appendix~\ref{A_results_noCuts}. 

\begin{table}[tb]
  \centering
  \begin{adjustbox}{max width=\textwidth}
  \begin{tabular}{l l
                  r r r r
                  r r r r r
                  r r r}
    \toprule
    & &
    \multicolumn{4}{c}{Dantzig--Wolfe} &
    \multicolumn{5}{c}{Arc--Flow} &
    \multicolumn{3}{c}{DW-to-AF Ratios} \\
    \cmidrule(lr){3-6}
    \cmidrule(lr){7-11}
    \cmidrule(lr){12-14}
    Feature & Value
      & LB & RMP (s) & PP (s) & Time (s)
      & LB & RMP (s) & PP (s) & Time (s) & Recomb.
      & Variables & Iterations & Time \\
    \midrule
    Customer  & 25
      & 300.95 & 0.07 & 1.06 & 1.21
      & 300.95 & 0.23 & 0.90 & 1.29 & 1.49
      & 0.26 & 1.21 & 0.94 \\
    number    & 50
      & 528.65 & 0.52 & 21.30 & 22.52
      & 528.65 & 2.61 & 14.43 & 18.22 & 2.27
      & 0.23 & 1.45 & 1.23 \\
    \midrule
    Customer  & R
      & 527.21 & 0.12 & 3.00 & 3.31
      & 527.21 & 0.57 & 3.02 & 3.94 & 1.11
      & 0.21 & 0.97 & 0.84 \\
    distribution & C
      & 269.57 & 0.33 & 6.37 & 7.02
      & 269.57 & 1.02 & 3.29 & 4.77 & 3.46
      & 0.29 & 1.96 & 1.47 \\
                & RC
      & 446.51 & 0.19 & 5.61 & 6.15
      & 446.51 & 0.79 & 4.74 & 6.08 & 1.62
      & 0.24 & 1.22 & 1.01 \\
    \midrule
    Time      & Tight
      & 417.39 & 0.09 & 1.49 & 1.69
      & 417.39 & 0.28 & 1.36 & 1.84 & 1.21
      & 0.25 & 1.11 & 0.91 \\
    windows   & Loose
      & 381.17 & 0.41 & 15.14 & 16.20
      & 381.17 & 2.10 & 9.60 & 12.81 & 2.81
      & 0.24 & 1.58 & 1.27 \\
    \midrule
    \midrule 
    Relaxation & $\Delta=6$
      & 398.87 & 0.19 & 4.75 & 5.23
      & 398.87 & 0.77 & 3.61 & 4.85 & 1.84
      & 0.25 & 1.32 & 1.08 \\
               & $\Delta=12$
      & 402.92 & 0.16 & 3.88 & 4.27
      & 402.92 & 0.70 & 3.57 & 4.81 & 1.37
      & 0.22 & 1.14 & 0.89 \\
    \bottomrule
  \end{tabular}
  \end{adjustbox}
  \caption{Comparison of DW and AF by instance features and relaxation strength (Base case: $\Delta = 6$)}
  \label{t_instanceFeatures}
\end{table}

For the 25-customer instances (first row), both methods take about 1.3 seconds, with DW being marginally faster overall. For the 50-customer instances (second row), AF is about 1.2 times faster than DW. The main difference is the distribution of computational effort. In the 50-customer cases, DW spends about half-a-second on the RMP and just over 21 seconds on pricing, whereas AF spends roughly 3 seconds on the RMP and 14 seconds on pricing. AF’s higher RMP overhead is due to introducing about four times more variables, whereas its lower pricing time is due to faster convergence. Through recombination, paths generate twice as many routes as those explicitly priced, and AF performs 30\% fewer iterations on average. 

With clustered customers and loose time windows, AF is about 1.4 and 1.3 times faster than DW, with recombination values of approximately 3.5 and 2.8, respectively. In these settings, high recombination gives AF a computational advantage. Conversely, with randomly distributed customers and tighter time windows, DW is $1.2$ and $1.1$ times faster, as the AF–RMP contains about four times more variables, but recombination is low, favoring DW’s compact column representation. Under the stronger relaxation ($\Delta = 12$), AF’s advantage diminishes as the state space expands and recombination decreases (from 1.8 to 1.4). Consequently, all DW-to-AF ratios decrease: variables from 0.25 to 0.22, iterations from 1.32 to 1.14, and time from 1.08 to 0.89.

To illustrate this behavior, Figure~\ref{f_BoundIterTime} shows bound progression over iterations and time for instances \texttt{R110-25}, where DW performs better, and \texttt{C208-25}, where AF does. In the former, both follow similar trends, with AF performing more iterations (33 vs. 23); since DW iterations are cheaper, it finishes sooner. In the latter, AF converges in far fewer iterations (75 vs. 173), while DW stalls between iterations 118–171. Here, AF’s higher per-iteration cost is offset by faster convergence.

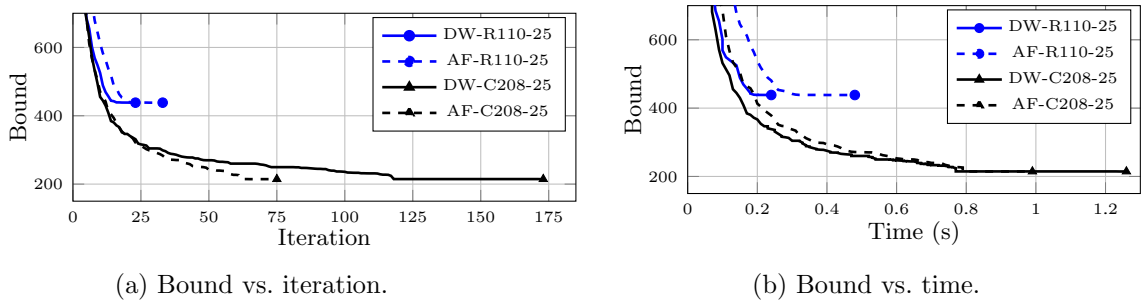
\begin{figure}[tb]
  \centering
  \begin{adjustbox}{max size={\textwidth}{0.20\textheight}, center}
    \begin{minipage}{0.90\textwidth}
      \centering

\begin{subfigure}[b]{0.45\textwidth}
  \pgfplotstableread[col sep=tab]{plot-bound.txt}\datatableIter
  \centering
  \begin{tikzpicture}
    \begin{axis}[
      scale only axis,
      width=\linewidth,
      height=2.5cm,
      font=\scriptsize,
      xlabel style={yshift=1ex},
      ylabel style={yshift=-1ex},
      tick label style={font=\tiny},
      legend style={
        font=\tiny,
        at={(0.98,0.98)},
        anchor=north east,
        legend image post style={scale=1.2},
      },
      xlabel={Iteration},
      ylabel={Bound},
      grid=major,
      ymin=150, ymax=700,
      xmin=0,   xmax=185,
      xtick distance=25,
      unbounded coords=discard,
    ]

      \addplot[blue,  solid,  line width=1pt, forget plot]
        table[x=iter,y=DWR110] \datatableIter;
      \addplot[blue,  dashed, line width=1pt, forget plot]
        table[x=iter,y=AFR110] \datatableIter;

      \addplot[black, solid,  line width=1pt, forget plot]
        table[x=iter,y=DWC208] \datatableIter;
      \addplot[black, dashed, line width=1pt, forget plot]
        table[x=iter,y=AFC208] \datatableIter;

      \addplot[only marks, mark=*,         mark size=1.8pt, color=blue,  forget plot]
        table[x=iter, y=DWR110-inc] \datatableIter;
      \addplot[only marks, mark=*,         mark size=1.8pt, color=blue,  forget plot]
        table[x=iter, y=AFR110-inc] \datatableIter;

      \addplot[only marks, mark=triangle*, mark size=1.8pt, color=black, forget plot]
        table[x=iter, y=AFC208-inc] \datatableIter;
      \addplot[only marks, mark=triangle*, mark size=1.8pt, color=black, forget plot]
        table[x=iter, y=DWC208-inc] \datatableIter; 

      \addlegendimage{blue,  solid,  line width=1pt, mark=*,         mark size=1.4pt}
      \addlegendentry{DW-R110-25}
    \addlegendimage{blue,  dashed, line width=1pt,
      mark=*, mark size=1.4pt, mark options={fill=blue, draw=blue}}
    \addlegendentry{AF-R110-25}

      \addlegendimage{black, solid,  line width=1pt, mark=triangle*, mark size=1.4pt}
      \addlegendentry{DW-C208-25}
     \addlegendimage{black, dashed, line width=1pt,
      mark=triangle*, mark size=1.4pt, mark options={fill=black, draw=black}}
    \addlegendentry{AF-C208-25}

    \end{axis}
  \end{tikzpicture}
  \subcaption{\small Bound vs.\ iteration.}
  \label{fig:bound-vs-iter}
\end{subfigure}
\hfill
\begin{subfigure}[b]{0.45\textwidth}
  \pgfplotstableread[col sep=tab]{plot-time.txt}\datatableTime
  \centering
  \begin{tikzpicture}
    \begin{axis}[
      scale only axis,
      width=0.9\linewidth,
      height=2.5cm,
      font=\scriptsize,
      xlabel style={yshift=1ex},
      ylabel style={yshift=-1ex},
      tick label style={font=\tiny},
      legend style={
        font=\tiny,
        at={(0.98,0.98)},
        anchor=north east
      },
      xlabel={Time (s)},
      ylabel={Bound},
      grid=major,
      ymin=150, ymax=700,
      xmin=0,   xmax=1.3,
    ]

      \addplot[blue,  solid,  line width=1pt, forget plot]
        table[x=DWR110_T, y=DWR110_B] \datatableTime;
      \addplot[blue,  dashed, line width=1pt, forget plot]
        table[x=AFR110_T, y=AFR110_B] \datatableTime;

      \addplot[black, solid,  line width=1pt, forget plot]
        table[x=DWC208_T, y=DWC208_B] \datatableTime;
      \addplot[black, dashed, line width=1pt, forget plot]
        table[x=AFC208_T, y=AFC208_B] \datatableTime;

\addplot[only marks, mark=*,         mark size=1.8pt, color=blue,  forget plot]
  table[x=DWR110_T, y=DWR110_T-inc] \datatableTime;
\addplot[only marks, mark=*,         mark size=1.8pt, color=blue,  forget plot]
  table[x=AFR110_T, y=AFR110_T-inc] \datatableTime;

\addplot[only marks, mark=triangle*, mark size=1.8pt, color=black, forget plot]
  table[x=DWC208_T, y=DWC208_T-inc] \datatableTime;
\addplot[only marks, mark=triangle*, mark size=1.8pt, color=black, forget plot]
  table[x=AFC208_T, y=AFC208_T-inc] \datatableTime;

      \addlegendimage{blue,  solid,  line width=1pt,
        mark=*,         mark size=1.4pt, mark options={fill=blue,  draw=blue}}
      \addlegendentry{DW-R110-25}

      \addlegendimage{blue,  dashed, line width=1pt,
        mark=*,         mark size=1.4pt, mark options={fill=blue,  draw=blue}}
      \addlegendentry{AF-R110-25}

      \addlegendimage{black, solid,  line width=1pt,
        mark=triangle*, mark size=1.4pt, mark options={fill=black, draw=black}}
      \addlegendentry{DW-C208-25}

      \addlegendimage{black, dashed, line width=1pt,
        mark=triangle*, mark size=1.4pt, mark options={fill=black, draw=black}}
      \addlegendentry{AF-C208-25}

    \end{axis}
  \end{tikzpicture}
  \subcaption{\small Bound vs.\ time.}
  \label{fig:bound-vs-time}
    \end{subfigure}

    \end{minipage}
  \end{adjustbox}

  \caption{Bound progression over iterations and time for \texttt{R110-25} and \texttt{C208-25} with $\Delta = 6$. }
  \label{f_BoundIterTime}
\end{figure}

\subsection{Benchmark of Dantzig-Wolfe and Arc-Flow with Cuts} \label{ss_benchmarkDW_AF_cuts}

The following set of experiments compares DW and AF when subset-row cuts (SRCs) of size three are added \citep{jepsen2008subset}. These rank-1 CG cuts~\eqref{eq_CGCut} use three nonzero multipliers, each equal to $1/2$. Separation is performed by enumeration and, since they increase subproblem complexity, we follow the rules used in \cite{yamin2025electric}: (i) add at most 30 cuts per iteration, prioritizing the most violated; (ii) allow each customer to appear in at most five triplets; and (iii) add cuts only if the strongest violation exceeds 0.1, with a total limit of 100 cuts. Prior DW research shows that these cuts effectively strengthen the LP bound, and by Proposition~\ref{p_projectionStrengthLambda}, AF should obtain the same improvement. However, as discussed in \S\ref{s_cutsLambda}, DW absorbs this complexity in the subproblem, while AF is also affected at the RMP, where the enlarged state space increases variables and constraints and reduces recombination. We therefore expect these cuts to impact AF more severely than DW.

Table~\ref{t_instanceFeaturesCuts} reports the same information as Table~\ref{t_instanceFeatures} with SRCs. Columns 4 and 9 additionally report the average number of cuts added to DW and AF. We report the arithmetic mean in these columns since some values are zero. Detailed results by instance group can be found in Appendix~\ref{A_results_cuts}. 

\begin{table}[b]
  \centering
  \begin{adjustbox}{max width=\textwidth}
  \begin{tabular}{l l
                  r r r r r
                  r r r r r r
                  r r r}
    \toprule
    & &
    \multicolumn{5}{c}{Dantzig--Wolfe} &
    \multicolumn{6}{c}{Arc--Flow} &
    \multicolumn{3}{c}{DW-to-AF Ratios} \\
    \cmidrule(lr){3-7}
    \cmidrule(lr){8-13}
    \cmidrule(lr){14-16}
     \multicolumn{1}{c}{Feature} &  \multicolumn{1}{c}{Value}
      &  \multicolumn{1}{c}{LB} &  \multicolumn{1}{c}{Cuts} &  \multicolumn{1}{c}{RMP (s)} &  \multicolumn{1}{c}{PP (s)} &  \multicolumn{1}{c}{Time (s)}
      &  \multicolumn{1}{c}{LB} &  \multicolumn{1}{c}{Cuts} &  \multicolumn{1}{c}{RMP (s)} &  \multicolumn{1}{c}{PP (s)} &  \multicolumn{1}{c}{Time (s)} &  \multicolumn{1}{c}{Recomb.}
      &  \multicolumn{1}{c}{Variables} &  \multicolumn{1}{c}{Iterations} &  \multicolumn{1}{c}{Time} \\
    \midrule
    Customer  & 25
      & 305.76 & 10.1 & 0.08 & 1.41 & 1.59
      & 305.76 & 10.6 & 0.35 & 1.33 & 1.90 & 1.43
      & 0.26 & 1.22 & 0.84 \\
    number    & 50
      & 549.97 & 38.3 & 0.85 & 35.23 & 37.36
      & 550.01 & 38.0 & 6.84 & 27.71 & 37.32 & 2.13
      & 0.23 & 1.34 & 1.00 \\
    \midrule
    Customer  & R
      & 540.09 & 37.0 & 0.19 & 6.03 & 6.68
      & 540.09 & 38.6 & 2.04 & 8.50 & 11.78 & 1.06
      & 0.20 & 0.89 & 0.57 \\
    distribution & C
      & 269.62 & 0.3 & 0.34 & 6.65 & 7.29
      & 269.62 & 0.3 & 0.99 & 3.56 & 5.00 & 3.36
      & 0.29 & 1.96 & 1.46 \\
                & RC
      & 473.56 & 35.3 & 0.27 & 8.70 & 9.44
      & 473.61 & 34.0 & 1.85 & 7.36 & 10.12 & 1.49
      & 0.25 & 1.20 & 0.93 \\
    \midrule
    Time      & Tight
      & 424.94 & 25.7 & 0.14 & 2.57 & 2.91
      & 424.97 & 24.8 & 0.61 & 2.41 & 3.40 & 1.20
      & 0.25 & 1.08 & 0.86 \\
    windows   & Loose
      & 395.73 & 22.7 & 0.49 & 19.28 & 20.50
      & 395.73 & 23.8 & 3.95 & 15.26 & 20.86 & 2.54
      & 0.24 & 1.52 & 0.98 \\
    \midrule
    \midrule 
    Relaxation & $\Delta=6$
      & 410.08 & 24.2 & 0.26 & 7.04 & 7.72
      & 410.09 & 24.3 & 1.55 & 6.07 & 8.42 & 1.74
      & 0.24 & 1.28 & 0.92 \\
               & $\Delta=12$
      & 408.25 & 19.1 & 0.22 & 5.98 & 6.53
      & 408.26 & 18.8 & 1.26 & 5.66 & 7.79 & 1.36
      & 0.21 & 1.13 & 0.84 \\
    \bottomrule
  \end{tabular}
  \end{adjustbox}
  \caption{Comparison of DW and AF with cuts by instance features and relaxation strength (Base case: $\Delta = 6$)}
  \label{t_instanceFeaturesCuts}
\end{table}

DW and AF do not always reach the same bound due to (i) solution symmetry, which can lead to different cuts being separated, and (ii) the cut limit being reached. When cuts are fully separated, both yield the same bound, consistent with Proposition~\ref{p_projectionStrengthLambda}. Compared to Table~\ref{t_instanceFeatures}, AF still performs best on clustered instances, where few or no cuts are separated, while DW performs better elsewhere, including cases with loose time windows where AF was previously superior. For $\Delta = 6$, the iteration ratio decreases from 1.32 to 1.28, the time ratio from 1.08 to 0.92, and AF recombination from 1.84 to 1.74. Results for $\Delta = 12$ confirm the trend: stronger relaxations combined with SRCs further weaken AF’s relative performance, with DW being approximately $1.2$ times faster overall.

\subsection{Assessment of Dantzig-Wolfe and Arc-Flow with Iterative Subproblem Strengthening} \label{s_computational_subproblem}

In this set of experiments, we assess DW and AF under decremental state-space relaxation (DSSR) and column elimination (CE). In DSSR, we start with $\Delta = 0$, solve the relaxation, and progressively increase the \emph{ng}-set size whenever cycles appear, until the elementary bound is reached. In CE, we also start with $\Delta = 0$, solve the relaxation, and iteratively remove \emph{conflicting paths} (one cycle per path) from the DAG until elementarity is achieved. Both methods rely on an explicit DAG and a standard shortest-path algorithm, enabling the analysis of subproblem size and avoiding reliance on the dominance rule \eqref{eq_dominance}, which is not readily applicable to CE. Although more efficient implementations exist, our purpose is to show that both DW and AF can employ these methods to reach the same bound, each exhibiting different trade-offs. We consider $\Delta = 0$ and only type 1 instances (i.e., with tight time windows), as the explicit DAG for several type 2 instances did not fit in memory. For reference, Appendix~\ref{A_explicit_vs_implicit} compares the explicit and implicit (labeling) approaches.

Table~\ref{t_DSSR_CE} compares DW and AF with DSSR and CE by instance features. For DW with DSSR, Columns 3--5 report the average number of strengthening iterations, routes eliminated from the RMP due to strengthenings, and CPU time, excluding the compilation time of the initial DAG. Columns 6--8 present the same information for AF with DSSR, with Column 9 additionally reporting the overall change in DAG dimension ($=|\E| + |\N|$), measured as the difference between the final and initial DAGs. This metric is omitted for DW, as it closely follows that of AF. Columns 10--16 report the same information for DW and AF under CE. For the number of strengthenings, eliminated routes, and the change in DAG dimension, we report the arithmetic mean (rather than the geometric mean) due to the presence of negative or zero values (e.g., DSSR can decrease the DAG dimension in specific cases due to forbidden extensions and resource constraints). The results include only instances solved by all approaches. Appendix~\ref{A_iterativeSubproblem} provides detailed results by instance group, showing that under DSSR seven instances were not solved, and under CE four, either because the time limit was exceeded or memory was exhausted. 

\begin{table}[tb]
  \centering
  \begin{adjustbox}{max width=\textwidth}
  \begin{tabular}{l l
                  r r r
                  r r r r
                  r r r
                  r r r r}
    \toprule
    & &
    \multicolumn{3}{c}{DW-DSSR} &
    \multicolumn{4}{c}{AF-DSSR} &
    \multicolumn{3}{c}{DW-CE} &
    \multicolumn{4}{c}{AF-CE} \\
    \cmidrule(lr){3-5}
    \cmidrule(lr){6-9}
    \cmidrule(lr){10-12}
    \cmidrule(lr){13-16}
        \multicolumn{1}{c}{Feature} &  \multicolumn{1}{c}{Value}
      &  \multicolumn{1}{c}{\shortstack{Strength.\\Iter.}} &  \multicolumn{1}{c}{\shortstack{Elim.\\Routes}} &  \multicolumn{1}{c}{Time (s)}
      &  \multicolumn{1}{c}{\shortstack{Strength.\\Iter.}} &  \multicolumn{1}{c}{\shortstack{Elim.\\Arcs}} &  \multicolumn{1}{c}{Time (s)} &  \multicolumn{1}{c}{\shortstack{Change\\in DAG}}
      &  \multicolumn{1}{c}{\shortstack{Strength.\\Iter.}} &  \multicolumn{1}{c}{\shortstack{Elim.\\Routes}} &  \multicolumn{1}{c}{Time (s)}
      &  \multicolumn{1}{c}{\shortstack{Strength.\\Iter.}} &  \multicolumn{1}{c}{\shortstack{Elim.\\Arcs}} &  \multicolumn{1}{c}{Time (s)} &  \multicolumn{1}{c}{\shortstack{Change\\in DAG}} \\

    \midrule
    Customer  & 25
      & 2.2 & 369.2 & 8.78
      & 2.2 & 1,310.7 & 9.25 & 1,221,596.3
      & 10.8 & 231.5 & 5.90
      & 10.2 & 51.5 & 6.48 & 982.7 \\
    number    & 50
      & 3.5 & 1,046.9 & 61.42
      & 3.2 & 4,622.8 & 68.11 & 4,944,384.5
      & 169.1 & 857.3 & 44.97
      & 171.1 & 431.2 & 45.20 & 9,696.1 \\
    \midrule
    Customer  & R
      & 2.8 & 335.8 & 11.48
      & 2.5 & 1,099.4 & 12.61 & 1,003,720.2
      & 14.6 & 236.2 & 6.57
      & 14.1 & 81.3 & 8.25 & 1,511.4 \\
    distribution & C
      & 0.5 & 321.8 & 25.88
      & 0.5 & 1,178.4 & 25.80 & 4,199,578.6
      & 3.1 & 71.5 & 20.77
      & 2.9 & 8.1 & 18.71 & 252.9 \\
                & RC
      & 5.0 & 1,388.9 & 29.85
      & 5.0 & 6,243.1 & 34.18 & 3,421,434.7
      & 232.3 & 1,256.8 & 22.41
      & 234.9 & 586.3 & 23.13 & 13,144.2 \\
    \midrule
    Time & \multirow{2}{*}{Tight}
      & \multirow{2}{*}{2.7} & \multirow{2}{*}{661.5} & \multirow{2}{*}{20.32}
      & \multirow{2}{*}{2.6} & \multirow{2}{*}{2,739.5} & \multirow{2}{*}{21.88} & \multirow{2}{*}{2,827,505.0}
      & \multirow{2}{*}{79.1} & \multirow{2}{*}{501.5} & \multirow{2}{*}{14.17}
      & \multirow{2}{*}{79.6} & \multirow{2}{*}{215.3} & \multirow{2}{*}{14.98} & \multirow{2}{*}{4,741.4} \\
    windows & & & & & & & & & & & & & & & \\
    \bottomrule
    \end{tabular}
    \end{adjustbox}
    \caption{Comparison of DW and AF with DSSR and CE by instance features.} 
    \label{t_DSSR_CE}
\end{table}

From Table~\ref{t_DSSR_CE}, DW and AF exhibit similar performance under both DSSR and CE. The number of strengthenings differs slightly because DW and AF may yield symmetric solutions. For both reformulations, DSSR requires fewer than three strengthenings on average and eliminates over 650 routes from the DW--RMP and 2,700 arcs from the AF--RMP, whereas CE requires about 80 iterations and removes roughly 500 routes from the DW--RMP and 215 arcs from the AF--RMP. Thus, DSSR involves fewer strengthenings but produces a state space about three orders of magnitude larger than that of CE, which maintains a smaller, easier-to-optimize DAG. This is reflected in the running times: DW-DSSR and AF-DSSR take about 21 seconds on average, compared to roughly 14 seconds for DW-CE and AF-CE. Note, however, that the explicit DAG approach affects DSSR more severely due to its global strengthening step, which requires recompiling the DAG at each iteration (explaining why it is typically used with an implicit approach), whereas CE involves only local DAG refinements.

Although general conclusions are difficult to draw, as results are sensitive to implementation choices, CE’s weaker strengthenings can be advantageous in memory-constrained settings (e.g., large instances), where DSSR may quickly exhaust memory, as shown in Appendix~\ref{A_iterativeSubproblem}. Conversely, when the gap between the target bound and the initial relaxation is large, CE’s local refinements may require substantially more iterations, whereas DSSR often converges in only a few, as illustrated in Figure~\ref{fig_combined_R112_C103_small}. In any case, both DW and AF can employ the same strengthening methods.

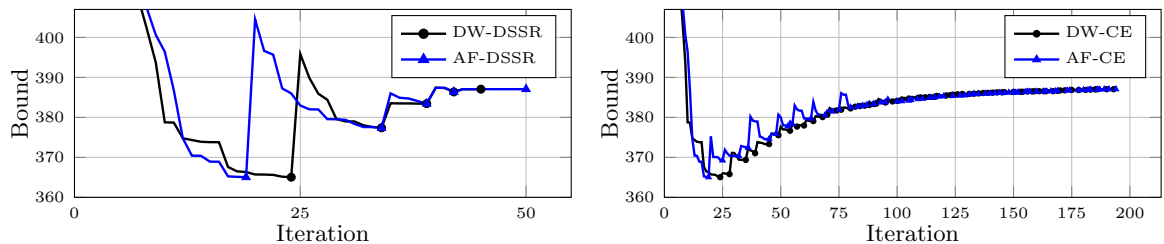
\begin{figure}[tb]
  \centering
  \begin{adjustbox}{max width=\textwidth, max totalheight=0.48\textheight, center}
      \centering

\begin{subfigure}[b]{0.40\textwidth}
  \pgfplotstableread[col sep=tab]{four-series-data-R112.txt}\datatableR
  \centering
  \begin{tikzpicture}
    \begin{axis}[
      scale only axis,
      width=\linewidth,
      height=2.5cm,
      font=\scriptsize,
      xlabel style={yshift=1ex},
      ylabel style={yshift=-1ex},
      tick label style={font=\tiny},
      legend style={font=\tiny, at={(0.98,0.02)}, anchor=south east},
      legend cell align=left,
      xlabel={Iteration},
      ylabel={Bound},
      legend pos=north east,
      grid=major,
      scaled y ticks=false,
      ymin=360, ymax=407,
      xmin=0,
      xtick distance=25,
      ytick distance=10,
    ]
\addplot[
  black, solid, mark=none, line width=0.9pt,
  legend image post style={mark=*, mark size=1.5pt}
] table[x=iter, y={DW-DSSR}] \datatableR;
\addplot[only marks, mark=*, mark size=1.5pt, color=black, forget plot]
  table[x=iter, y={DW-DSSR-inc}] \datatableR;
\addlegendentry{DW-DSSR}

\addplot[
  blue, solid, mark=none, line width=0.9pt,
  legend image post style={mark=triangle*, mark size=1.8pt}
] table[x=iter, y={AF-DSSR}] \datatableR;
\addplot[only marks, mark=triangle*, mark size=1.8pt, color=blue, forget plot]
  table[x=iter, y={AF-DSSR-inc}] \datatableR;
\addlegendentry{AF-DSSR}
    \end{axis}
  \end{tikzpicture}
\end{subfigure}%
\hspace{0.075\textwidth}%
\begin{subfigure}[b]{0.40\textwidth}
  \pgfplotstableread[col sep=tab]{four-series-data-R112.txt}\datatableRce
  \centering
  \begin{tikzpicture}
    \begin{axis}[
      scale only axis,
      width=\linewidth,
      height=2.5cm,
      font=\scriptsize,
      xlabel style={yshift=1ex},
      ylabel style={yshift=-1ex},
      tick label style={font=\tiny},
      legend style={font=\tiny, at={(0.98,0.02)}, anchor=south east},
      legend cell align=left,
      legend pos=north east,
      xlabel={Iteration},
      ylabel={Bound},
      grid=major,
      scaled y ticks=false,
      ymin=360, ymax=407,
      xmin=0,
      xtick distance=25,
      ytick distance=10,
    ]

      \addplot[
        black, solid, mark=none, line width=0.9pt,
        legend image post style={mark=*, mark size=1pt}
      ] table[x=iter, y={DW-CE}] \datatableRce;
      \addlegendentry{DW-CE}
      \addplot[only marks, mark=*, mark size=1pt, color=black, forget plot]
        table[x=iter, y={DW-CE-inc}] \datatableRce;

      \addplot[
        blue, solid, mark=none, line width=0.9pt,
        legend image post style={mark=triangle*, mark size=1.2pt}
      ] table[x=iter, y={AF-CE}] \datatableRce;
      \addlegendentry{AF-CE}
      \addplot[only marks, mark=triangle*, mark size=1.2pt, color=blue, forget plot]
        table[x=iter, y={AF-CE-inc}] \datatableRce;

    \end{axis}
    \end{tikzpicture}
    \end{subfigure}
    \end{adjustbox}
    \caption{Bound progression over iterations for DW and AF with DSSR and CE on \texttt{R112-25}.}
    \label{fig_combined_R112_C103_small}
\end{figure}

\section{Conclusion} \label{s_conclusions}

This paper presented a unified framework for comparing Dantzig-Wolfe and Arc-Flow reformulations of integer programs. We formalized several relationships that had previously appeared in the literature, either as formal claims or (informal) observations, establishing precise correspondences between variables, feasible sets, and relaxations in both formulations.  We showed that Dantzig-Wolfe and Arc-Flow models yield identical LP bounds, and that valid inequalities, whether defined in the original, DW, or AF spaces, preserve their strength when appropriately translated. We also demonstrated that iterative subproblem-strengthening methods--such as decremental state-space relaxation and column elimination--can be viewed as sequences of subproblem-level cutting planes, clarifying their theoretical foundation.

Computational experiments on the vehicle routing problem with time windows provide empirical support for the theoretical insights developed in this work. In this case study, both reformulations achieve the same bounds, but their computational behavior differs in a structured way. AF tends to perform better when subproblems are highly relaxed or low-dimensional, where path recombination can accelerate convergence. In contrast, DW shows advantages when the master problem remains compact, benefiting from smaller LPs and faster iterations.

Although these empirical results are specific to the VRPTW instances considered, they illustrate how structural properties (e.g., relaxation strength and state-space size) may influence the relative performance of the two reformulations. Overall, the theoretical results provide a formal basis for comparing these reformulations, while the computational study offers problem-specific guidance for selecting between them, potentially extending to closely related decomposition-based integer optimization settings. Further computational experiments across different problem classes are warranted to assess the broader applicability of these findings.


\appendix

\section{Block-Diagonal Structure} \label{s_blockDiagonal}
In many applications,~\eqref{eq_IP} has the following structure: \[
\max\left\{\vb c \vb x \,: \, \vb A \vb x \leq \vb b, \, \vb D^{k} \vb x^{k} \leq \vb d^{k} \, \, \, k=1,\ldots K, \, \vb x \in \Z_{+}^{n} \right\},
\]
and the set $X$ decomposes into $K$ independent subproblems $X^{k} := \{\vb x \in \Z_{+}^{n_{k}} \, : \, \vb D^{k} \vb x \leq \vb d^{k}\}$. By reformulating each $X^{k}$ as in~\eqref{eq_DWlink}, the DW reformulation follows:
\[
\max \left\{ \sum_{k=1}^{K} \sum_{q \in Q^{k}} ( \vb c^{k} \vb x^{q,k}) \lambda_{q}^{k} \;:\; 
\sum_{k=1}^{K} \sum_{q \in Q^{k}} ( \vb A^{k} \vb x^{q,k}) \lambda_{q}^{k} \leq \vb b,\;
\sum_{q \in Q^{k}} \lambda_{q}^{k} = 1 \;\; \forall k,\;
\boldsymbol{\lambda}^{k} \in \{0,1\}^{|Q^{k}|} \;\; \forall k
\right\},
\]

\noindent where $\vb c^{k}$ and $\vb A^{k}$ are the corresponding blocks in $\vb c$ and $\vb A$, and $\{\vb x^{q,k}\}_{q \in Q^{k}}$ are the integer points  in $X^{k}$. When blocks are identical, the variables $\lambda_{q}^{k}$ across different $k$ correspond to the same integer point $q \in Q$ and can be \emph{aggregated} into a single variable $\lambda_{q} := \sum_{k=1}^{K} \lambda_{q}^{k}$, yielding an aggregated convexity constraint: $\sum_{q \in Q}\lambda_{q} = K$. 
For AF, a DAG $\D^{k}=(\N^{k}, \E^{k})$ satisfying Definition~\ref{d_arcflowRefor}(i)-(ii) with respect to $X^{k}$ must be defined for each block $k=1,\ldots, K$. The AF reformulation follows:
\[
\max \; \left\{ \sum_{k=1}^{K} (\vb c^{k}\vb T^{k}) \vb y^{k} \; : \; 
   \sum_{k=1}^{K} (\vb A^{k}\vb T^{k}) \vb y^{k} \leq \vb b, 
   \; \vb F^{k} \vb y^{k} = \vb f^{k} \;\; \forall k,\;
   \vb y^{k} \in \{0,1\}^{|\E^{k}|} \;\; \forall k \right\},
\]
\noindent where $\vb F^{k}\vb y^{k} = \vb f^{k}$ denotes the flow constraints \eqref{eq_AF2}--\eqref{eq_AF3} on $\D^{k}$. In the case of identical blocks, we have $K$ identical DAGs that can be aggregated into a single one $\D = (\N, \E)$. Accordingly, for each arc $e \in \E$, the variables $y_{e}^{k}$ aggregate into a single variable $y_{e} := \sum_{k=1}^{K} y_e^k$, leading to aggregated flow constraints: $  \sum_{e \in \delta^{+}(r)} y_{e} = K$ and $ \sum_{e \in \delta^{+}(u)}y_{e} - \sum_{e \in \delta^{-}(u)} y_{e} = 0$ for all $u \in \N\setminus\{r,t\}$.

\section{Additional Computations for the Illustrative Example}

\subsection{Projection Matrix, Cost, and Constraint Coefficients for the VRPTW Arc-Flow Reformulation} \label{A_example_computation}

Table \ref{t_computations} shows the projection matrix $\vb T$ and the computations $(\vb c \vb T)$ and $(\vb A \vb T)$ for the illustrative example in Figure \ref{f_example}. 

\begin{table}[htbp]
  \centering
  \begin{adjustbox}{max width=\textwidth}
  \begin{tabular}{r l *{21}{c}}
    \toprule
    & & 
    $y_{1}$ & $y_{2}$ & $y_{3}$ & $y_{4}$ & $y_{5}$ & $y_{6}$ & $y_{7}$ & $y_{8}$ & $y_{9}$ 
    & $y_{10}$ & $y_{11}$ & $y_{12}$ & $y_{13}$ & $y_{14}$ & $y_{15}$ & $y_{16}$ & $y_{17}$ & $y_{18}$ & $y_{19}$ & $y_{20}$ & $y_{21}$  \\
    \midrule
    \multirow{11}{*}{$\vb T:= \;\;\;\;$}
      & $x_{\underline{0},\overline{0}}$ & 0 & 0 & 0 & 0 & 0 & 0 & 0 & 0 & 0 & 0 & 0 & 1 & 0 & 0 & 0 & 0 & 0 & 0 & 0 & 0 & 0\\
      & $x_{\underline{0},1}$            & 1 & 0 & 0 & 0 & 0 & 0 & 0 & 0 & 0 & 0 & 0 & 0 & 0 & 0 & 0 & 0 & 0 & 0 & 0 & 0 & 0\\
      & $x_{\underline{0},2}$            & 0 & 1 & 0 & 0 & 0 & 0 & 0 & 0 & 0 & 0 & 0 & 0 & 0 & 0 & 0 & 0 & 0 & 0 & 0 & 0 & 0\\
      & $x_{\underline{0},3}$            & 0 & 0 & 1 & 0 & 0 & 0 & 0 & 0 & 0 & 0 & 0 & 0 & 0 & 0 & 0 & 0 & 0 & 0 & 0 & 0 & 0\\
      & $x_{1,2}$                        & 0 & 0 & 0 & 1 & 0 & 0 & 0 & 1 & 0 & 0 & 0 & 0 & 0 & 0 & 0 & 0 & 0 & 0 & 0 & 0 & 0\\
      & $x_{2,1}$                        & 0 & 0 & 0 & 0 & 1 & 0 & 0 & 0 & 1 & 0 & 0 & 0 & 0 & 0 & 0 & 0 & 0 & 0 & 0 & 0 & 0\\
      & $x_{2,3}$                        & 0 & 0 & 0 & 0 & 0 & 1 & 0 & 0 & 0 & 1 & 0 & 0 & 0 & 0 & 0 & 0 & 0 & 0 & 0 & 0 & 0\\
      & $x_{3,2}$                        & 0 & 0 & 0 & 0 & 0 & 0 & 1 & 0 & 0 & 0 & 1 & 0 & 0 & 0 & 0 & 0 & 0 & 0 & 0 & 0 & 0\\
      & $x_{1,\overline{0}}$             & 0 & 0 & 0 & 0 & 0 & 0 & 0 & 0 & 0 & 0 & 0 & 0 & 1 & 0 & 0 & 1 & 0 & 0 & 1 & 0 & 0\\
      & $x_{2,\overline{0}}$             & 0 & 0 & 0 & 0 & 0 & 0 & 0 & 0 & 0 & 0 & 0 & 0 & 0 & 1 & 0 & 0 & 1 & 0 & 0 & 1 & 0\\
      & $x_{3,\overline{0}}$             & 0 & 0 & 0 & 0 & 0 & 0 & 0 & 0 & 0 & 0 & 0 & 0 & 0 & 0 & 1 & 0 & 0 & 1 & 0 & 0 & 1\\
    \midrule
    $(\vb c \vb T)=$  
      & $(c_{e})_{e \in \E}$             & 5 & 10 & 5 & 10 & 10 & 10 & 10 & 10 & 10 
                                          & 10 & 10 & 0 & 5 & 10 & 5 & 5 & 10 & 5 & 5 & 10 & 5 \\
    \midrule
    \multirow{3}{*}{$(\vb A\vb T)=$}
      &                                  & 0 & 0 & 0 & 1 & 0 & 0 & 0 & 1 & 0 & 0 & 0 & 0 & 1 & 0 & 0 & 1 & 0 & 0 & 1 & 0 & 0 \\
      & $(z_{i})_{i\in N} =$             & 0 & 0 & 0 & 0 & 1 & 1 & 0 & 0 & 1 & 1 & 0 & 0 & 0 & 1 & 0 & 0 & 1 & 0 & 0 & 1 & 0 \\
      &                                  & 0 & 0 & 0 & 0 & 0 & 0 & 1 & 0 & 0 & 0 & 1 & 0 & 0 & 0 & 1 & 0 & 0 & 1 & 0 & 0 & 1 \\
    \bottomrule
    \end{tabular}
    \end{adjustbox}
    \caption{Projection matrix $\vb T$ and computations $(\vb c\vb T)$ and $(\vb A\vb T)$ for the illustrative example.}
    \label{t_computations}
\end{table}

\subsection{CG-Rank 1 Cut in the DW and AF Spaces} \label{A_CGRank1Cut}

Figure~\ref{f_lambdaCut} shows a CG rank-1 cut with multipliers $\vb u = (1/2, 1/2, 0)$ in the DW and AF spaces in our illustrative example (Figure \ref{f_example}). These cuts are: $\sum_{k=1}^{K} (\lambda_{4}^{k} + \lambda_{5}^{k} + \lambda_{8}^{k} + \lambda_{9}^{k} + \lambda_{10}^{k}+\lambda_{11}^{k} + \lambda_{12}^{k}) \leq 1$  and $\sum_{k=1}^{K} (y^{k}_{4} + y^{k}_{5} + y^{k}_{10} + y^{k}_{13}) \leq 1$. 

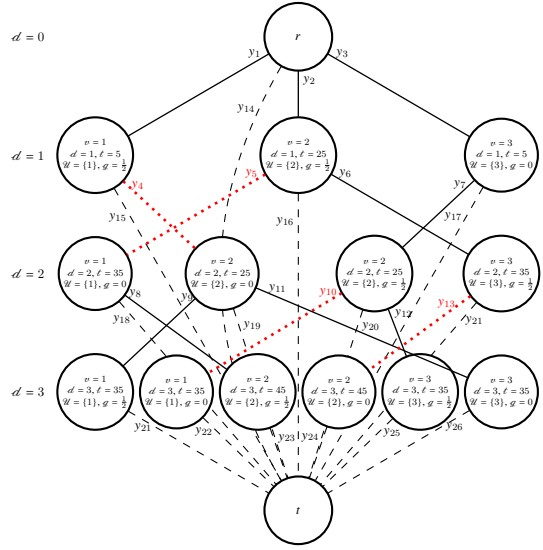
\begin{figure}[ht]
  \centering

  \begin{subfigure}[b]{0.40\textwidth}
   \centering
    \scriptsize
    \renewcommand{\arraystretch}{1.1}
    \begin{tabular}{c l c}
      \toprule
      $q$ & Route & $\gamma_{q}=\left\lfloor \tfrac12\,(z_{1}^{q}+z_{2}^{q}) \right\rfloor$ \\
      \midrule
       1  & $0\to1\to0$            & 0 \\
       2  & $0\to2\to0$            & 0 \\
       3  & $0\to3\to0$            & 0 \\
      \midrule
      \rowcolor{red!20}\textbf{4}  & \textbf{$0\to1\to2\to0$}        & \textbf{1} \\
      \rowcolor{red!20}\textbf{5}  & \textbf{$0\to2\to1\to0$}        & \textbf{1} \\
       6  & $0\to2\to3\to0$        & 0 \\
       7  & $0\to3\to2\to0$        & 0 \\
      \midrule
      \rowcolor{red!20}\textbf{8}  & \textbf{$0\to1\to2\to1\to0$}    & \textbf{1} \\
      \rowcolor{red!20}\textbf{9}  & \textbf{$0\to1\to2\to3\to0$}    & \textbf{1} \\
      \rowcolor{red!20}\textbf{10} & \textbf{$0\to2\to1\to2\to0$}    & \textbf{1} \\
      \rowcolor{red!20} \textbf{11} & \textbf{$0\to2\to3\to2\to0$}     & \textbf{1} \\
      \rowcolor{red!20}\textbf{12} & \textbf{$0\to3\to2\to1\to0$}    & \textbf{1} \\
      13 & $0\to3\to2\to3\to0$     & 0 \\
      \bottomrule
    \end{tabular}
    \caption{\small Cut in the DW space.}
\end{subfigure}%
  \hfill
  \begin{subfigure}[b]{0.60\textwidth}
    \centering
    \begin{tikzpicture}[
        scale=0.45,
        every node/.style={transform shape},
        >=Stealth,
        thick,
        node distance=2.2cm and 3.0cm,
        on grid, auto,
        myNode/.style={
          circle, draw, fill=white,
          minimum size=20mm, inner sep=1pt,
          font=\scriptsize, align=center
        }
      ]
      \pgfdeclarelayer{background}
      \pgfsetlayers{background,main}

      \node (layer0) at (-8,   0) {$\mathscr{d}=0$};
      \node (layer1) at (-8, -3.5) {$\mathscr{d}=1$};
      \node (layer2) at (-8, -7.0) {$\mathscr{d}=2$};
      \node (layer3) at (-8,-10.5) {$\mathscr{d}=3$};

      \node[myNode] (R) at (0,0) {\normalsize $r$};

      \node[myNode] (A) at (-6,-3.5) {
        $\mathscr{v}=1$\\$\mathscr{d}=1,\mathscr{t}=5$\\$\mathscr{U}=\{1\},\mathscr{g}=\frac{1}{2}$
      };
      \node[myNode] (B) at (0,-3.5) {
        $\mathscr{v}=2$\\$\mathscr{d}=1,\mathscr{t}=25$ \\$\mathscr{U}=\{2\}, \mathscr{g}=\frac{1}{2}$
      };
      \node[myNode] (C) at (6,-3.5) {
        $\mathscr{v}=3$\\$\mathscr{d}=1,\mathscr{t}=5$ \\$\mathscr{U}=\{3\}, \mathscr{g}=0$
      };

      \node[myNode] (M) at (-2.25,-7.0) {
        $\mathscr{v}=2$\\$\mathscr{d}=2,\mathscr{t}=25$ \\ $\mathscr{U}=\{2\}, \mathscr{g}=0$
      };

      \node[myNode] (Mnew) at (2.25,-7.0) {
        $\mathscr{v}=2$\\$\mathscr{d}=2,\mathscr{t}=25$ \\ $\mathscr{U}=\{2\}, \mathscr{g}=\frac{1}{2}$
      };
      
      \node[myNode] (F) at (-6,-7.0) {
        $\mathscr{v}=1$\\$\mathscr{d}=2,\mathscr{t}=35$ \\$\mathscr{U}=\{1\}, \mathscr{g}=0$
      };
      \node[myNode] (G) at (6,-7.0) {
        $\mathscr{v}=3$ \\$\mathscr{d}=2,\mathscr{t}=35$ \\ $\mathscr{U}=\{3\}, \mathscr{g}=\frac{1}{2}$
      };

      \node[myNode] (Z) at (-6,-10.5) {
        $\mathscr{v}=1$ \\$\mathscr{d}=3,\mathscr{t}=35$ \\ $\mathscr{U}=\{1\}, \mathscr{g}=\frac{1}{2}$
      };

     \node[myNode] (W) at (-3.6,-10.5) {
        $\mathscr{v}=1$ \\$\mathscr{d}=3,\mathscr{t}=35$ \\ $\mathscr{U}=\{1\}, \mathscr{g}=0$
      };

      \node[myNode] (X) at (-1.2,-10.5) {
        $\mathscr{v}=2$ \\$\mathscr{d}=3,\mathscr{t}=45$ \\ $\mathscr{U}=\{2\}, \mathscr{g}=\frac{1}{2}$
      };

      \node[myNode] (Y) at (1.2,-10.5) {
        $\mathscr{v}=2$ \\$\mathscr{d}=3,\mathscr{t}=45$ \\ $\mathscr{U}=\{2\}, \mathscr{g}=0$
      };

        \node[myNode] (S) at (3.6,-10.5) {
        $\mathscr{v}=3$\\$\mathscr{d}=3,\mathscr{t}=35$ \\ $\mathscr{U}=\{3\}, \mathscr{g}=\frac{1}{2}$
      };
      
      \node[myNode] (U) at (6,-10.5) {
        $\mathscr{v}=3$\\$\mathscr{d}=3,\mathscr{t}=35$ \\ $\mathscr{U}=\{3\}, \mathscr{g}=0$
      };

      \node[myNode] (T) at (0,-14) {\normalsize $t$};

      \draw[line width=0.5pt] (R) -- node[pos=0.1,above] {$y_{1}$} (A);
      \draw[line width=0.5pt] (R) -- node[pos=0.2,right] {$y_{2}$} (B);
      \draw[line width=0.5pt] (R) -- node[pos=0.1,above] {$y_{3}$} (C);
      \draw[red, line width=1pt, dotted] (A) -- node[pos=0.2,above] {$y_{4}$} (M);
      \draw[red, line width=1pt, dotted] (B) -- node[pos=0.1,above] {$y_{5}$} (F);
      \draw[line width=0.5pt] (B) -- node[pos=0.1,above] {$y_{6}$} (G);
      \draw[line width=0.5pt] (C) -- node[pos=0.2,above] {$y_{7}$} (Mnew);
      \draw[line width=0.5pt] (F) -- node[pos=0.1,above]  {$y_{8}$} (X);
      \draw[line width=0.5pt] (M) -- node[pos=0.11,above]  {$y_{9}$} (Z);
      \draw[red, line width=1pt, dotted] (Mnew) -- node[pos=0.1,above]  {$y_{10}$} (W);
      \draw[line width=0.5pt] (M) -- node[pos=0.1,above]  {$y_{11}$} (U);
      \draw[line width=0.5pt] (Mnew) -- node[pos=0.1,right]  {$y_{12}$} (S);
      \draw[red, line width=1pt, dotted] (G) -- node[pos=0.1,left]  {$y_{13}$} (Y);

      \begin{pgfonlayer}{background}
        \draw[dashed, bend right=30] (R)  to node[pos=0.1,left]  {$y_{14}$} (T);
        \draw[dashed] (A)  to node[pos=0.1,left]  {$y_{15}$} (T);
        \draw[dashed] (B)  to node[pos=0.1,left]  {$y_{16}$} (T);
        \draw[dashed] (C)  to node[pos=0.1,left]  {$y_{17}$} (T);

        \draw[dashed] (M)  to node[pos=0.1,right]  {$y_{19}$} (T);
        \draw[dashed] (Mnew)  to node[pos=0.1,right]  {$y_{20}$} (T);
        \draw[dashed] (F)  to node[pos=0.1,left]  {$y_{18}$} (T);
        \draw[dashed] (G)  to node[pos=0.1,right]  {$y_{21}$} (T);

         \draw[dashed] (Z)  to node[pos=0.2,left]  {$y_{21}$} (T);
         \draw[dashed] (W)  to node[pos=0.2,left]  {$y_{22}$} (T);
         \draw[dashed] (X)  to node[pos=0.2,right]  {$y_{23}$} (T);
         \draw[dashed] (Y)  to node[pos=0.2,left]  {$y_{24}$} (T);
         \draw[dashed] (S)  to node[pos=0.2,right]  {$y_{25}$} (T);
         \draw[dashed] (U)  to node[pos=0.2,right]  {$y_{26}$} (T);
      \end{pgfonlayer}

    \end{tikzpicture}
    \caption{\small Cut in the AF space.}
  \end{subfigure}

  \caption{CG rank-1 cut with $\vb u = (\frac{1}{2}, \frac{1}{2}, 0)$. Variables in the cut are shown in bold red and dotted lines.}
  \label{f_lambdaCut}
\end{figure}

\subsection{Decremental State-Space Relaxation and Column Elimination} \label{A_subproblemRelax}

The \emph{ng}-relaxation is a state-space relaxation. Figure~\ref{f_subproblemRelax} illustrates its strengthening in our illustrative example (Figure~\ref{f_example}). Decremental state-space relaxation forbids cycles of the form $1 \rightarrow 2 \rightarrow 1$ globally by adding customer 1 to the \emph{ng}-set of customer 2, yielding the DAG in Figure~\ref{f_subproblemRelax}(a) with three additional arcs and two additional nodes. As Figure~\ref{f_subproblemRelax}(c) illustrates, column elimination instead locally removes the conflicting path $\underline{0}\rightarrow1\rightarrow2\rightarrow1\rightarrow\overline{0}$ by adding only two arcs and one node.

\begin{figure}[htbp]
  \centering

   \begin{subfigure}[b]{0.48\textwidth}
    \centering
     \begin{tikzpicture}[
         scale=0.45,
        every node/.style={transform shape},
        >=Stealth,
        thick,
        node distance=2.2cm and 3.0cm,
        on grid,
        auto,
        myNode/.style={
          circle, draw,
          fill=white,
          minimum size=22mm,
          inner sep=1pt,
          font=\footnotesize,
          align=center        
        },
        every edge/.style={->},
        every label/.append style={midway,font=\tiny},
      ]
      \pgfdeclarelayer{background}
      \pgfsetlayers{background,main}

      \node (layer0) at (-8,   0) {$\mathscr{d}=0$};
      \node (layer1) at (-8, -3.5) {$\mathscr{d}=1$};
      \node (layer2) at (-8, -7.0) {$\mathscr{d}=2$};
      \node (layer3) at (-8,-10.5) {$\mathscr{d}=3$};

      \node[myNode] (R) at (0,0) {\normalsize $r$};

      \node[myNode] (A) at (-6,-3.5) {
        $\mathscr{v}=1$\\$\mathscr{d}=1,\mathscr{t}=5$\\$\mathscr{U}=\{1\}$
      };
      \node[myNode] (B) at (0,-3.5) {
        $\mathscr{v}=2$\\$\mathscr{d}=1,\mathscr{t}=25$ \\$\mathscr{U}=\{2\}$
      };
      \node[myNode] (C) at (6,-3.5) {
        $\mathscr{v}=3$\\$\mathscr{d}=1,\mathscr{t}=5$ \\$\mathscr{U}=\{3\}$
      };

      \node[myNode, fill=red!20] (M) at (-3,-7.0) {
        $\mathscr{v}=2$\\$\mathscr{d}=2,\mathscr{t}=25$ \\ $\mathscr{U}=\{1,2\}$
      };

      \node[myNode] (Mnew) at (0,-7.0) {
        $\mathscr{v}=2$\\$\mathscr{d}=2,\mathscr{t}=25$ \\ $\mathscr{U}=\{2\}$
      };

      \node[myNode] (F) at (-6,-7.0) {
        $\mathscr{v}=1$\\$\mathscr{d}=2,\mathscr{t}=35$ \\$\mathscr{U}=\{1\}$
      };
      \node[myNode] (G) at (6,-7.0) {
        $\mathscr{v}=3$ \\$\mathscr{d}=2,\mathscr{t}=35$ \\ $\mathscr{U}=\{3\}$
      };

      \node[myNode] (I) at (-6,-10.5) {
        $\mathscr{v}=1$ \\$\mathscr{d}=3,\mathscr{t}=35$ \\ $\mathscr{U}=\{1\}$
      };

     \node[myNode] (I2) at (0,-10.5) {
        $\mathscr{v}=2$ \\$\mathscr{d}=3,\mathscr{t}=45$ \\ $\mathscr{U}=\{2\}$
      };

        \node[myNode, fill=red!20] (I2new) at (-3,-10.5) {
        $\mathscr{v}=2$ \\$\mathscr{d}=3,\mathscr{t}=45$ \\ $\mathscr{U}=\{1,2\}$
      };
      
      \node[myNode] (J) at (6,-10.5) {
        $\mathscr{v}=3$\\$\mathscr{d}=3,\mathscr{t}=35$ \\ $\mathscr{U}=\{3\}$
      };

      \node[myNode] (T) at (0,-14) {\normalsize $t$};

      \draw[line width=0.5pt] (R) -- node[pos=0.1, above] {$y_{1}$} (A);
      \draw[line width=0.5pt] (R) -- node[pos=0.2, right] {$y_{2}$} (B);
      \draw[line width=0.5pt] (R) -- node[pos=0.1, above] {$y_{3}$} (C);
      \draw[line width=0.5pt] (A) -- node[pos=0.2, above] {$y_{4}$} (M);
      \draw[line width=0.5pt] (B) -- node[pos=0.1, above] {$y_{5}$} (F);
      \draw[line width=0.5pt] (B) -- node[pos=0.1, above] {$y_{6}$} (G);
      \draw[line width=0.5pt] (C) -- node[pos=0.1, above] {$y_{7}$} (Mnew);
      \draw[line width=0.5pt] (F) -- node[pos=0.2, right] {$y_{8}$} (I2new);
      \draw[line width=0.5pt] (Mnew) -- node[pos=0.25, above] {$y_{9}$} (I.east);
      \draw[line width=0.5pt] (Mnew) -- node[pos=0.1, above] {$y_{10}$} (J);
      \draw[red, line width=0.5pt] (M) -- node[pos=0.08, above] {$y_{22}$} (J.west);
      \draw[line width=0.5pt] (G) -- node[pos=0.1, above] {$y_{11}$} (I2);

      \begin{pgfonlayer}{background}
        \draw[dashed, line width=0.1pt, bend right=30] (R) to node[pos=0.1, left] {$y_{12}$} (T);
        \draw[dashed, line width=0.1pt] (A) -- node[pos=0.1, left] {$y_{13}$} (T);
        \draw[dashed, line width=0.1pt, bend left=30] (B) to node[pos=0.1, right] {$y_{14}$} (T);
        \draw[dashed, line width=0.1pt] (C) -- node[pos=0.1, right] {$y_{15}$} (T);
        \draw[dashed, line width=0.1pt, bend right=30] (Mnew) to node[pos=0.2, left] {$y_{17}$} (T);
        \draw[dashed, line width=0.1pt, bend right=20] (F) to node[pos=0.1, left] {$y_{16}$} (T);
        \draw[dashed, line width=0.1pt] (G) -- node[pos=0.1, right] {$y_{18}$} (T);
        \draw[dashed, line width=0.1pt] (I) -- node[pos=0.2, left] {$y_{19}$} (T);
        \draw[dashed, line width=0.1pt] (I2) -- node[pos=0.2, right] {$y_{20}$} (T);
        \draw[dashed, line width=0.1pt] (J) -- node[pos=0.2, right] {$y_{21}$} (T);

        \draw[dashed, red, line width=0.1pt] (M) to node[pos=0.1, left] {$y_{23}$} (T);
        \draw[dashed, red, line width=0.1pt] (I2new) to node[pos=0.25, left] {$y_{24}$} (T);
      \end{pgfonlayer}

    \end{tikzpicture}

    \caption{\small Decremental state-space relaxation.}
  \end{subfigure}
  \hfill%
  \begin{subfigure}[b]{0.48\textwidth}
    \centering
   \begin{tikzpicture}[
        scale=0.45,
        every node/.style={transform shape},
        >=Stealth,
        thick,
        node distance=2.2cm and 3.0cm,
        on grid,
        auto,
        myNode/.style={
          circle, draw,
          fill=white,
          minimum size=22mm,
          inner sep=1pt,
          font=\footnotesize,
          align=center        
        },
        every edge/.style={->},
        every label/.append style={midway,font=\tiny},
      ]
      \pgfdeclarelayer{background}
      \pgfsetlayers{background,main}

      \node (layer0) at (-8,   0) {$\mathscr{d}=0$};
      \node (layer1) at (-8, -3.5) {$\mathscr{d}=1$};
      \node (layer2) at (-8, -7.0) {$\mathscr{d}=2$};
      \node (layer3) at (-8,-10.5) {$\mathscr{d}=3$};

      \node[myNode] (R) at (0,0) {\normalsize $r$};

      \node[myNode] (A) at (-6,-3.5) {
        $\mathscr{v}=1$\\$\mathscr{d}=1,\mathscr{t}=5$\\$\mathscr{U}=\{1\}$
      };
      \node[myNode] (B) at (0,-3.5) {
        $\mathscr{v}=2$\\$\mathscr{d}=1,\mathscr{t}=25$ \\$\mathscr{U}=\{2\}$
      };
      \node[myNode] (C) at (6,-3.5) {
        $\mathscr{v}=3$\\$\mathscr{d}=1,\mathscr{t}=5$ \\$\mathscr{U}=\{3\}$
      };

      \node[myNode, fill=red!20] (M) at (-3,-7.0) {
        $\mathscr{v}=2$\\$\mathscr{d}=2,\mathscr{t}=25$ \\ $\mathscr{U}=\{1,2\}$
      };

      \node[myNode] (Mnew) at (0,-7.0) {
        $\mathscr{v}=2$\\$\mathscr{d}=2,\mathscr{t}=25$ \\ $\mathscr{U}=\{2\}$
      };

      \node[myNode] (F) at (-6,-7.0) {
        $\mathscr{v}=1$\\$\mathscr{d}=2,\mathscr{t}=35$ \\$\mathscr{U}=\{1\}$
      };
      \node[myNode] (G) at (6,-7.0) {
        $\mathscr{v}=3$ \\$\mathscr{d}=2,\mathscr{t}=35$ \\ $\mathscr{U}=\{3\}$
      };

      \node[myNode] (I) at (-6,-10.5) {
        $\mathscr{v}=1$ \\$\mathscr{d}=3,\mathscr{t}=35$ \\ $\mathscr{U}=\{1\}$
      };

     \node[myNode] (I2) at (0,-10.5) {
        $\mathscr{v}=2$ \\$\mathscr{d}=3,\mathscr{t}=45$ \\ $\mathscr{U}=\{2\}$
      };
      
      \node[myNode] (J) at (6,-10.5) {
        $\mathscr{v}=3$\\$\mathscr{d}=3,\mathscr{t}=35$ \\ $\mathscr{U}=\{3\}$
      };

      \node[myNode] (T) at (0,-14) {\normalsize $t$};

      \draw[line width=0.5pt] (R) -- node[pos=0.1, above] {$y_{1}$} (A);
      \draw[line width=0.5pt] (R) -- node[pos=0.2, right] {$y_{2}$} (B);
      \draw[line width=0.5pt] (R) -- node[pos=0.1, above] {$y_{3}$} (C);
      \draw[line width=0.5pt] (A) -- node[pos=0.2, above] {$y_{4}$} (M);
      \draw[line width=0.5pt] (B) -- node[pos=0.1, above] {$y_{5}$} (F);
      \draw[line width=0.5pt] (B) -- node[pos=0.1, above] {$y_{6}$} (G);
      \draw[line width=0.5pt] (C) -- node[pos=0.1, above] {$y_{7}$} (Mnew);
      \draw[line width=0.5pt] (F) -- node[pos=0.1, above] {$y_{8}$} (I2);
      \draw[line width=0.5pt] (Mnew) -- node[pos=0.25, above] {$y_{9}$} (I.east);
      \draw[line width=0.5pt] (Mnew) -- node[pos=0.1, above] {$y_{10}$} (J);
      \draw[red, line width=0.5pt] (M) -- node[pos=0.08, above] {$y_{22}$} (J.west);
      \draw[line width=0.5pt] (G) -- node[pos=0.1, above] {$y_{11}$} (I2);

      \begin{pgfonlayer}{background}
        \draw[dashed, line width=0.1pt, bend right=30] (R) to node[pos=0.1, left] {$y_{12}$} (T);
        \draw[dashed, line width=0.1pt] (A) -- node[pos=0.1, left] {$y_{13}$} (T);
        \draw[dashed, line width=0.1pt, bend left=30] (B) to node[pos=0.1, right] {$y_{14}$} (T);
        \draw[dashed, line width=0.1pt] (C) -- node[pos=0.1, right] {$y_{15}$} (T);
        \draw[dashed, line width=0.1pt, bend right=30] (Mnew) to node[pos=0.2, left] {$y_{17}$} (T);
        \draw[dashed, line width=0.1pt] (F) -- node[pos=0.1, left] {$y_{16}$} (T);
        \draw[dashed, line width=0.1pt] (G) -- node[pos=0.1, right] {$y_{18}$} (T);
        \draw[dashed, line width=0.1pt] (I) -- node[pos=0.2, left] {$y_{19}$} (T);
        \draw[dashed, line width=0.1pt] (I2) -- node[pos=0.2, right] {$y_{20}$} (T);
        \draw[dashed, line width=0.1pt] (J) -- node[pos=0.2, right] {$y_{21}$} (T);

        \draw[dashed, red, line width=0.1pt] (M) to node[pos=0.1, left] {$y_{23}$} (T);
      \end{pgfonlayer}

    \end{tikzpicture}

    \caption{\small Column elimination. }
  \end{subfigure}

  \caption{Strengthening subproblem relaxations.}
  \label{f_subproblemRelax}
\end{figure}
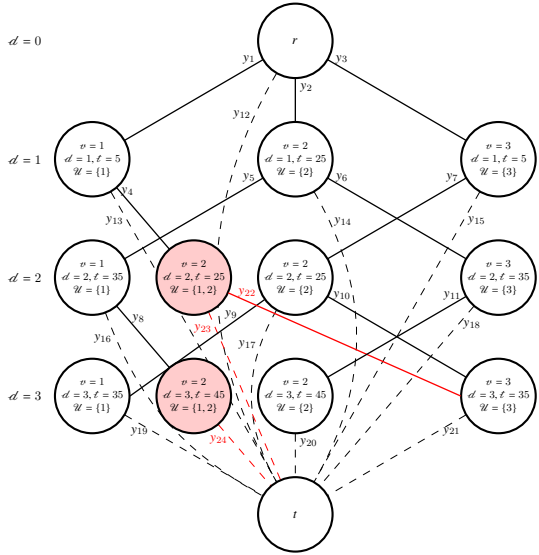
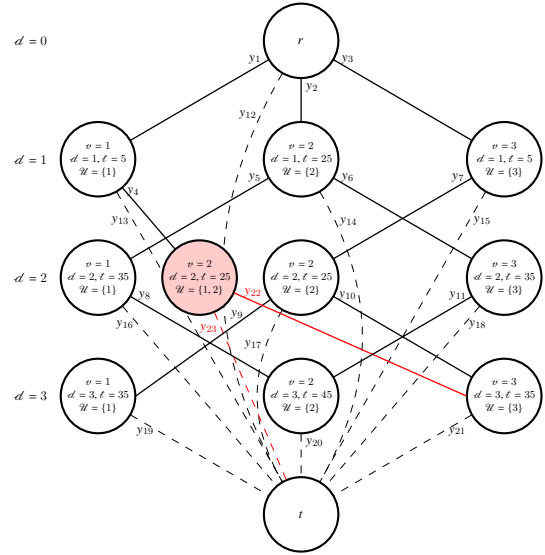

\section{Detailed Computational Results}

\subsection{Dantzig-Wolfe and Arc-Flow Performance (Without Subset-Row Cuts)} \label{A_results_noCuts}

Table~\ref{t_DWvsAF_combined} summarizes the performance of DW and AF for both $\Delta = 6$ and $\Delta = 12$, by instance group. For DW, Columns 2--6 present the lower bound, the time spent on the restricted master problem (RMP), the time spent on the pricing problem (PP), the total CPU time, and the number of instances solved within the time limit. Columns 7--12 present the corresponding AF metrics, with Column 11 additionally reporting the ratio of DAG paths to priced routes as a measure of path recombination. Finally, Columns 13--15 report DW-to-AF ratios for the number of variables, pricing iterations, and total CPU time.

\begin{table}[htbp]
  \centering
  \begin{adjustbox}{max width=\textwidth}
  \begin{tabular}{l
                  r r r r r
                  r r r r r r
                  r r r}
    \toprule
    & \multicolumn{5}{c}{Dantzig-Wolfe}
    & \multicolumn{6}{c}{Arc-Flow}
    & \multicolumn{3}{c}{DW-to-AF Ratios} \\
    \cmidrule(lr){2-6}
    \cmidrule(lr){7-12}
    \cmidrule(lr){13-15}
    Instances
      & LB     & RMP (s) & PP (s) & Time (s) & Solved
      & LB     & RMP (s) & PP (s) & Time (s) & Recomb. & Solved
      & Variables & Iterations & Time \\
    \midrule

    \multicolumn{15}{c}{\textbf{$\Delta = 6$}} \\
    \midrule
    R1-25  & 457.24 & 0.03 & 0.28 & 0.34 & 12/12
           & 457.24 & 0.08 & 0.31 & 0.46 & 1.02 & 12/12
           & 0.27 & 0.88 & 0.74 \\
    C1-25  & 190.58 & 0.06 & 0.79 & 0.91 &  9/9
           & 190.58 & 0.15 & 0.52 & 0.80 & 1.41 &  9/9
           & 0.30 & 1.46 & 1.13 \\
    RC1-25 & 340.78 & 0.06 & 0.58 & 0.70 &  8/8
           & 340.78 & 0.12 & 0.56 & 0.80 & 1.14 &  8/8
           & 0.27 & 1.13 & 0.88 \\
    R2-25  & 376.73 & 0.07 & 1.26 & 1.38 & 11/11
           & 376.73 & 0.28 & 1.20 & 1.64 & 1.12 & 11/11
           & 0.21 & 1.05 & 0.85 \\
    C2-25  & 214.30 & 0.15 & 2.20 & 2.48 &  8/8
           & 214.30 & 0.57 & 1.59 & 2.34 & 3.39 &  8/8
           & 0.25 & 1.62 & 1.06 \\
    RC2-25 & 309.91 & 0.12 & 3.96 & 4.29 &  8/8
           & 309.91 & 0.62 & 3.15 & 4.15 & 1.79 &  8/8
           & 0.25 & 1.29 & 1.03 \\
    \midrule
    \textbf{Geo.\ Mean}
           & \textbf{300.95} & \textbf{0.07} & \textbf{1.06} & \textbf{1.21} & \textbf{56/56}
           & \textbf{300.95} & \textbf{0.23} & \textbf{0.90} & \textbf{1.29} & \textbf{1.49} & \textbf{56/56}
           & \textbf{0.26} & \textbf{1.21} & \textbf{0.94} \\
    \midrule
    R1-50  & 747.96 & 0.14 & 3.30 & 3.61 & 12/12
           & 747.96 & 0.64 & 3.62 & 4.49 & 1.04 & 12/12
           & 0.22 & 0.94 & 0.80 \\
    C1-50  & 361.69 & 0.18 & 5.82 & 6.14 &  9/9
           & 361.69 & 0.73 & 4.26 & 5.24 & 1.53 &  9/9
           & 0.25 & 1.33 & 1.17 \\
    RC1-50 & 658.20 & 0.21 & 4.42 & 4.78 &  8/8
           & 658.20 & 0.78 & 4.48 & 5.61 & 1.20 &  8/8
           & 0.21 & 1.03 & 0.85 \\
    R2-50  & 599.64 & 0.61 & 69.12 & 70.51 & 11/11
           & 599.64 & 7.41 & 61.57 & 71.43 & 1.30 & 11/11
           & 0.16 & 1.02 & 0.99 \\
    C2-50  & 357.48 & 7.75 & 162.85 & 175.31 &  8/8
           & 357.48 & 17.02 & 33.12 & 52.73 & 19.64 &  8/8
           & 0.39 & 4.67 & 3.32 \\
    RC2-50 & 571.84 & 0.83 & 97.60 & 99.51 &  7/8
           & 571.84 & 6.82 & 64.01 & 73.55 & 2.84 &  8/8
           & 0.23 & 1.49 & 1.35 \\
    \midrule
    \textbf{Geo.\ Mean}
           & \textbf{528.65} & \textbf{0.52} & \textbf{21.30} & \textbf{22.52} & \textbf{55/56}
           & \textbf{528.65} & \textbf{2.61} & \textbf{14.43} & \textbf{18.22} & \textbf{2.27} & \textbf{56/56}
           & \textbf{0.23} & \textbf{1.45} & \textbf{1.23} \\

    \midrule
    \multicolumn{15}{c}{\textbf{$\Delta = 12$}} \\
    \midrule
    R1-25  & 457.24 & 0.03 & 0.27 & 0.34 & 12/12
           & 457.24 & 0.07 & 0.32 & 0.46 & 1.01 & 12/12
           & 0.26 & 0.85 & 0.75 \\
    C1-25  & 190.58 & 0.06 & 0.81 & 0.93 &  9/9
           & 190.58 & 0.13 & 0.65 & 0.91 & 1.10 &  9/9
           & 0.26 & 1.29 & 1.02 \\
    RC1-25 & 340.86 & 0.04 & 0.52 & 0.59 &  8/8
           & 340.86 & 0.08 & 0.51 & 0.72 & 1.10 &  8/8
           & 0.25 & 1.05 & 0.83 \\
    R2-25  & 377.79 & 0.06 & 1.52 & 1.65 & 11/11
           & 377.80 & 0.25 & 1.42 & 1.83 & 1.08 & 11/11
           & 0.21 & 0.99 & 0.90 \\
    C2-25  & 214.30 & 0.16 & 2.22 & 2.47 &  8/8
           & 214.30 & 0.82 & 1.93 & 3.03 & 1.89 &  8/8
           & 0.20 & 1.32 & 0.82 \\
    RC2-25 & 318.11 & 0.06 & 2.40 & 2.59 &  8/8
           & 318.11 & 0.39 & 2.79 & 3.44 & 1.34 &  8/8
           & 0.22 & 0.98 & 0.75 \\
    \midrule
    \textbf{Geo.\ Mean}
           & \textbf{302.42} & \textbf{0.06} & \textbf{0.99} & \textbf{1.12} & \textbf{56/56}
           & \textbf{302.42} & \textbf{0.20} & \textbf{0.97} & \textbf{1.34} & \textbf{1.22} & \textbf{56/56}
           & \textbf{0.23} & \textbf{1.07} & \textbf{0.84} \\
    \midrule
    R1-50  & 748.50 & 0.17 & 3.46 & 3.78 & 12/12
           & 748.50 & 0.70 & 3.69 & 4.74 & 1.02 & 12/12
           & 0.22 & 0.93 & 0.80 \\
    C1-50  & 361.69 & 0.22 & 6.27 & 6.74 &  9/9
           & 361.69 & 1.03 & 5.47 & 6.90 & 1.10 &  9/9
           & 0.20 & 1.09 & 0.98 \\
    RC1-50 & 659.34 & 0.19 & 3.69 & 4.02 &  8/8
           & 659.34 & 0.61 & 4.37 & 5.23 & 1.10 &  8/8
           & 0.20 & 0.98 & 0.77 \\
    R2-50  & 617.60 & 0.45 & 48.06 & 49.00 & 10/11
           & 617.61 & 6.24 & 51.40 & 60.26 & 1.11 & 11/11
           & 0.16 & 0.98 & 0.81 \\
    C2-50  & 357.48 & 7.17 & 146.79 & 158.18 &  8/8
           & 357.48 & 32.37 & 42.19 & 81.62 & 5.78 &  8/8
           & 0.27 & 3.47 & 1.94 \\
    RC2-50 & 607.18 & 0.32 & 22.49 & 23.23 &  6/8
           & 607.18 & 2.76 & 27.72 & 31.45 & 1.66 &  6/8
           & 0.19 & 0.96 & 0.74 \\
    \midrule
    \textbf{Geo.\ Mean}
           & \textbf{536.81} & \textbf{0.44} & \textbf{15.27} & \textbf{16.25} & \textbf{53/56}
           & \textbf{536.81} & \textbf{2.50} & \textbf{13.21} & \textbf{17.26} & \textbf{1.54} & \textbf{54/56}
           & \textbf{0.20} & \textbf{1.22} & \textbf{0.94} \\
    \bottomrule
  \end{tabular}
  \end{adjustbox}
  \caption{Summary of DW and AF performance for both $\Delta=6$ and $\Delta=12$.}
  \label{t_DWvsAF_combined}
\end{table}

\subsection{Dantzig-Wolfe and Arc-Flow Performance With Subset-Row Cuts} \label{A_results_cuts}
Table~\ref{t_DWvsAF_cuts_combined} reports the same information as Table~\ref{t_DWvsAF_combined} when subset-row cuts (SRCs) are introduced. Columns 3 and 9  additionally report the number of cuts added to DW and AF, respectively. We report the arithmetic mean in these columns since some values are zero.

\begin{table}[htbp]
    \centering
    \begin{adjustbox}{max width=\textwidth}
    \begin{tabular}{l
                  r r r r r r
                  r r r r r r r
                  r r r}
    \toprule
    & \multicolumn{6}{c}{Dantzig-Wolfe}
    & \multicolumn{7}{c}{Arc-Flow}
    & \multicolumn{3}{c}{DW-to-AF Ratios} \\
    \cmidrule(lr){2-7}
    \cmidrule(lr){8-14}
    \cmidrule(lr){15-17}
    \multicolumn{1}{c}{Instances}
      &  \multicolumn{1}{c}{LB} &  \multicolumn{1}{c}{Cuts} &  \multicolumn{1}{c}{RMP (s)} &  \multicolumn{1}{c}{PP (s)} &  \multicolumn{1}{c}{Time (s)} &  \multicolumn{1}{c}{Solved}
      &  \multicolumn{1}{c}{LB} &  \multicolumn{1}{c}{Cuts} &  \multicolumn{1}{c}{RMP (s)} &  \multicolumn{1}{c}{PP (s)} &  \multicolumn{1}{c}{Time (s)} &  \multicolumn{1}{c}{Recomb.} &  \multicolumn{1}{c}{Solved}
      &  \multicolumn{1}{c}{Variables} &  \multicolumn{1}{c}{Iterations} &  \multicolumn{1}{c}{Time} \\
    \midrule

    \multicolumn{17}{c}{\textbf{$\Delta = 6$}} \\
    \midrule
    R1-25  & 459.22 & 12.3 & 0.04 & 0.38  & 0.47  & 12/12
           & 459.22 & 12.3 & 0.15 & 0.43  & 0.69  & 1.02 & 12/12
           & 0.26 & 0.86 & 0.68 \\
    C1-25  & 190.58 & 0.0  & 0.05 & 0.81  & 0.91  &  9/9
           & 190.58 & 0.0  & 0.14 & 0.60  & 0.82  & 1.41 &  9/9
           & 0.30 & 1.46 & 1.11 \\
    RC1-25 & 346.26 & 9.1  & 0.06 & 0.78  & 0.92  &  8/8
           & 346.26 & 9.3  & 0.21 & 0.65  & 1.03  & 1.14 &  8/8
           & 0.27 & 1.18 & 0.89 \\
    R2-25  & 380.62 & 24.8 & 0.10 & 3.90  & 4.21  & 11/11
           & 380.62 & 28.1 & 1.21 & 6.28  & 8.05  & 1.11 & 11/11
           & 0.20 & 0.97 & 0.52 \\
    C2-25  & 214.45 & 1.3  & 0.15 & 2.59  & 2.86  &  8/8
           & 214.45 & 1.3  & 0.62 & 1.88  & 2.75  & 3.00 &  8/8
           & 0.25 & 1.64 & 1.04 \\
    RC2-25 & 330.37 & 13.0 & 0.14 & 3.20  & 3.47  &  6/8
           & 330.37 & 12.5 & 0.58 & 2.77  & 3.62  & 1.54 &  6/8
           & 0.30 & 1.43 & 0.96 \\
    \midrule
    \textbf{Geo.\ Mean}
           & \textbf{305.76} & \textbf{10.1} & \textbf{0.08} & \textbf{1.41} & \textbf{1.59} & \textbf{54/56}
           & \textbf{305.76} & \textbf{10.6} & \textbf{0.35} & \textbf{1.33} & \textbf{1.90} & \textbf{1.43} & \textbf{54/56}
           & \textbf{0.26} & \textbf{1.22} & \textbf{0.84} \\
    \midrule
    R1-50  & 754.36 & 48.5 & 0.37 & 7.30  & 8.13  & 12/12
           & 754.34 & 48.9 & 2.81 & 8.67  & 12.19 & 1.02 & 12/12
           & 0.22 & 0.87 & 0.67 \\
    C1-50  & 361.69 & 0.0  & 0.23 & 5.73  & 6.18  &  9/9
           & 361.69 & 0.0  & 0.66 & 4.32  & 5.25  & 1.53 &  9/9
           & 0.25 & 1.33 & 1.18 \\
    RC1-50 & 712.11 & 84.5 & 0.66 & 28.72 & 30.43 &  8/8
           & 712.41 & 78.5 & 6.37 & 31.28 & 41.06 & 1.16 &  8/8
           & 0.20 & 0.91 & 0.74 \\
    R2-50  & 645.32 & 62.5 & 0.94 & 121.85 & 124.01 &  8/11
           & 645.32 & 65.0 & 33.90 & 223.42 & 284.78 & 1.08 &  8/11
           & 0.14 & 0.88 & 0.44 \\
    C2-50  & 357.48 & 0.0  & 7.80 & 163.05 & 175.59 &  8/8
           & 357.48 & 0.0  & 16.95 & 33.13 & 52.67 & 19.64 &  8/8
           & 0.39 & 4.67 & 3.33 \\
    RC2-50 & 617.39 & 34.4 & 0.90 & 80.04 & 81.70 &  5/8
           & 617.39 & 35.8 & 15.15 & 52.17 & 68.56 & 2.45 &  5/8
           & 0.24 & 1.34 & 1.19 \\
    \midrule
    \textbf{Geo.\ Mean}
           & \textbf{549.97} & \textbf{38.3} & \textbf{0.85} & \textbf{35.23} & \textbf{37.36} & \textbf{50/56}
           & \textbf{550.01} & \textbf{38.0} & \textbf{6.84} & \textbf{27.71} & \textbf{37.32} & \textbf{2.13} & \textbf{50/56}
           & \textbf{0.23} & \textbf{1.34} & \textbf{1.00} \\

    \midrule
    \multicolumn{17}{c}{\textbf{$\Delta = 12$}} \\
    \midrule
    R1-25  & 459.22 & 12.3 & 0.05 & 0.41  & 0.49  & 12/12
           & 459.22 & 13.4 & 0.12 & 0.41  & 0.63  & 1.01 & 12/12
           & 0.25 & 0.87 & 0.78 \\
    C1-25  & 190.58 & 0.0  & 0.05 & 0.83  & 0.93  &  9/9
           & 190.58 & 0.0  & 0.15 & 0.64  & 0.90  & 1.10 &  9/9
           & 0.26 & 1.29 & 1.02 \\
    RC1-25 & 346.26 & 7.9  & 0.05 & 0.61  & 0.71  &  8/8
           & 346.26 & 6.9  & 0.13 & 0.59  & 0.83  & 1.10 &  8/8
           & 0.25 & 1.05 & 0.85 \\
    R2-25  & 380.62 & 24.4 & 0.14 & 4.69  & 4.94  & 11/11
           & 380.62 & 25.1 & 1.00 & 6.52  & 8.25  & 1.08 & 11/11
           & 0.18 & 0.93 & 0.60 \\
    C2-25  & 214.45 & 1.3  & 0.17 & 2.68  & 3.06  &  8/8
           & 214.45 & 1.3  & 0.79 & 2.37  & 3.45  & 1.76 &  8/8
           & 0.20 & 1.35 & 0.89 \\
    RC2-25 & 318.11 & 0.0  & 0.06 & 2.47  & 2.67  &  8/8
           & 318.11 & 0.0  & 0.36 & 2.75  & 3.37  & 1.34 &  8/8
           & 0.22 & 0.98 & 0.79 \\
    \midrule
    \textbf{Geo.\ Mean}
           & \textbf{303.84} & \textbf{7.3} & \textbf{0.08} & \textbf{1.36} & \textbf{1.53} & \textbf{56/56}
           & \textbf{303.84} & \textbf{7.8} & \textbf{0.30} & \textbf{1.37} & \textbf{1.89} & \textbf{1.21} & \textbf{56/56}
           & \textbf{0.22} & \textbf{1.06} & \textbf{0.81} \\
    \midrule
    R1-50  & 754.30 & 48.3 & 0.35 & 7.84  & 8.59  & 12/12
           & 754.33 & 46.9 & 2.62 & 9.20  & 12.59 & 1.01 & 12/12
           & 0.21 & 0.89 & 0.68 \\
    C1-50  & 361.69 & 0.0  & 0.20 & 6.27  & 6.69  &  9/9
           & 361.69 & 0.0  & 1.04 & 5.40  & 6.86  & 1.10 &  9/9
           & 0.20 & 1.09 & 0.98 \\
    RC1-50 & 712.45 & 81.9 & 0.67 & 18.59 & 19.90 &  8/8
           & 712.55 & 76.8 & 4.56 & 18.49 & 25.59 & 1.10 &  8/8
           & 0.21 & 0.96 & 0.78 \\
    R2-50  & 645.75 & 55.0 & 0.83 & 106.96 & 109.10 &  8/11
           & 645.77 & 55.0 & 21.50 & 152.97 & 191.53 & 1.07 &  8/11
           & 0.15 & 0.92 & 0.57 \\
    C2-50  & 357.48 & 0.0  & 7.00 & 144.65 & 155.66 &  8/8
           & 357.48 & 0.0  & 32.45 & 42.51 & 81.79 & 5.78 &  8/8
           & 0.27 & 3.47 & 1.90 \\
    RC2-50 & 607.18 & 0.0  & 0.32 & 22.87 & 23.62 &  6/8
           & 607.18 & 0.0  & 2.91 & 27.69 & 31.74 & 1.66 &  6/8
           & 0.19 & 0.96 & 0.74 \\
    \midrule
    \textbf{Geo.\ Mean}
           & \textbf{548.54} & \textbf{30.9} & \textbf{0.67} & \textbf{26.20} & \textbf{27.77} & \textbf{51/56}
           & \textbf{548.56} & \textbf{29.8} & \textbf{5.42} & \textbf{23.43} & \textbf{32.12} & \textbf{1.52} & \textbf{51/56}
           & \textbf{0.20} & \textbf{1.19} & \textbf{0.86} \\
    \bottomrule
    \end{tabular}
    \end{adjustbox}
    \caption{Summary of DW and AF performance with SRCs for both $\Delta=6$ and $\Delta=12$.}
    \label{t_DWvsAF_cuts_combined}
\end{table}

\subsection{Dantzig-Wolfe and Arc-Flow Performance Under Explicit and Implicit DAG Representations} \label{A_explicit_vs_implicit}

We compare DW and AF under explicit and implicit DAG representations. In the explicit approach, the unfolded \emph{ng}-relaxation state space is stored in memory, and the exact pricing problem is solved using a standard shortest path algorithm. We implement a \emph{beam search}-based heuristic pricing procedure that explores only the three best extensions from each node. In contrast, the implicit approach relies on label extensions and dominance checks (see \S\ref{ss_VRPTW_DW}). For this experiment, we consider $\Delta = 0$ and only type 1 instances (i.e., with tight time windows).

Table \ref{t_explicit_implicit} reports the CPU times (in seconds) of DW and AF under the two approaches. For DW, Columns 2–4 report the CPU times with the explicit DAG, the implicit DAG, and their ratio for the 25-customer instances. Columns 5–7 provide the corresponding results for AF. Finally, Columns 9–14 present the same information as Columns 2–7 for the 50-customer instances. For DW, the implicit approach is about three 3.7 and 3.0 times faster than the explicit approach on the 25- and 50-customer instances, respectively. For AF, these factors are 3.6 and 3.8, indicating that the implicit approach is generally much faster for both reformulations, although the explicit representation remains competitive in some cases. These results provide context for the discussion in \S\ref{s_computationalImplications} and \S\ref{s_computational_subproblem}, but they are not generalizable, as performance highly depends on implementation. For example, \cite{karahalios2025column} found the explicit DAG approach to be competitive with state-of-the-art implicit methods when the \emph{ng}-route relaxation was combined with time bucketing, relaxed capacity constraints, and a specialized minimum-update successive shortest path algorithm, among other enhancements.

\renewcommand{\arraystretch}{0.75}
\begin{table}[htbp]
\centering
\begin{adjustbox}{max width=\textwidth}
\begin{tabular}{l rrr rrr | l rrr rrr}
\toprule
& \multicolumn{3}{c}{Dantzig-Wolfe} & \multicolumn{3}{c}{Arc-Flow} &
& \multicolumn{3}{c}{Dantzig-Wolfe} & \multicolumn{3}{c}{Arc-Flow} \\
\cmidrule(lr){2-4}\cmidrule(lr){5-7}\cmidrule(lr){9-11}\cmidrule(lr){12-14}
Instance & Explicit (s) & Implicit (s) & Ratio & Explicit (s) & Implicit (s) & Ratio &
Instance & Explicit (s) & Implicit (s) & Ratio & Explicit (s) & Implicit (s) & Ratio \\
\midrule
R101-25 & 0.07 & 0.07 & 1.00 & 0.11 & 0.14 & 0.79 & R101-50 & 0.17 & 0.31 & 0.55 & 0.29 & 0.44 & 0.66 \\
R102-25 & 0.88 & 0.17 & 5.18 & 0.99 & 0.28 & 3.54 & R102-50 & 26.55 & 1.23 & 21.59 & 30.43 & 1.48 & 20.56 \\
R103-25 & 3.38 & 0.42 & 8.05 & 3.29 & 0.68 & 4.84 & R103-50 & 152.80 & 5.96 & 25.64 & 118.98 & 7.25 & 16.41 \\
R104-25 & 12.85 & 0.65 & 19.77 & 15.11 & 0.91 & 16.60 & R104-50 & 474.93 & 43.26 & 10.98 & 595.52 & 43.94 & 13.55 \\
R105-25 & 0.11 & 0.14 & 0.79 & 0.18 & 0.22 & 0.82 & R105-50 & 0.64 & 0.59 & 1.09 & 0.89 & 0.94 & 0.95 \\
R106-25 & 1.79 & 0.29 & 6.17 & 2.88 & 0.45 & 6.40 & R106-50 & 43.05 & 3.06 & 14.07 & 50.79 & 3.86 & 13.16 \\
R107-25 & 7.04 & 0.54 & 13.04 & 6.30 & 0.80 & 7.88 & R107-50 & 197.75 & 10.51 & 18.82 & 185.51 & 12.19 & 15.22 \\
R108-25 & 20.13 & 0.86 & 23.41 & 19.68 & 1.21 & 16.26 & R108-50 & 474.09 & 56.35 & 8.41 & 590.31 & 57.40 & 10.28 \\
R109-25 & 0.42 & 0.29 & 1.45 & 0.57 & 0.51 & 1.12 & R109-50 & 8.00 & 2.65 & 3.02 & 7.46 & 3.68 & 2.03 \\
R110-25 & 1.76 & 0.53 & 3.32 & 1.64 & 0.79 & 2.08 & R110-50 & 33.56 & 7.72 & 4.35 & 37.79 & 8.69 & 4.35 \\
R111-25 & 4.61 & 0.57 & 8.09 & 4.25 & 0.72 & 5.90 & R111-50 & 105.99 & 12.43 & 8.53 & 102.63 & 13.42 & 7.65 \\
R112-25 & 3.93 & 1.00 & 3.93 & 4.98 & 1.45 & 3.43 & R112-50 & 98.36 & 30.57 & 3.22 & 99.97 & 31.73 & 3.15 \\
C101-25 & 0.34 & 0.29 & 1.17 & 0.25 & 0.33 & 0.76 & C101-50 & 2.45 & 1.39 & 1.76 & 2.11 & 1.05 & 2.01 \\
C102-25 & 21.08 & 1.21 & 17.42 & 20.15 & 0.94 & 21.44 & C102-50 & 100.15 & 11.71 & 8.55 & 106.14 & 4.57 & 23.23 \\
C103-25 & 78.23 & 4.10 & 19.08 & 79.60 & 1.59 & 50.06 & C103-50 & 564.86 & 63.87 & 8.84 & 408.89 & 19.97 & 20.48 \\
C104-25 & 129.05 & 5.62 & 22.96 & 101.45 & 2.57 & 39.48 & C104-50 & 962.01 & 169.35 & 5.68 & 1,205.83 & 91.76 & 13.14 \\
C105-25 & 0.84 & 0.34 & 2.47 & 0.74 & 0.35 & 2.11 & C105-50 & 11.48 & 1.69 & 6.79 & 7.48 & 1.34 & 5.58 \\
C106-25 & 0.38 & 0.28 & 1.36 & 0.34 & 0.33 & 1.03 & C106-50 & 6.04 & 1.45 & 4.17 & 6.20 & 1.07 & 5.79 \\
C107-25 & 3.17 & 0.50 & 6.34 & 2.29 & 0.40 & 5.73 & C107-50 & 18.50 & 2.11 & 8.77 & 22.67 & 1.50 & 15.11 \\
C108-25 & 6.34 & 1.06 & 5.98 & 9.27 & 0.97 & 9.56 & C108-50 & 45.03 & 13.09 & 3.44 & 78.93 & 5.92 & 13.33 \\
C109-25 & 16.51 & 1.48 & 11.16 & 20.21 & 1.40 & 14.44 & C109-50 & 127.70 & 25.23 & 5.06 & 108.40 & 13.95 & 7.77 \\
RC101-25 & 0.14 & 0.26 & 0.54 & 0.20 & 0.33 & 0.61 & RC101-50 & 0.39 & 0.84 & 0.46 & 0.56 & 1.00 & 0.56 \\
RC102-25 & 1.08 & 0.70 & 1.54 & 1.21 & 0.72 & 1.68 & RC102-50 & 3.62 & 3.76 & 0.96 & 4.15 & 3.19 & 1.30 \\
RC103-25 & 2.77 & 1.43 & 1.94 & 2.68 & 1.25 & 2.14 & RC103-50 & 10.16 & 14.62 & 0.70 & 12.95 & 10.59 & 1.22 \\
RC104-25 & 6.57 & 2.09 & 3.14 & 5.66 & 1.52 & 3.72 & RC104-50 & 34.61 & 71.94 & 0.48 & 34.34 & 55.91 & 0.61 \\
RC105-25 & 0.59 & 0.51 & 1.16 & 0.60 & 0.50 & 1.20 & RC105-50 & 3.07 & 2.56 & 1.20 & 3.50 & 2.62 & 1.34 \\
RC106-25 & 0.63 & 0.67 & 0.94 & 0.95 & 0.94 & 1.01 & RC106-50 & 2.48 & 4.22 & 0.59 & 3.43 & 4.49 & 0.76 \\
RC107-25 & 2.56 & 1.64 & 1.56 & 2.81 & 1.70 & 1.65 & RC107-50 & 8.18 & 17.77 & 0.46 & 9.17 & 14.54 & 0.63 \\
RC108-25 & 4.16 & 3.37 & 1.23 & 4.90 & 3.24 & 1.51 & RC108-50 & 16.56 & 46.25 & 0.36 & 15.67 & 47.54 & 0.33 \\
\midrule
\textbf{Geo. Mean} & \textbf{2.38} & \textbf{0.65} & \textbf{3.68} & \textbf{2.58} & \textbf{0.73} & \textbf{3.56} &
\textbf{Geo. Mean} & \textbf{22.26} & \textbf{7.20} & \textbf{3.09} & \textbf{24.33} & \textbf{6.33} & \textbf{3.85} \\
\bottomrule
\end{tabular}
\end{adjustbox}
\caption{CPU time (in seconds) of DW and AF with an explicit and implicit DAG representation ($\Delta = 0$).}
\label{t_explicit_implicit}
\end{table}
\renewcommand{\arraystretch}{1.0}

\subsection{Dantzig-Wolfe and Arc-Flow Performance with Iterative Subproblem Strengthening} \label{A_iterativeSubproblem}

Table \ref{t_DSSR_CE_byInstanceGroup} reports the same information as Table \ref{t_DSSR_CE}, but aggregated by instance group and with the number of solved instances.

\begin{table}[htb]
  \centering
  \begin{adjustbox}{max width=\textwidth}
  \begin{tabular}{l
                  r r r c
                  r r r r c
                  r r r c
                  r r r r c}
    \toprule
    \multicolumn{1}{l}{} &
    \multicolumn{4}{c}{DW\text{-}DSSR} &
    \multicolumn{5}{c}{AF\text{-}DSSR} &
    \multicolumn{4}{c}{DW\text{-}CE} &
    \multicolumn{5}{c}{AF\text{-}CE} \\
    \cmidrule(lr){2-5}
    \cmidrule(lr){6-10}
    \cmidrule(lr){11-14}
    \cmidrule(lr){15-19}
    \multicolumn{1}{c}{Instances} &
    \multicolumn{1}{c}{\shortstack{Strength.\\Iter.}} & \multicolumn{1}{c}{\shortstack{Elim.\\Routes}} & \multicolumn{1}{c}{Time (s)} & \multicolumn{1}{c}{Solved} &
    \multicolumn{1}{c}{\shortstack{Strength.\\Iter.}} & \multicolumn{1}{c}{\shortstack{Elim.\\Arcs}} & \multicolumn{1}{c}{Time (s)} & \multicolumn{1}{c}{\shortstack{Change\\in DAG}} & \multicolumn{1}{c}{Solved} &
    \multicolumn{1}{c}{\shortstack{Strength.\\Iter.}} & \multicolumn{1}{c}{\shortstack{Elim.\\Routes}} & \multicolumn{1}{c}{Time (s)} & \multicolumn{1}{c}{Solved} &
    \multicolumn{1}{c}{\shortstack{Strength.\\Iter.}} & \multicolumn{1}{c}{\shortstack{Elim.\\Arcs}} & \multicolumn{1}{c}{Time (s)} & \multicolumn{1}{c}{\shortstack{Change\\in DAG}} & \multicolumn{1}{c}{Solved} \\
    \midrule
    R1-25      & 2.5 &   250.1 &  7.35 & 12/12 & 2.7 &   903.3 &  8.49 &     893,894.3 & 12/12 &  7.4 &   85.4 &  3.99 & 12/12 &  7.4 &   21.1 &  4.88 &   1,330.9 & 12/12 \\
    C1-25      & 1.0 &   398.5 & 11.78 &  9/9  & 1.0 & 1,576.3 & 10.58 &   4,864,992.4 &  9/9  &  4.8 &   81.3 &  8.34 &  9/9  &  4.7 &   10.9 &  7.42 &     377.0 &  9/9  \\
    RC1-25     & 2.8 &   533.9 &  8.23 &  8/8  & 2.8 & 1,669.3 &  9.03 &   1,002,700.6 &  8/8  & 11.2 &  340.8 &  4.81 &  8/8  & 10.8 &   52.0 &  4.97 &   1,786.3 &  8/8  \\
   \midrule
    \textbf{Geo.\ Mean} & \textbf{2.1} & \textbf{366.3} & \textbf{8.78} & \textbf{29/29} & \textbf{2.2} & \textbf{1,281.7} & \textbf{9.25} & \textbf{1,221,596.3} & \textbf{29/29} & \textbf{7.8} & \textbf{139.7} & \textbf{5.28} & \textbf{29/29} & \textbf{7.7} & \textbf{25.2} & \textbf{5.59} & \textbf{970.8} & \textbf{29/29} \\
    \midrule
    R1-50      & 4.0 &   346.8 & 28.06 &  6/12 & 2.7 &   790.0 & 27.77 &   3,959,035.3 &  6/12 &  8.3 &  130.2 & 13.54 &  8/12 &  8.1 &   33.3 & 17.37 &   3,846.8 &  8/12 \\
    C1-50      & 1.0 &   846.2 & 62.72 &  8/9  & 1.0 & 2,743.5 & 70.31 &  11,766,968.5 &  8/9  &  5.4 &  112.5 & 49.51 &  9/9  &  4.0 &   12.4 & 44.01 &     593.3 &  9/9  \\
    RC1-50     & 5.7 & 1,796.9 & 108.27 &  8/8  & 5.7 & 6,977.5 & 129.31 &   7,119,176.3 &  8/8  & 172.7 & 1,427.6 & 83.02 &  8/8  & 170.6 &  620.0 & 83.05 &  14,928.8 &  8/8  \\
        \midrule
    \textbf{Geo.\ Mean} & \textbf{3.4} & \textbf{986.6} & \textbf{61.42} & \textbf{22/29} & \textbf{3.1} & \textbf{3,205.1} & \textbf{68.11} & \textbf{4,944,384.5} & \textbf{22/29} & \textbf{34.1} & \textbf{415.7} & \textbf{41.96} & \textbf{25/29} & \textbf{31.1} & \textbf{112.2} & \textbf{43.03} & \textbf{4,878.8} & \textbf{25/29} \\
    \bottomrule
  \end{tabular}
  \end{adjustbox}
  \caption{Summary of DW and AF performance with DSSR and CE.}
  \label{t_DSSR_CE_byInstanceGroup}
\end{table}

\end{document}